\documentclass[leqno,final]{siamltex}
\usepackage{amssymb}
\usepackage{amsmath}
\usepackage[dvipsnames]{xcolor}
\usepackage[dvips]{graphicx}
\usepackage{subfigure}
\usepackage{psfrag}
\usepackage{cite}
\usepackage[dvips,bookmarks=true]{hyperref}

\newtheorem{remark}[theorem]{\sc Remark}

\newtheorem{example}[theorem]{\sc example}

\providecommand{\Div}{\operatorname{div}}          
\providecommand{\curl}{\operatorname{{\bf curl}}}  
\renewcommand{\grad}{\operatorname{{\bf grad}}}       

\providecommand*{\Dist}[2]{\operatorname{dist}({#1};{#2})}   
\providecommand*{\Dist}[2]{\Dist{#1}{#2}}
\providecommand*{\Span}[1]{\operatorname{Span}\left\{{#1}\right\}}     
\providecommand{\Supp}{\operatorname{supp}}                            
\providecommand{\supp}{\Supp}

\newcommand{\Vb}{{\mathbf{b}}}

\newcommand{\Vf}{{\mathbf{f}}}

\newcommand{\Vp}{{\mathbf{p}}}

\newcommand{\Vu}{{\mathbf{u}}}
\newcommand{\Vv}{{\mathbf{v}}}
\newcommand{\Vw}{{\mathbf{w}}}

\newcommand{\Ba}{{\boldsymbol{a}}}
\newcommand{\Bb}{{\boldsymbol{b}}}

\newcommand{\Bn}{{\boldsymbol{n}}}

\newcommand{\Bp}{{\boldsymbol{p}}}
\newcommand{\Bq}{{\boldsymbol{q}}}

\newcommand{\Bx}{{\boldsymbol{x}}}
\newcommand{\By}{{\boldsymbol{y}}}
\newcommand{\Bz}{{\boldsymbol{z}}}

\newcommand{\VH}{{\mathbf{H}}}

\newcommand{\VU}{{\mathbf{U}}}
\newcommand{\VV}{{\mathbf{V}}}

\newcommand{\BH}{{\boldsymbol{H}}}

\newcommand{\BPi}{\mathbf{\Pi}}
\newcommand{\BPsi}{\mathbf{\Psi}}

\newcommand{\nubf}{\boldsymbol{\nu}}
\newcommand{\xibf}{\boldsymbol{\xi}}

\newcommand{\Pibf}{\boldsymbol{\Pi}}
\newcommand{\Psibf}{\boldsymbol{\Psi}}

\newcommand{\Phibf}{\boldsymbol{\Phi}}

\newcommand{\Ce}{{\cal E}}

\newcommand{\Ci}{{\cal I}}

\newcommand{\Cl}{{\cal L}}
\newcommand{\Cm}{{\cal M}}
\newcommand{\Cn}{{\cal N}}

\newcommand{\Cp}{{\cal P}}

\newcommand{\Ct}{{\cal T}}

\newcommand{\Cv}{{\cal V}}

\newcommand{\bbN}{\mathbb{N}}

\newcommand{\bbP}{\mathbb{P}}

\newcommand{\bbR}{\mathbb{R}}

\newcommand{\FB}{\mathfrak{B}}

\newcommand*{\SP}[2]{\left<{#1},{#2}\right>} 

\providecommand*{\wt}[1]{\widetilde{#1}}
\providecommand*{\wh}[1]{\widehat{#1}}
\newcommand{\DS}{\displaystyle}

\newcommand*{\N}[1]{\left\|{#1}\right\|}     
\newcommand*{\SN}[1]{\left|{#1}\right|}      
\renewcommand*{\SP}[2]{\left({#1},{#2}\right)} 

\newcommand{\mesh}{\Cm}

\newcommand*{\blf}[1]{\mathsf{#1}} 

\newcommand*{\Lp}[2][\defaultdomain]{L^{#2}({#1})}

\newcommand*{\NLp}[3][\defaultdomain]{\N{#2}_{\Lp[#1]{#3}}}

\newcommand*{\Ltwo}[1][\defaultdomain]{\Lp[#1]{2}}

\newcommand*{\NLtwo}[2][\defaultdomain]{\NLp[#1]{#2}{2}}

\newcommand*{\Linf}[1][\defaultdomain]{L^{\infty}({#1})}

\newcommand*{\Hm}[2][\defaultdomain]{H^{#2}({#1})}

\newcommand*{\bHm}[3][\defaultdomain]{H_{#3}^{#2}({#1})}

\newcommand*{\NHm}[3][\defaultdomain]{{\N{#2}}_{\Hm[{#1}]{#3}}}

\newcommand*{\SNHm}[3][\defaultdomain]{{\SN{#2}}_{\Hm[{#1}]{#3}}}

\newcommand*{\Hone}[1][\defaultdomain]{\Hm[#1]{1}}

\newcommand*{\bHone}[2][\defaultdomain]{\bHm[#1]{1}{#2}}

\newcommand*{\NHone}[2][\defaultdomain]{{\N{#2}}_{\Hone[{#1}]}}

\newcommand*{\SNHone}[2][\defaultdomain]{{\SN{#2}}_{\Hone[{#1}]}}

\newcommand{\hlb}{\frac{1}{2}}

\newcommand*{\Hdiv}[1][\defaultdomain]{\boldsymbol{H}(\Div,{#1})}

\newcommand*{\kHdiv}[1][\defaultdomain]{\boldsymbol{H}(\Div0,{#1})}

\newcommand*{\Hcurl}[1][\defaultdomain]{\boldsymbol{H}(\curl,{#1})}
\newcommand*{\bHcurl}[2][\defaultdomain]{\boldsymbol{H}_{#2}(\curl,{#1})}

\newcommand*{\kHcurl}[1][\defaultdomain]{\boldsymbol{H}(\curl0,{#1})}
\newcommand*{\bkHcurl}[2][\defaultdomain]{\boldsymbol{H}_{#2}(\curl0,{#1})}

\newcommand*{\NHcurl}[2][\defaultdomain]{\N{#2}_{\Hcurl[#1]}}

\newcommand{\mcal}{\mathcal}
\newcommand{\ol}{\overline}

\newcommand{\script}[1]{\mathrm{\scriptsize #1}}

\newcommand{\rh}[1]{{#1}}
\newcommand{\rhrev}[1]{{#1}}

\newtheorem{assumption}{Assumption}[theorem]

\title{Local Multigrid in $\mathbf{H}(\mathbf{\curl})$}
\author{Ralf Hiptmair\thanks{SAM, ETH Z\"urich, CH-8092 Z\"urich,
    hiptmair\symbol{64}sam.math.ethz.ch} \and
   Weiying Zheng \thanks{LSEC,
Institute of Computational Mathematics, Academy of Mathematics and
System Sciences, Chinese Academy of Sciences, Beijing, 100080,
People's Republic of China. This author was supported in part by
China NSF under the grant 10401040 (zwy@lsec.cc.ac.cn).}}

\date{Research Report 2007-03, Seminar for Applied Mathematics, ETH Z\"urich}

\begin{document}
\DeclareGraphicsRule{.eps.bz2}{eps}{.eps.bb}{`bunzip2 -c #1}
\DeclareGraphicsRule{.ps.bz2}{eps}{.ps.bb}{`bunzip2 -c #1}
\DeclareGraphicsRule{.eps.bz}{eps}{.eps.bb}{`bunzip2 -c #1}
\DeclareGraphicsRule{.eps.gz}{eps}{.eps.bb}{`gunzip -c #1}
\DeclareGraphicsRule{.ps.bz}{eps}{.ps.bb}{`bunzip2 -c #1}
\DeclareGraphicsRule{.ps.gz}{eps}{.ps.bb}{`gunzip -c #1}

  
  
  


\maketitle


\begin{abstract}
  We consider $\Hcurl$-elliptic variational problems on bounded Lipschitz polyhedra
  and their finite element Galerkin discretization by means of lowest order edge
  elements. We assume that the underlying tetrahedral mesh has been created by
  successive local mesh refinement, either by local uniform refinement with hanging
  nodes or bisection refinement. In this setting we develop a convergence theory for
  the the so-called local multigrid correction scheme with hybrid smoothing. We
  establish that its convergence rate is uniform with respect to the number of
  refinement steps. The proof relies on corresponding results for local multigrid in
  a $\Hone$-context along with local discrete Helmholtz-type decompositions of
  the edge element space. 
\end{abstract}



\pagestyle{myheadings} \thispagestyle{plain} \markboth{R. Hiptmair
and W.-Y. Zheng}{Local Multigrid in $\mathbf{H}(\mathbf{\curl})$}

\newcommand{\LFE}{V}   
\newcommand{\VLFE}{\VV} 
\newcommand{\EFE}{\VU} 
\newcommand{\Bas}{\FB} 
\newcommand{\EIP}{\Pibf} 
\newcommand{\LIP}{\Ci} 
\newcommand{\mwf}{h}
\newcommand{\lev}{\ell}
\newcommand{\QIP}{\mathsf{Q}} 
\newcommand{\QIOp }{\QIP_h} 

\begin{table}[!htb]
  \centering
  
\noindent\begin{tabular}[c]{l@{: }p{0.75\textwidth}}
  $\Hcurl$ & Sobolev space of square integrable vector fields on
  $\Omega\subset\bbR^{3}$ with square integrable
  $\curl$ \\
  $\bHcurl{\Gamma_{D}}$ & vector fields in $\Hcurl$ with vanishing tangential
  components on $\Gamma_{D}\subset\partial\Omega$ \\
  $\mesh$, $\Ct$ & tetrahedral finite element meshes, may contain hanging nodes \\
  $\Cn(\mesh)$ & set of vertices (nodes) of a mesh $\mesh$ \\
  $\Ce(\mesh)$ & set of edges of a mesh $\mesh$ \\
  $\rho_K,\rho_{\mesh}$ & shape regularity measures \\
  $h$ & \parbox[t]{0.75\textwidth}{
    -- local meshwidth function for a finite element mesh \\
    -- (as subscript) tag for finite element functions} \\
  $\EFE(\mesh)$ & lowest order edge element space on $\mesh$ \\
  $\Vb_E$ & nodal basis function of $\EFE(\mesh)$ associated with edge $E$ \\
  $\LFE(\mesh)$ & space of continuous piecewise linear functions on $\mesh$ \\
  $\LFE_{2}(\mesh)$ & quadratic Lagrangian finite element space on $\mesh$ \\
  $\wt{\LFE}_{2}(\mesh)$ & quadratic surplus space, see \eqref{def:wtlfe2} \\
  $b_{\Bp}$ & nodal basis function of $\LFE(\mesh)$ (``tent function'') associated
  with vertex $\Bp$ \\
  $\Bas_{X}(\mesh)$ & set of nodal basis functions for finite element space $X$
  on mesh $\mesh$ \\
  $\EIP_{h}$ & nodal edge interpolation operator onto $\EFE(\mesh)$, see 
  \eqref{eq:fem14} \\
  $\LIP_{h}$ & vertex based piecewise linar interpolation onto $\LFE(\mesh)$\\
  $\bbP_{p}$ & space of 3-variate polynomials of total degree $\leq p$\\
  $\ol{\EFE}(\mesh)$, $\ol{\LFE}(\mesh)$ & finite element spaces oblivious of 
  zero boundary conditions \\
  $\prec$ & nesting of finite element meshes \\
  $\lev(K)$ & level of element $K$ in hierarchy of refined meshes \\
  $\omega_l$ & refinement zone, see \eqref{eq:3}\\
  $\Sigma_l$ & refinement strip, see \eqref{eq:sp7} \\
  $\Bas_{\LFE}^{l}$, $\Bas_{\EFE}^{l}$ & sets of basis functions supported inside
  refinement zones, see \eqref{eq:NB} \\
  $\QIOp$ & quasi-interpolation operator, Def.~\ref{def:QIP} \\
\end{tabular}
\caption{Important notation used in this paper}
\label{tab:notations}
\end{table}

\section{Introduction}
\label{sec:introduction}

\newcommand{\DBd}{\Gamma_D} 
\newcommand{\aform}{\blf{a}} 

On a polyhedron \rh{$\Omega{\subset\mathbb{R}^{3}}$}, scaled such that
$\operatorname{diam}(\Omega)=1$, we consider the variational problem: seek
$\Vu\in\bHcurl{\DBd}$ such that
\begin{align}
  \label{eq:VPcurl}
  \underbrace{\SP{\curl\Vu}{\curl\Vv}_{\Ltwo} +
    \SP{\Vu}{\Vv}_{\Ltwo}}_{=: \aform(\Vu,\Vv)}
  = \SP{\Vf}{\Vv}_{\Ltwo}\quad\forall \Vv\in \bHcurl{\DBd}\;.
\end{align}
For the Hilbert space of square integrable vector fields with square integrable
$\curl$ and vanishing tangential components on $\DBd$ we use the symbol
$\bHcurl{\DBd}$, see \cite[Ch.~1]{GIR86} for details. The source term $\Vf$ in
\eqref{eq:VPcurl} is a vector field in $(\Ltwo)^{3}$. The left hand side of
\eqref{eq:VPcurl} agrees with the inner product of $\bHcurl{\DBd}$ and will be
abbreviated by $\aform(\Vu,\Vv)$ (``energy inner product'').

Further, $\DBd$ denotes the part of the boundary $\partial\Omega$ on which
homogeneous Dirichlet boundary conditions in the form of vanishing tangential traces
of $\Vu$ are imposed. The geometry of the Dirichlet boundary part $\DBd$ is supposed
to be simple in the following sense: for each connected component $\Gamma_{i}$ of
$\DBd$ we can find an open Lipschitz domain $\Omega_{i}\subset\bbR^{3}$ such that
\begin{gather}
  \label{eq:1}
  \ol{\Omega}_{i}\cap \ol{\Omega} = \Gamma_{i}\;,\quad
  \Omega_{i}\cap\Omega=\emptyset\;,
\end{gather}
and $\Omega_{i}$ and $\Omega_{j}$ have positive distance for $i\not=j$. Further, the
interior of $\ol{\Omega}\cup\ol{\Omega_{1}}\cup\ol{\Omega_{2}}\dots$ is expected to
be a Lipschitz-domain, too (see Fig.  \ref{fig:attachdom}). This is not a severe
restriction, because variational problems related to \eqref{eq:VPcurl} usually arise
in quasi-static electromagnetic modelling, where simple geometries are common. Of
course, $\DBd=\emptyset$ is admitted.

Lowest order $\bHcurl{\DBd}$-conforming edge elements are widely used for the finite
element Galerkin discretization of variational problems like \eqref{eq:VPcurl}. Then,
for \rh{a solution $\Vu\in(\Hone)^{3}$ with $\curl\Vu\in(\Hone)^{3}$} we can expect
the optimal asymptotic convergence rate
\begin{gather}
  \label{eq:optrate}
  \NHcurl{\Vu-\Vu_{h}} \leq C N_{h}^{-1/3}\;,
\end{gather}
on families of finite element meshes arising from global refinement. Here, $\Vu_{h}$
is the finite element solution, $N_{h}$ the dimension of the finite element space,
and $C>0$ does not depend on $N_{h}$. However, often $\Vu$ will fail to possess the
required regularity due to singularities arising at edges/corners of $\partial\Omega$
and material interfaces \cite{CDN03,COD00}. Fortunately, it seems to be possible to
retain \eqref{eq:optrate} by the use of adaptive local mesh refinement based on a
posteriori error estimates, see \cite{STE06,BDD04} for theory in $H^{1}$-setting,
\cite{ZCW05a,BDH97} for numerical evidence in the case of edge element
discretization, \rh{and \cite{SHO05,HOS07,BHH98} for related theoretical investigations.}

We also need ways to \emph{compute} the asymptotically optimal finite element
solution with optimal computational effort, that is, with a number of operations
proportional to $N_{h}$. This can only be achieved by
means of iterative solvers, whose convergence remains fast regardless of the depth of
refinement. Multigrid methods are the most prominent class of iterative solvers that
achieve this goal. By now, geometric multigrid methods for discrete $\Hcurl$-elliptic
variational problems like \eqref{eq:VPcurl} have become well established
\cite{HIP99,SHW06,WEB05,CFW04}. Their asymptotic theory on sequencies of regularly
refined meshes has also matured \cite{AFW99,GPD03,HIP99,HIP00b,SHO05p}. It confirms
\emph{asymptotic optimality}: the speed of convergence is uniformly fast regardless
of the number of refinement levels involved. In addition, the costs of one step of
the iteration scale linearly with the number of unknowns.

Yet, the latter property is lost when the standard multigrid correction scheme is
applied to meshes generated by pronounced local refinement. Optimal computational
costs can only be maintained, if one adopts the local multigrid policy, which was
pioneered by \rh{A. Brandt et al. in \cite{BAB87}, see also \cite{MIT92}}. Crudely
speaking, its gist is to confine relaxations to ``new'' degrees of freedom located in
zones where refinement has changed the mesh. Thus an exponential increase of
computational costs with the number of refinement levels can be avoided: the total
costs of a V-cycle remain proportional to the number of unknowns. \rh{An algorithm
  blending the local multigrid idea with the geometric multigrid correction scheme of
  \cite{HIP99} is described in \cite{SHW06}.} On the other hand, a proof of uniform
asymptotic convergence has remained elusive so far. It is the objective of this paper
to provide it, see Theorem~\ref{thm:main}.

We recall the key insight that \eqref{eq:VPcurl} is one member of a family of
variational problems. Its kin is obtained by replacing $\curl$ with $\grad$ or
$\Div$, respectively. All these differential operators turn out to be incarnations of
the fundamental exterior derivative of differential geometry, \textit{cf.}
\cite[Sect.~2]{HIP99}. They are closely connected in the deRham complex \cite{AFW06}
and, thus, it is hardly surprising that results about the related
$\bHone{\DBd}$-elliptic variational problem, which seeks $u\in\bHone{\DBd}$ such that
\begin{gather}
  \label{eq:VPgrad}
  \SP{\grad u}{\grad v}_{\Ltwo} +
  \SP{u}{v}_{\Ltwo} = \SP{f}{v}_{\Ltwo}\quad\forall \,v\in \bHone{\DBd}\;,
\end{gather}
prove instrumental in the multigrid analysis for discretized versions of
\eqref{eq:VPcurl}. Here $\bHone{\DBd}$ is the subspace of $\Hone$ whose functions
have vanishing traces on $\DBd$.

Thus, when tackling \eqref{eq:VPcurl}, we take the cue from the local multigrid
theory for \eqref{eq:VPgrad} discretized by means of linear continuous finite
elements. This theory has been developed in various settings, \textit{cf.}
\cite{BAB87,BOY93,BPW91,BRP93,JXU97}. In \cite{AIM01} local refinement with hanging
nodes is treated. Recently, H.~Wu and Z.~Chen \cite{WUC05} proved the uniform
convergence of local multigrid V-cycles on adaptively refined meshes in two
dimensions. Their mesh refinements are controlled by a posteriori error estimators
and carried out according to the ``newest vertex bisection'' strategy introduced,
\rh{independently, in \cite{MIT89,BAE91}}.

As in the case of global multigrid, the essential new aspect of local multigrid
theory for \eqref{eq:VPcurl} compared to \eqref{eq:VPgrad} is the need to deal with
the kernel of the $\curl$-operator, \textit{cf.}  \cite[Sect.~3]{HIP99}. In this
context, the availability of discrete scalar potential representations for
irrotational edge element vector fields is pivotal. Therefore, we devote the entire
Sect.~\ref{sec:edge-elements} to the discussion of edge elements and their
relationship with conventional Lagrangian finite elements. \rh{Meshes with hanging
  nodes will receive particular attention.} Next, in Sect.~\ref{sec:LMG} we present
details about local mesh refinement, because some parts of the proofs rest on the
subtleties of how elements are split. The following Sect.~\ref{sec:local-multigrid}
introduces the local multigrid method from the abstract perspective of successive
subspace correction.

The proof of uniform convergence (Theorem~\ref{thm:main}) is tackled in
Sects.~\ref{sec:stability} and \ref{sec:quasi-orthogonality}, which form the core of
the article. In particular, the investigation of the stability of the local
multilevel splitting requires several steps, the first of which addresses the issue
for the bilinear form from \eqref{eq:VPgrad} and linear finite elements. These
results are already available in the literature, but are re-derived to make the
presentation self-contained. This also applies to the continuous and discrete
Helmholtz-type decompositions covered in Sect.~\ref{sec:helmh-type-decomp}. Many
developments are rather technical and to aid the reader important notations are
listed in Table~\ref{tab:notations}. Eventually, in Sect.~\ref{sec:numer-exper}, we
report two numerical experiments to show the competitive performance of the local
multigrid method and the relevance of the convergence theory.

\begin{remark}
  In this article we forgo generality and do not discuss
  the more general bi-linear form
  \begin{align}
    \label{eq:VPcurlgen}
     \aform(\Vu,\Vv) := \SP{\alpha\curl\Vu}{\curl\Vv}_{\Ltwo} +
      \SP{\beta\Vu}{\Vv}_{\Ltwo}\;,
      \quad\forall \Vu,\Vv\in \bHcurl{\DBd}\;,
  \end{align}  
  with uniformly positive coefficient functions $\alpha,\beta\in L^{\infty}(\Omega)$.
  We do this partly for the sake of lucidity and partly, because the current theory
  cannot provide estimates that are robust with respect to large variations of 
  $\alpha$ and $\beta$, \textit{cf.} \cite{HIX06}. We refer to \cite{XUZ07} for
  further information and references. 
\end{remark}


%
\section{Finite element spaces}
\label{sec:edge-elements}


Whenever we refer to a finite element mesh in this article, we have in mind a
tetrahedral triangulation of $\Omega$, see \cite[Ch.~3]{CIA78}. In certain settings,
it may feature hanging nodes, that is, the face of one tetrahedron can coincide with
the union of faces of other tetrahedra. Further, the mesh is supposed to resolve the
Dirichlet boundary in the sense that $\DBd$ is the union of faces of tetrahedra. The
symbol $\mesh$ with optional subscripts is reserved for finite element meshes and the
sets of their elements alike.

We write $h\in\Linf$ for the piecewise constant function, which assumes value $h_{K}
:= \operatorname{diam}(K)$ in each element $K\in\mesh$. The ratio of
$\operatorname{diam}(K)$ to the radius of the largest ball contained in $K$ is called
the shape regularity measure $\rho_{K}$ \cite[Ch.~3, \S3.1]{CIA78}. The shape
regularity measure $\rho_{\mesh}$ of $\mesh$ is the maximum of all $\rho_{K}$,
$K\in\mesh$.

\subsection{Conforming meshes}
\label{sec:conforming-meshes}

\rh{Provisionally, we consider only finite element meshes $\mesh$ that are
  conforming, that is, each face of a tetrahedron is either contained in
  $\partial\Omega$ or a face of another tetrahedron, see \cite[Ch.~2, \S~2.2]{CIA78}.
  In particular, this rules out hanging nodes}. Following \cite{NED80,BOS88a}, we
introduce the space of lowest order $\bHcurl{\DBd}$-conforming edge finite elements,
also known as Whitney-1-forms \cite{WIT57},
\begin{gather}
  \label{fem:udef}
  \EFE(\mesh) := \{\Vv_{h}\in\bHcurl{\DBd}:\; \forall\, K\in\mesh:
  \exists\, \Ba,\Bb\in\bbR^{3}:\\
  \label{fem:locform}
  \Vv_{h}(\Bx) = \Ba + \Bb\times\Bx,\;\Bx\in K\}\;.
\end{gather}
For a detailed derivation and description please consult \cite[Sect.~3]{HIP02} or the
monographs \cite{MON03,BOS98b}. Notice that $\curl\EFE(\mesh)$ is a space of
piecewise constant vector fields. We also remark that appropriate global degrees of
freedom (d.o.f.) for $\EFE(\mesh)$ are given by
\begin{gather}
  \label{eq:fem12}
  \left\{
    \begin{array}[c]{rcl}
      \EFE(\mesh) &\mapsto&\bbR\\
      \Vv_{h} &\mapsto& \int\nolimits_{E} \Vv_{h}\cdot\mathrm{d}\vec{s}
    \end{array}\right.\quad,\quad E\in \Ce(\mesh)\;,
\end{gather}
where $\Ce(\mesh)$ is the set of edges of $\mesh$ \rh{not contained in $\Gamma_{D}$}.
We write $\Bas_{\EFE}(\mesh)$ for the nodal basis of $\EFE(\mesh)$ dual to the global
d.o.f. \eqref{eq:fem12}. Basis functions are associated with active edges. Hence, we
can write $\Bas_{\EFE}(\mesh) = \{\Vb_{E}\}_{E\in\Ce(\mesh)}$. 
The support of the basis function $\Vb_{E}$ is the union of tetrahedra
sharing the edge $E$. We recall the simple formula for local shape functions
\begin{gather}
  \label{eq:fem13}
  {\Vb_{E}}_{|K} = \lambda_{i}\grad \lambda_{j} - \lambda_{j}\grad \lambda_{i}\quad
  E=[\Ba_{i},\Ba_{j}]\subset\ol{K}
\end{gather}
for any tetrahedron $K\in\mesh$ with vertices $\Ba_{i}$, $i=1,2,3,4$, and
associated barycentric coordinate functions $\lambda_{i}$.

The edge element space $\EFE(\mesh)$ with basis $\Bas_{\EFE}(\mesh)$ is perfectly
suited for the finite element Galerkin discretization of \eqref{eq:VPcurl}. The
discrete problem based on $\EFE(\mesh)$ reads: seek $\Vu_{h}\in\EFE(\mesh)$ such
that
\begin{align}
  \label{eq:VPcurlh}
  \SP{\curl\Vu_{h}}{\curl\Vv_h}_{\Ltwo} +
  \SP{\Vu_{h}}{\Vv_h}_{\Ltwo} = \SP{\Vf}{\Vv_h}_{\Ltwo}\quad\forall
  \Vv_h\in \EFE(\mesh)\;.
\end{align}
The properties of $\EFE(\mesh)$ will be key to constructing and analyzing the local
multigrid method for the large sparse linear system of equations resulting from
\eqref{eq:VPcurlh}. Next, we collect important facts.

The basis $\Bas_{\EFE}(\mesh)$ enjoys uniform $L^{2}$-stability,
meaning the existence of a constant\footnote{The symbol $C$ will stand for generic
  positive constants throughout this article. Its value may vary between different
  occurrences. We will always specify on which quantities these constants may
  depend.} $C=C(\rho_{\mesh})>0$ such that for all
$\Vv_{h} =
\sum\limits_{E\in\Ce(\mesh)}\alpha_{E}\Vb_{E}\in\EFE(\mesh)$,
$\alpha_{E}\in\bbR$,
\begin{gather}
  \label{eq:fem17}
  C^{-1}\NLtwo{\Vv_{h}}^{2}\leq
  \sum\limits_{E\in\Ce(\mesh)}\alpha_{E}^{2}\NLtwo{\Vb_{E}}^{2}
  \leq C\NLtwo{\Vv_{h}}^{2}\;.
\end{gather}

The global d.o.f. induce a nodal edge interpolation operator
\begin{gather}
  \label{eq:fem14}
  \EIP_{h}:
  \left\{
    \begin{array}[c]{rcl}
      \operatorname{dom}(\EIP_{h})\subset\bHcurl{\DBd}&\mapsto& \EFE(\mesh)\\
      \Vv & \mapsto & \sum\limits_{E\in\Ce(\mesh)} \Bigl(\int\nolimits_{E}
      \Vv\cdot\mathrm{d}\vec{s}\Bigr) \cdot \Vb_{E}\;.
    \end{array}\right.
\end{gather}
Obviously, $\EIP_{h}$ provides a local projection, but it turns out to be
unbounded even on $(\Hone)^{3}$. Only for vector fields with discrete
rotation the following interpolation error estimate is available, see
\cite[Lemma~4.6]{HIP02}:

\begin{lemma}
  \label{lem:42}
  The interpolation operator $\EIP_{h}$ is bounded on $\{\Psibf\in (\Hone)^{3},\,
  \curl\Psibf\in\curl\EFE(\mesh)\}\rh{\subset (\Hone)^{3}}$, and for any \emph{conforming}
  mesh there is $C=C(\rho_{\mesh})>0$ such that
    \begin{gather*}
      \NLtwo{\mwf^{-1}(Id-\EIP_{h})\Psibf} \leq C \SNHone{\Psibf}\quad
      \forall \Psibf\in (\Hone)^{3},\;\curl\Psibf\in \curl\EFE(\mesh)\;.
    \end{gather*}
\end{lemma}

If $\Omega$ is homeomorphic to a ball, then $\grad\Hone = \kHcurl :=
\{\Vv\in\Hcurl,\,\curl\Vv=0\}$, that is, $\Hone$ provides scalar potentials for $\Hcurl$.
To state a discrete analogue of this relationship we need the Lagrangian
finite element space of piecewise linear continuous functions on $\mesh$
\begin{gather}
  \label{fem:vdef}
  \LFE(\mesh) := \{u_{h}\in\bHone{\DBd}:\;
  {u_{h}}_{|K}\in\bbP_{1}(K)\;\forall K\in\mesh\}\;,
\end{gather}
where $\bbP_{p}(K)$ is the space of 3-variate polynomials of degree $\leq p$ on $K$.
The global degrees of freedom for $\LFE(\mesh)$ boil down to point evaluations at
the vertices of $\mesh$ away from $\overline{\Gamma}_{D}$ (set $\Cn(\mesh)$). The
dual basis of ``tent functions'' will be denoted by $\Bas_{\LFE}(\mesh) =
{\{b_{\Bp}\}}_{\Bp\in\Cn(\mesh)}$. Its unconditional $L^{2}$-stability is well known:
with a universal constant $C>0$ we have for all $u_{h} =
\sum\limits_{\Bp\in\Cn(\mesh)} \alpha_{\Bp}b_{\Bp}\in\LFE(\mesh)$,
$\alpha_{\Bp}\in\bbR$,
\begin{gather}
  \label{eq:fem18}
  C^{-1}\NLtwo{u_{h}}^{2} \leq \sum\limits_{\Bp\in\Cn(\mesh)} \alpha_{\Bp}^{2}
  \NLtwo{b_{\Bp}}^{2} \leq C \NLtwo{u_{h}}^{2}\;.
\end{gather}

For the nodal interpolation operator related to $\Bas_{\LFE}$ we write
$\LIP_{h}:\operatorname{dom}(\LIP_{h})\subset\bHone{\DBd}\mapsto \LFE(\mesh)$.
Recall the standard estimate for linear interpolation on \emph{conforming} meshes
(i.e., no hanging nodes allowed), \cite[Thm.~3.2.1]{CIA78}, that asserts the existence of
$C=C(k,\rho_{\mesh})>0$ such that
\begin{gather}
  \label{eq:fem15}
  \NHm{\mwf^{k-2}(Id-\LIP_{h})u}{k} \leq C \SNHm{u}{2}\quad \forall
  u\in\Hm{2}\cap\bHone{\DBd},\; k\in\{0,1\}\;.
\end{gather}
Obviously, $\grad \LFE(\mesh)\subset\EFE(\mesh)$, and immediate from Stokes
theorem is the crucial \emph{commuting diagram property}
\begin{gather}
  \label{CDP}
  \EIP_{h}\circ\grad = \grad\circ\LIP_{h}\quad\text{on } \operatorname{dom}(\LIP_{h})\;.
\end{gather}%
This enables us to give an elementary proof of Lemma~\ref{lem:42}.

\newcommand{\PM}{\operatorname{\Cl}}
\begin{proof}[of Lemma~\ref{lem:42}] Pick one $K\in\mesh$ and, without loss of
  generality, assume $0\in K$. Then define the lifting operator, \textit{cf.} 
  the ``Koszul lifting'' \cite[Sect.~3.2]{AFW06},
  \begin{gather}
    \label{eq:Poinc}
    \Vw \mapsto \PM\Vw\;,\quad
    \PM\Vw(\Bx) := \tfrac{1}{3} \Vw(\Bx)\times\Bx
    \;,\quad
    \Bx\in K\;.
  \end{gather}
  Elementary calculations reveal that for any \rh{constant vectorfield $\Vw \in (\Cp_{0}(K))^{3}$}
  \begin{gather}
    \label{eq:femLP1}
    \curl\PM\Vw = \Vw\;,\\
    \label{eq:femLP2}
    \NLtwo[K]{\PM \Vw} \leq h_{K}\NLtwo[K]{\Vw}\;,\\
    \label{eq:femLP3}
    \PM\Vw \in \EFE(K)\;.
  \end{gather}
  The continuity \eqref{eq:femLP2} permits us to extend $\PM$ to $(\Ltwo[K])^{3}$.

  Given $\Psibf\in(\Hone[K])^{3}$ with $\curl\Psibf \equiv \mathrm{const}^{3}$,
  by \eqref{eq:femLP3} we know $\PM\curl\Psibf\in (\Cp_{1}(K))^{3}$. Thus, an inverse
  inequality leads to
  \begin{gather}
    \label{eq:femLP4}
    \SNHone[K]{\PM\curl\Psibf} \leq Ch_{K}^{-1} \NLtwo[K]{\PM\curl\Psibf}
    \overset{\text{\eqref{eq:femLP2}}}{\leq} C \NLtwo[K]{\curl\Psibf}\;,
  \end{gather}
  with $C=C(\rho_{K})>0$. Next, \eqref{eq:femLP1} implies
  \begin{gather}
    \label{eq:femLP5}
    \curl(\Psibf-\PM\curl\Psibf) = 0 \quad\Rightarrow\quad
    \exists p\in\Hone[K]:\quad \Psibf-\PM\curl\Psibf = \grad p \;.
  \end{gather}
  From \eqref{eq:femLP4} we conclude that $p\in\Hm[K]{2}$ and $\SNHm[K]{p}{2} \leq C
  \SNHone[K]{\Psibf}$. Moreover, thanks to the commuting diagram property we have
  \begin{gather}
    \label{eq:femLP6}
    \Psibf-\EIP_{h}\Psibf =
    \underbrace{\PM\curl\Psibf - \EIP_{h}\PM\curl\Psibf}_{=0\;\text{by
        \eqref{eq:femLP3}}} + \grad(p-\LIP_{h}p)\;,
  \end{gather}
  which means, by the standard estimate \eqref{eq:fem15} for linear interpolation on
  $K$,
  \begin{gather*}
    \NLtwo[K]{\Psibf-\EIP_{h}\Psibf} = \SNHone[K]{p-\LIP_{h}p} \leq
    Ch_{K}\SNHm[K]{p}{2} \leq Ch_{K}\SNHone[K]{\Psibf}\;.
  \end{gather*}
  Summation over all elements finishes the proof.
\end{proof}

As theoretical tools we need ``higher order'' counterparts of the above
finite element spaces. We recall the quadratic Lagrangian finite
element space
\begin{gather}
  \label{def:wtlfe2}
  \LFE_{2}(\mesh) := \{u_{h}\in\bHone{\DBd}:\;
  {u_{h}}_{|K}\in\bbP_{2}(K)\;\forall K\in\mesh\}\;,
\end{gather}
and its subspace of quadratic surpluses
\begin{gather}
  \label{def:splfe}
  \wt{\LFE}_{2}(\mesh) := \{u_{h}\in\LFE_{2}(\mesh):\;
  \LIP_{h}u_{h} = 0 \}\;.
\end{gather}
This implies a direct splitting
\begin{gather}
  \label{eq:fem20}
  \LFE_{2}(\mesh) = \LFE(\mesh) \oplus \wt{\LFE}_{2}(\mesh)\;,
\end{gather}
which is unconditionally $H^{1}$-stable: there is a $C=C(\rho_{\mesh})>0$ such
that
\begin{gather}
  \label{eq:fem21}
  C^{-1}\SNHone{{u}_{h}}^{2} \leq \SNHone{(Id-\LIP_{h}){u}_{h}}^{2} +
  \SNHone{\LIP_{h}{u}_{h}}^{2} \leq C\SNHone{{u}_{h}}^{2}\;,
\end{gather}
for all ${u}_{h}\in {\LFE}_{2}(\mesh)$.

Next, we examine the space $(\LFE(\mesh))^{3}$ of continuous piecewise linear
vector fields that vanish on $\DBd$. Standard affine equivalence techniques for edge elements,
see \cite[Sect.~3.6]{HIP02}, confirm
\begin{gather}
  \label{eq:fem22}
  \exists C=C(\rho_{\mesh})>0:\quad
  \NLtwo{\EIP_{h}\Psibf_{h}} \leq C \NLtwo{\Psibf_{h}}\quad\forall
  \Psibf_{h}\in (\LFE(\mesh))^{3}\;.
\end{gather}

\begin{lemma}
  \label{lem:VLFEdec}
  For all $\Psibf_{h}\in (\LFE(\mesh))^{3}$ we can find
  $\wt{v}_{h}\in \wt{\LFE}_{2}(\mesh)$ such that
  \begin{gather*}
    \Psibf_{h} = \BPi_h\Psibf_{h} + \grad \wt{v}_{h}\;,
  \end{gather*}
  and, with $C=C(\rho_{\mesh})>0$,
  \begin{gather*}
    C^{-1}\NLtwo{\Psibf_{h}}^{2} \leq
    \NLtwo{\BPi_h\Psibf_{h}}^{2} + \NLtwo{\grad \wt{v}_{h}}^{2} \leq C \NLtwo{\Psibf_{h}}^{2}\;.
  \end{gather*}
\end{lemma}

For the proof we rely on a very useful insight, which relieves us from all worries concerning
the topology of $\Omega$:

\begin{lemma}
  \label{lem:notop}
  If $\Vv\in\bkHcurl{\DBd}$ and $\EIP_{h}\Vv=0$, then $\Vv\in\grad\bHone{\DBd}$.
\end{lemma}

\begin{proof}
  Since the mesh covers $\Omega$, the relative homology group $H_{1}(\Omega;\DBd)$ is
  generated by a set of edge paths. By definition \eqref{eq:fem12} of the d.o.f. of
  $\EFE(\mesh)$, the path integrals of $\Vv$ along all these paths vanish. As an
  irrotational vector field with vanishing circulation along a complete set of
  $\DBd$-relative fundamental cycles, $\Vv$ must be a gradient.
\end{proof}

\begin{proof}[of Lemma~\ref{lem:VLFEdec}]
  Given $\Psibf_{h}\in (\LFE(\mesh))^{3}$, we decompose it according to
  \begin{gather}
    \label{eq:fem23}
    \Psibf_{h} = \EIP_{h}\Psibf_h +
    \underbrace{(Id-\EIP_{h})\Psibf_{h}}_{=: \grad \wt{v}_{h}}\;.
  \end{gather}
  Note that $\curl (Id-\EIP_{h})\Psibf_{h}$ is piecewise constant with
  vanishing flux through all triangular faces of $\mesh$. Then Stokes'
  theorem teaches that $\curl (Id-\EIP_{h})\Psibf_{h}=0$.

  By the projector property of $\EIP_{h}$, $(Id-\EIP_{h})\Psibf_{h}$ satisfies the
  assumptions of Lemma~\ref{lem:notop}. Taking into account that, moreover, the field
  is piecewise linear, it is clear that $(Id-\EIP_{h})\Psibf_{h} = \grad\psi$ with
  $\psi\in \LFE_{2}(\mesh)$. \rh{Along an arbitrary edge path $\gamma$ in $\mesh$ we
  have $\int\nolimits_{\gamma}(Id-\EIP_{h})\Psibf_{h}\cdot\mathrm{d}\vec{s} =0$
  so that $\psi$ attains the same value (w.l.o.g. $=0$) on all vertices of $\mesh$.}
  The stability of the splitting is a consequence of \eqref{eq:fem22}.
\end{proof}\vspace{2mm}

By definition, the spaces $\EFE(\mesh)$ and $\LFE(\mesh)$ accommodate the homogeneous
boundary conditions on $\DBd$. Later, we will also need finite element spaces
oblivious of boundary conditions, that is, for the case $\DBd=\emptyset$. These
will be tagged by a bar on top, e.g., $\ol{\EFE}(\mesh)$, $\ol{\LFE}(\mesh)$, etc.
The same convention will be employed for notions and operators associated with
finite element spaces: if they refer to the particular case $\DBd=\emptyset$, they will be
endowed with an overbar, e.g. $\ol{\EIP}_{h}$, $\ol{\LIP}_{h}$,
$\ol{\Bas}_{\EFE}(\mesh)$, $\ol{\Cn}(\mesh)$, etc.

\subsection{Meshes with hanging nodes}
\label{sec:meshes-with-hanging}

\rh{Now, general tetrahedral meshes \emph{with hanging nodes} are admitted. We simply
  retain the definitions \eqref{fem:vdef} and \eqref{def:wtlfe2} of the spaces
  $\LFE(\mesh)$ and $\LFE_{2}(\mesh)$ of continuous finite element functions. Degrees
  of freedom for $\LFE(\mesh)$ are point evaluations at \emph{active vertices} of
  $\mesh$. A vertex is called active, if it is not located in the interior of an edge/face
  of $\mesh$ or on $\Gamma_{D}$.}  A 2D\footnote{For ease of visualization, we will
  often elucidate geometric concepts in two-dimensional settings. Their underlying
  ideas are the same in 2D and 3D.} illustration is given in
Fig.~\ref{fig:2drefactvert}.

\begin{figure}[!htb]
  \centering
  \begin{minipage}[c]{0.25\textwidth}\centering
    \includegraphics[width=0.95\textwidth]{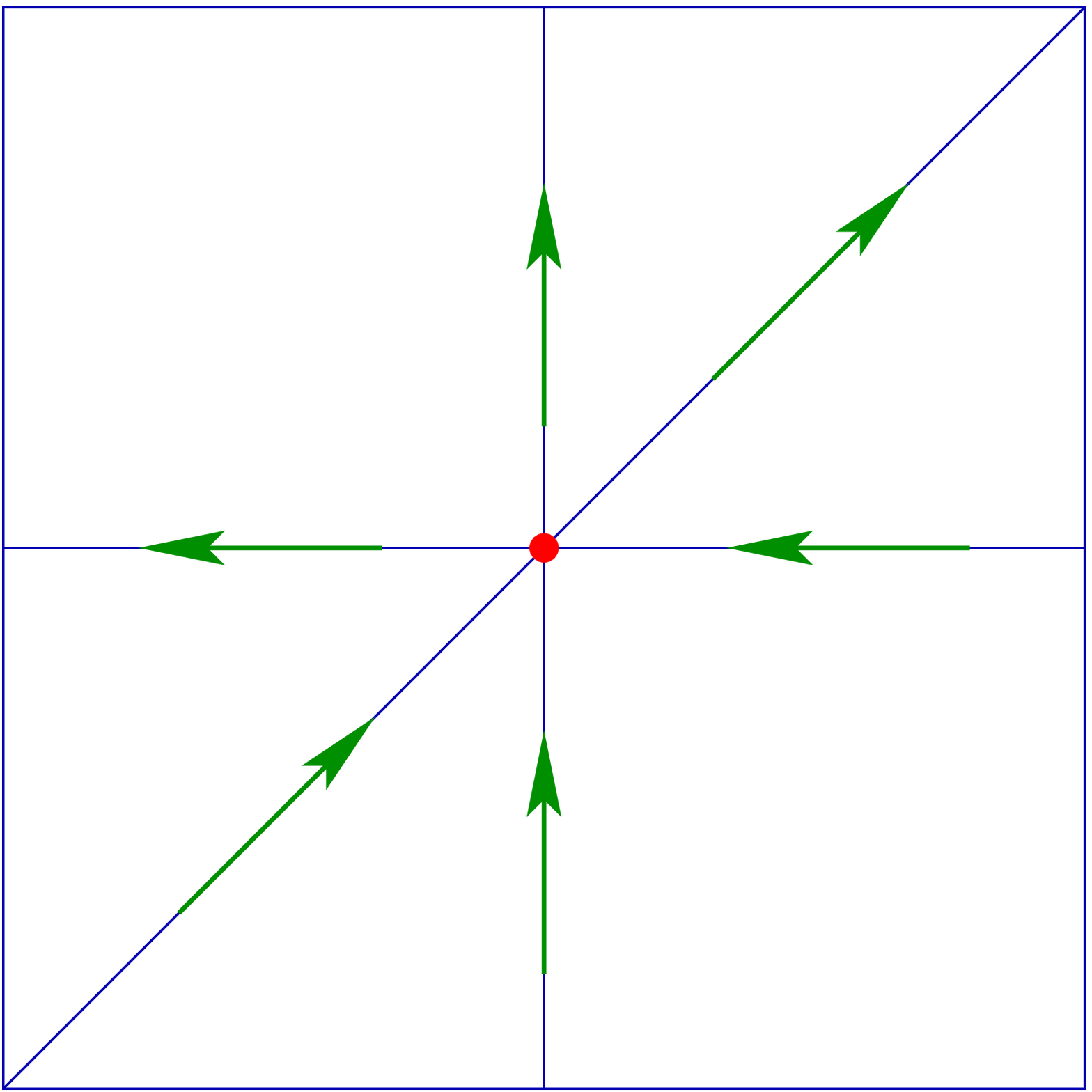}

    $\mesh_{0}$
  \end{minipage}%
  \begin{minipage}[c]{0.25\textwidth}\centering
    \includegraphics[width=0.95\textwidth]{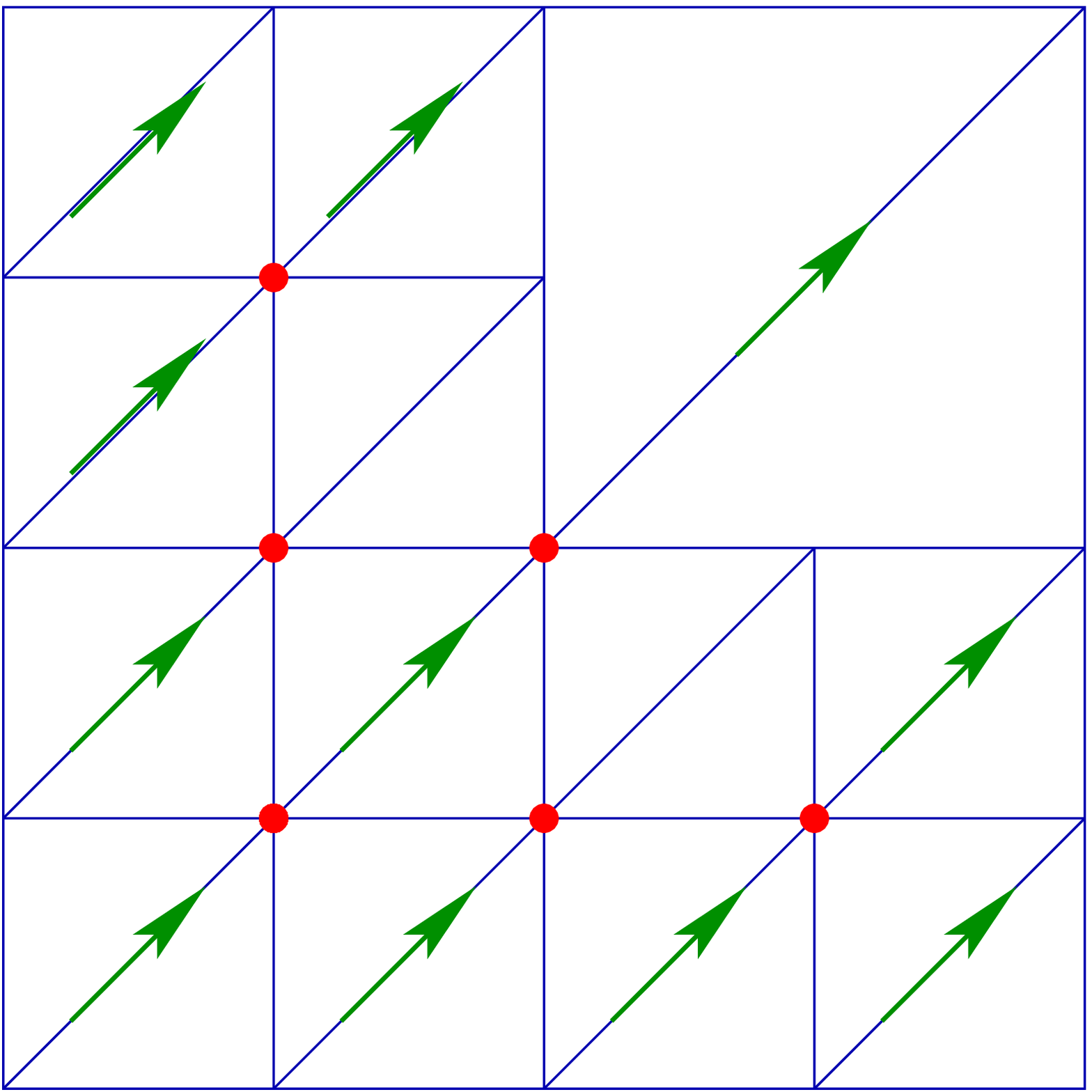}

    $\mesh_{1}$
  \end{minipage}%
  \begin{minipage}[c]{0.25\textwidth}\centering
    \includegraphics[width=0.95\textwidth]{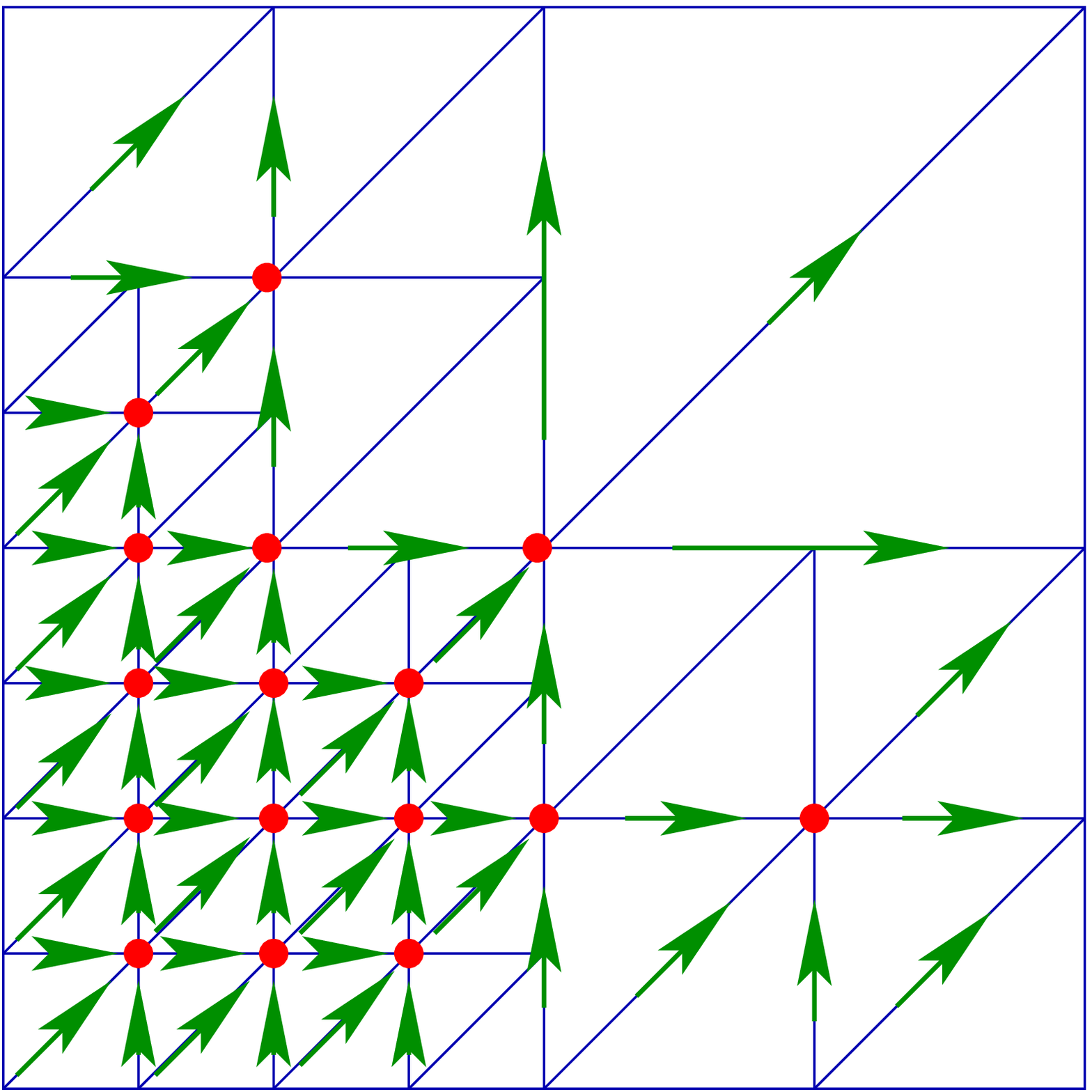}

    $\mesh_{2}$
  \end{minipage}%
  \begin{minipage}[c]{0.25\textwidth}\centering
    \includegraphics[width=0.95\textwidth]{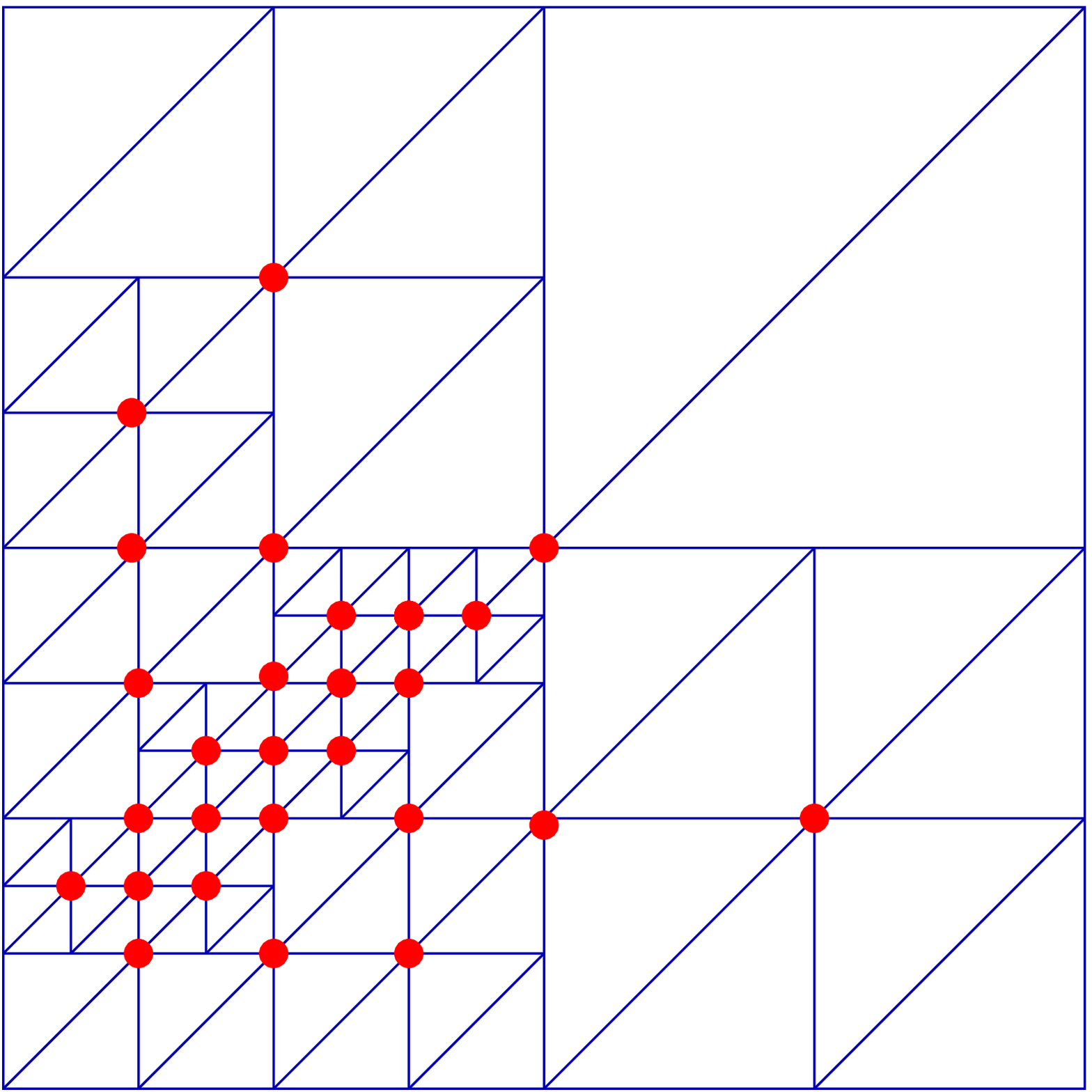}

    $\mesh_{3}$
  \end{minipage}%
  \caption{Active vertices ({\color{red}$\bullet$}) of 2D triangular meshes with hanging nodes,
    $\Omega=]0,1[^{2}$, $\Gamma_{D}=\partial\Omega$. In
    $\mesh_{1},\mesh_{2},\mesh_{3}$ active edges are marked with green arrows. 
  }
  \label{fig:2drefactvert}
\end{figure}

\rh{
The values of a finite element function at the remaining (``slave'') vertices
are determined by recursive affine interpolation. A dual nodal basis
$\Bas_{V}(\mesh)$ and corresponding interpolation operator $\LIP_{h}$ can be
defined as above. 

In principle, the definition \eqref{fem:udef} of the edge element space could be
retained on non-conforming meshes, as well. Yet, for this choice an edge
interpolation operator $\EIP_{h}$ that satisfies the commuting diagram property
\eqref{CDP} is not available. Thus, we construct basis functions directly and
rely on the notion of \emph{active edges}, see Fig.~\ref{fig:2drefactvert}.

\begin{definition}
  \label{def:actedge}
  An edge of $\mesh$ is \emph{active}, if it is an edge of some $K\in\mesh$, not
  contained in $\Gamma_{D}$, and connects two vertices that are either active or
  located on $\Gamma_{D}$.
\end{definition}

We keep the symbol $\Ce(\mesh)$ to designate the set of active edges of $\mesh$.  To
each $E\in\Ce(\mesh)$ we associate a basis function $\Vb_{E}$, which, locally on the
tetrahedra of $\mesh$, is} \rhrev{a polynomial of the form \eqref{fem:locform}}.
\rhrev{In order to fix this basis function completely,} \rh{it suffices to speficify its
  path integrals \eqref{eq:fem12} along \emph{all} edges of $\mesh$. In the spirit of
  duality, we demand
\begin{gather}
  \label{eq:ae}
  \int\nolimits_{F}\Vb_{E}\cdot\mathrm{d}\vec{s} =
  \begin{cases}
    1 & \text{, if } F=E\;,\\
    0 & \text{, if } F\in\Ce(\mesh)\setminus\{E\}\;.
  \end{cases}
\end{gather}
For the non-active (``slave'') edges of $\mesh$ the path integrals of $\Vb_{E}$
(subsequently called ``weights'') are chosen to fit \eqref{CDP}, keeping in mind that
$\Bas_{\EFE}(\mesh) := \{\Vb_{E}\}_{E\in\Ce(\mesh)}$, and that the d.o.f. and
$\EIP_{h}$ are still defined according to \eqref{eq:fem12} and \eqref{eq:fem14},
respectively. Ultimately, we set $\EFE(\mesh) := \Span{\Bas_{\EFE}(\mesh)}$.

Let us explain the policy for setting the weights in the case of the subdivided
tetrahedron of Fig.~\ref{fig:tetsub} with hanging nodes at the midpoints of edges,
which will turn out to be the only relevant situation, \textit{cf.}
Sect.~\ref{sec:helmh-type-decomp}.}
\rhrev{Weights have to be assigned to the ``small edges'' of the refined tetrahedron,
some of which will be active, and some of which will have ``slave'' status, see the
caption of Fig.~\ref{fig:tetsub}.}

\rh{We write the direction vectors of 
slave edges as linear combinations of active edges, for instance,
\begin{align*}
  \Bq_{1} - \Bp_{3} & = \tfrac{1}{2}(\Bp_{4}-\Bp_{3})\;, \\
  \Bq_{1} - \Bq_{2} & = \tfrac{1}{2}(\Bp_{4}+\Bp_{3}) - \tfrac{1}{2}(\Bp_{2}+\Bp_{3})
  = \tfrac{1}{2}(\Bp_{4}-\Bp_{2})\;,\\
  \Bp_{5}-\Bq_{4} & = \Bp_{5}-\tfrac{1}{2}(\Bp_{1}+\Bp_{3}) = 
  \Bp_{5}-\Bp_{1} + \tfrac{1}{2}(\Bp_{1}-\Bp_{3})\;,\\
  \Bq_{4}-\Bq_{3} & = \tfrac{1}{2}(\Bp_{1}+\Bp_{3}) - \tfrac{1}{2}(\Bp_{4}+\Bp_{2}) = 
  \tfrac{1}{2}(\Bp_{1}-\Bp_{2}) + \tfrac{1}{2}(\Bp_{3}-\Bp_{4})\;.
\end{align*}
In a sence, we express slave edges as ``linear combinations'' of active edges.
In a different context, this policy is explained in more detail in \cite{GRH99}. }

\begin{figure}[!htb]
  \centering
  \psfrag{p1}{$\Bp_{1}$}
  \psfrag{p2}{$\Bp_{2}$}
  \psfrag{p3}{$\Bp_{3}$}
  \psfrag{p4}{$\Bp_{4}$}
  \psfrag{p5}{$\Bp_{5}$}
  \psfrag{p6}{$\Bp_{6}$}
  \psfrag{q1}{$\Bq_{1}$}
  \psfrag{q2}{$\Bq_{2}$}
  \psfrag{q3}{$\Bq_{3}$}
  \psfrag{q4}{$\Bq_{4}$}
  \psfrag{E}{$E$}
  \includegraphics[width=0.8\textwidth]{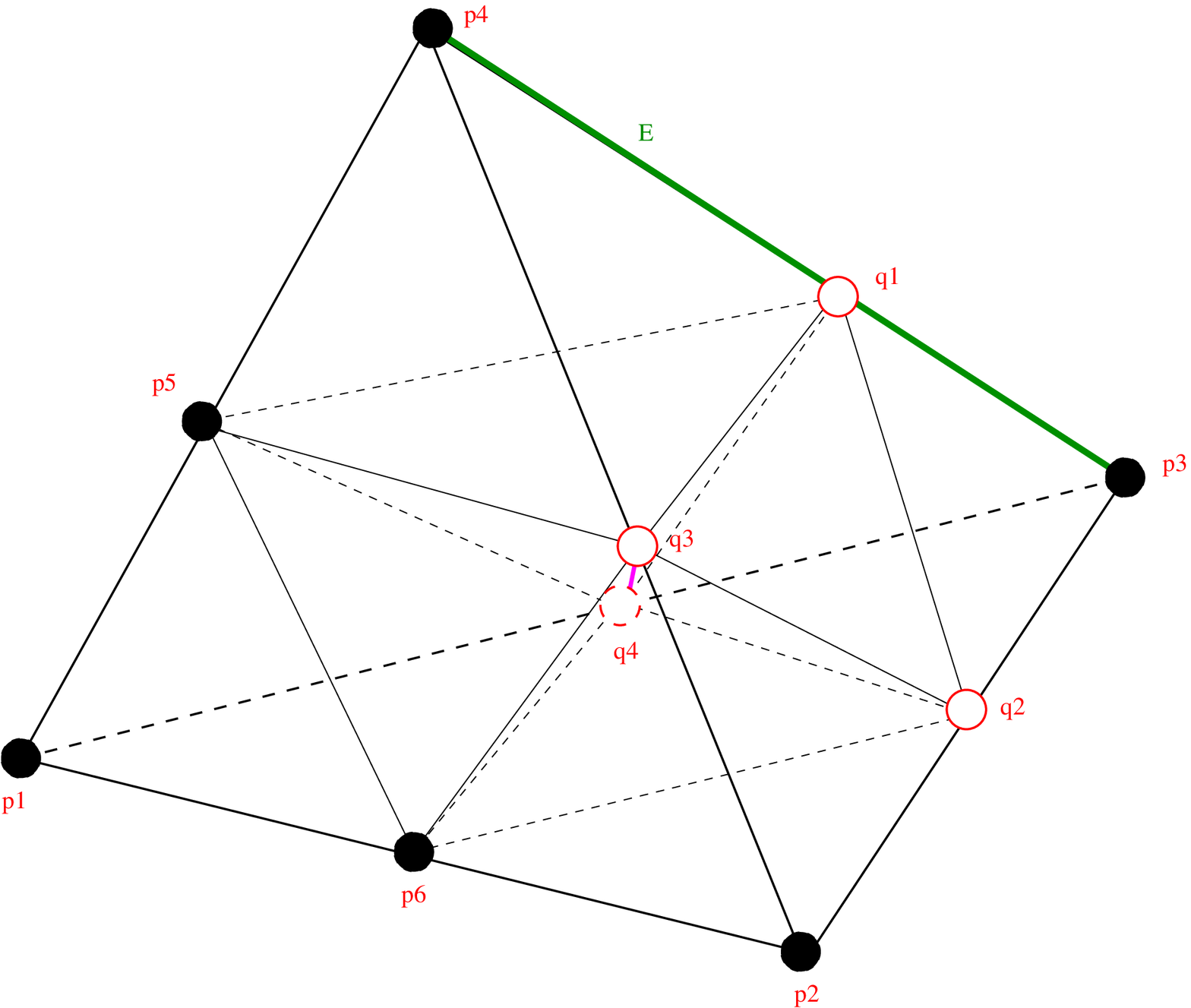}
  \caption{Subdivided tetrahedron, active vertices ($\bullet$)
     $\Bp_{1},\ldots,\Bp_{6}$, slave vertices 
     ({\color{red}$\circ$})
    $\Bq_{1},\ldots,\Bq_{4}$, active edges
    $[\Bp_{1},\Bp_{5}]$,
    $[\Bp_{1},\Bp_{6}]$,
    $[\Bp_{4},\Bp_{5}]$,
    $[\Bp_{2},\Bp_{6}]$,
    $[\Bp_{2},\Bp_{3}]$,
    $[\Bp_{2},\Bp_{4}]$,
    $[\Bp_{3},\Bp_{4}]$,
    $[\Bp_{5},\Bp_{6}]$,
    $[\Bp_{1},\Bp_{3}]$,
    \rhrev{slave edges
      $[\Bp_{1},\Bq_{4}]$,
      $[\Bq_{4},\Bp_{3}]$,
      $[\Bp_{2},\Bq_{2}]$,
      $[\Bq_{2},\Bp_{3}]$,
      $[\Bp_{3},\Bq_{1}]$,
      $[\Bq_{1},\Bp_{4}]$,
      $[\Bp_{6},\Bq_{4}]$,
      $[\Bq_{2},\Bq_{4}]$,
      $[\Bp_{6},\Bq_{2}]$,
      $[\Bq_{1},\Bq_{2}]$,
      $[\Bq_{2},\Bq_{3}]$,
      $[\Bq_{1},\Bq_{3}]$,
      $[\Bq_{1},\Bp_{5}]$,
      $[\Bq_{3},\Bp_{5}]$,
      $[\Bq_{4},\Bp_{5}]$
    } 
  }
  \label{fig:tetsub}
\end{figure}
 
\rh{The coefficients in the combinations tell us the weights. For example, for the active
edge $E=[\Bp_{3},\Bp_{4}]$ in Fig.~\ref{fig:tetsub} they are given in
Table~\ref{tab:sw}. Using these weights and the formula \eqref{eq:fem13}, $\Vb_{E}$
can be assembled on the tetrahedron}
\rhrev{by imposing (see Table~\ref{tab:sw} for notations)
  \begin{gather*}
    \int\nolimits_{S} \Vb_{E}\cdot\mathrm{d}\vec{s} =
    \begin{cases} w_{S} & 
    \text{for any contributing slave edge } S\;,\\
    0 & \text{for all other (slave) edges.}
    \end{cases}\;,\quad S\in \{\text{``small edges''}\}\;.
  \end{gather*}}%

\begin{table}[!htb]
  \centering
  \begin{tabular}[c]{||l|c|c|c|c|c||}\hline
    Slave edge \rhrev{$S$} & $[\Bq_{1},\Vp_{4}]$ & $[\Bp_{3},\Bq_{1}]$ & $[\Bq_{2},\Bq_3]$ &
    $[\Bq_{4},\Bp_{5}]$ &
    $[\Bq_{3},\Bq_{4}]$ \\\hline
    weight \rhrev{$w_{S}$} &   $\frac{1}{2}$    &   $\frac{1}{2}$    & $\frac{1}{2}$     &
    $\frac{1}{2}$ & $-\frac{1}{2}$ \\\hline
  \end{tabular}
  \caption{Weights for slave edges in Fig.~\ref{fig:tetsub} relative to
    active edge $E=[\Bp_{3},\Bp_{4}]$. Only slave edges with non-zero weights 
  are listed.}
\label{tab:sw}
\end{table}

\rh{Firstly, the procedure for the selection of weight guarantees that $\grad
  \LFE(\mesh) \subset \EFE(\mesh)$.}  \rhrev{%
  For illustration, we single out the gradient $\Vw_{h}$ of the nodal basis function
  belonging to vertex $\Bp_{5}$ in Fig.~\ref{fig:tetsub}. Its path integral equals
  $1$ along the (oriented) edges $[\Bp_{1},\Bp_{5}],[\Bp_{3},\Bp_{5}],
  [\Bp_{6},\Bp_{5}],[\Bq_{4},\Bp_{5}],[\Bq_{3},\Bp_{5}],[\Bq_{1},\Bp_{5}]$, and
  vanishes on all other edges. Hence we expect
  \begin{gather}
    \label{eq:slaveedge}
    \Vw_{h} = \Vb_{[\Bp_{1},\Bp_{5}]} + \Vb_{[\Bp_{3},\Bp_{5}]} + \Vb_{[\Bp_{6},\Bp_{5}]}\;.
  \end{gather}
  This can be verified through showing equality of path integrals along slave
  edges. We take a close look at the slave edge $[\Bq_{4},\Bp_{5}]$. By construction
  the basis functions belonging to active edges satisfy
  \begin{align*}
    & \int\limits_{[\Bq_{4},\Bp_{5}]}\Vb_{[\Bp_{1},\Bp_{5}]}\cdot\mathrm{d}\vec{s} =
    1\;,
    \int\limits_{[\Bq_{4},\Bp_{5}]}\Vb_{[\Bp_{5},\Bp_{4}]}\cdot\mathrm{d}\vec{s} =
    0\;,
     \int\limits_{[\Bq_{4},\Bp_{5}]}\Vb_{[\Bp_{3},\Bp_{4}]}\cdot\mathrm{d}\vec{s} =
     0\;,\\ & 
     \int\limits_{[\Bq_{4},\Bp_{5}]}\Vb_{[\Bp_{1},\Bp_{3}]}\cdot\mathrm{d}\vec{s} =
     -\tfrac{1}{2}\;,
    \int\limits_{[\Bq_{4},\Bp_{5}]}\Vb_{[\Bp_{1},\Bp_{6}]}\cdot\mathrm{d}\vec{s} =
     0\;
    \int\limits_{[\Bq_{4},\Bp_{5}]}\Vb_{[\Bp_{2},\Bp_{6}]}\cdot\mathrm{d}\vec{s} =
     0\;,\\ & 
    \int\limits_{[\Bq_{4},\Bp_{5}]}\Vb_{[\Bp_{2},\Bp_{3}]}\cdot\mathrm{d}\vec{s} =
     0\;.
   \end{align*}
   Then, evidently, 
   \begin{multline*}
     1= \int\limits_{[\Bq_{4},\Bp_{5}]}\Vw_{h}\cdot\mathrm{d}\vec{s} \\
     = \int\limits_{[\Bq_{4},\Bp_{5}]}
     \Vb_{[\Bp_{1},\Bp_{5}]}\cdot\mathrm{d}\vec{s}
     + \int\limits_{[\Bq_{4},\Bp_{5}]}\Vb_{[\Bp_{3},\Bp_{5}]}\cdot\mathrm{d}\vec{s}
     + \int\limits_{[\Bq_{4},\Bp_{5}]}\Vb_{[\Bp_{6},\Bp_{5}]}\cdot\mathrm{d}\vec{s}
     = 1 + 0 + 0 \;.
  \end{multline*}
  The same considerations apply to all other slave edges and \eqref{eq:slaveedge}
  is established.}
\rh{Secondly, the construction ensures the commuting diagram
  property \eqref{CDP}: again appealing to Fig.~\ref{fig:tetsub} we find, for
  example,
\begin{multline*}
  \int\limits_{[\Bq_{3},\Bq_{4}]}\grad \LIP_{h}u\cdot\mathrm{d}\vec{s} = 
  \LIP_{h}u(\Bq_{4})- \LIP_{h}u(\Bq_{3}) = \\
  \tfrac{1}{2}(u(\Bp_{4})+u(\Bp_{2})) - \tfrac{1}{2}(u(\Bp_{1})+u(\Bp_{3}))
  = \tfrac{1}{2}\int\limits_{[\Bp_{3},\Bp_{4}]}\grad u\cdot \mathrm{d}\vec{s} + 
  \tfrac{1}{2}\int\limits_{[\Bp_{1},\Bp_{2}]}\grad u\cdot \mathrm{d}\vec{s} = \\
  \tfrac{1}{2}\int\limits_{[\Bp_{3},\Bp_{4}]}\grad u\cdot \mathrm{d}\vec{s} + 
  \tfrac{1}{4}\int\limits_{[\Bp_{1},\Bp_{6}]}\grad u\cdot \mathrm{d}\vec{s} + 
  \tfrac{1}{4}\int\limits_{[\Bp_{6},\Bp_{2}]}\grad u\cdot \mathrm{d}\vec{s}\;.
\end{multline*}
In words, combining the path integrals of $\grad u$ along active edges with the
relative weights of the slave edge $[\Bq_{3},\Bq_{4}]$ yields the same result as
evaluating the path integral of the gradient of the interpolant $\LIP_{h}u$ along
$[\Bq_{3},\Bq_{4}]$. 

The definitions \eqref{def:wtlfe2} and \eqref{def:splfe} also carry over to meshes
with hanging nodes. This remains true for the splitting asserted in Lemma~2.2.
However, though the algebraic relationships like \eqref{CDP} remain valid, the
estimates and norm equivalences of the previous section do not hold for general
families of meshes with hanging nodes. This entails restrictions on the location of
hanging nodes, whose discussion will be postponed until
Sect.~\ref{sec:recurs-bisect-refin}, \textit{cf.} Assumption~\ref{ass:hn}. }

\begin{remark}
  Our presentation is confined to tetrahedral meshes and lowest order edge elements
  for the sake of simplicity. \rh{Extension of all results to hexahedral meshes and
  higher order edge elements is possible, but will be technical and tedious.}
\end{remark}


\section{Local mesh refinement}
\label{sec:LMG}

We study the case where the actual finite element mesh $\mesh_{h}$ of $\Omega$ has
been created by successive local refinement of a relatively uniform initial mesh
$\mesh_{0}$.  Concerning $\mesh_{h}$ and $\mesh_{0}$ the following \emph{asumptions}
will be made:
\begin{enumerate}
\item Given $\mesh_{0}$ and $\mesh_{h}$ we can construct a \emph{virtual} refinement hierarchy
  of $L+1$ \emph{nested}\footnote{
  two finite element meshes $\Cm$ and $\Ct$ are nested, $\Cm\prec\Ct$, if every
  element of $\Cm$ is the union of elements of $\Ct$.} tetrahedral meshes, $L\in\bbN$:
  \begin{gather}
    \label{eq:ref2}
    \mesh_{0}\; \prec\; \mesh_{1} \;\prec\; \mesh_{2}\;\prec\;\cdots\;\prec \mesh_{L} = \mesh_{h}\;.
  \end{gather}
  Please note that the virtual refinement hierarchy may be different from the actual
  sequence of meshes spawned during adaptive refinement\footnote{%
    For the local multigrid algorithm examined in this article the implementation
    must provide access to the virtual refinement hierarchy. This entails suitable
    bookkeeping data structures, which are available in the ALBERTA package used
    for the numerical experiments in Sect.~\ref{sec:numer-exper}}.
\item Inductively, we assign to each tetrahedron $K\in\mesh_{l}$ a level
  $\lev(K)\in\bbN_{0}$ by counting the number of subdivisions it took to generate it
  from an element of $\mesh_{0}$.
\item For all $0\leq l < L$ the mesh $\mesh_{l+1}$  is created by subdividing some or all
  of the tetrahedra in $\{K\in\mesh_{l}:\,\lev(K)=l\}$.
\item The shape regularity measures of the meshes $\mesh_{l}$ are uniformly bounded
  independently of $L$.
\end{enumerate}

Refinement may be local, but it must be regular in the following sense, \textit{cf.}
\rh{\cite[Sect.~4.2.2]{OSW94} and} \cite{WUC05}: we can find a second sequence of
nested tetrahedral meshes of $\Omega$
\begin{gather}
  \label{eq:ref24}
  \mesh_{0}=\wh{\mesh}_{0}\; \prec\;
  \wh{\mesh}_{1} \;\prec\; \wh{\mesh}_{2}\;\prec\;\cdots\;\prec \wh{\mesh}_{L}\;.
\end{gather}
that satisfies
\begin{enumerate}
\item $\mesh_{l}\prec \wh{\mesh}_{l}$ and $\{K\in\mesh_{l}:\,\lev(K)=l\} \subset
  \wh{\mesh}_{l}$, $l=0,\ldots,L$,
\item that the shape regularity measure $\rho_{\wh{\mesh}_{l}}$ is bounded
  independently of $l$,
\item and that there exist two constants $C>0$ and $0<\theta<1$ independent of $l$
  and $L$ such that
  \begin{gather}
    \label{eq:ref25}
    C^{-1} \theta^{l} \leq h_{K} \leq C \theta^{l}\quad\forall
    K\in\wh{\mesh}_{l}\;,\quad
    0\leq l \leq L\;.
  \end{gather}
  This means that the family ${\{\wh{\mesh}_{l}\}}_{l}$ is quasi-uniform. Hence, it
  makes sense to refer to a mesh width $h_{l} := \max\{h_{K},\,K\in\wh{\mesh}_{l}\}$
  of $\wh{\mesh}_{l}$. It decreases geometrically for growing $l$.
\end{enumerate}

\rh{Our analysis targets two popular tetrahedral refinement schemes that generate
  sequences of meshes that meet the above requirements.} 

\subsection{Local regular refinement}
\label{sec:local-regul-refin}
\rh{This scheme produces $\mesh_{l+1}$ by splitting some of the tetrahedra of the
  current mesh $\mesh_{l}$ into eight smaller ones, possibly creating hanging nodes
  in the process \cite{AIM01}.}  An illustrative 2D example with hanging nodes is
depicted in Figure \ref{fig:2dref}.  The accompanying sequence
$\{\wh{\mesh}_{l}\}_{0\leq l \leq L}$ is produced by global regular refinement, which
implies \eqref{eq:ref25} with $\theta=\frac{1}{2}$. Uniform shape-regularity can also
be guaranteed for repeated regular refinement of tetrahedra, see \cite{BEY94}.

\begin{figure}[!htb]
  \centering
  \begin{minipage}[c]{0.25\textwidth}\centering
    \includegraphics[width=0.95\textwidth]{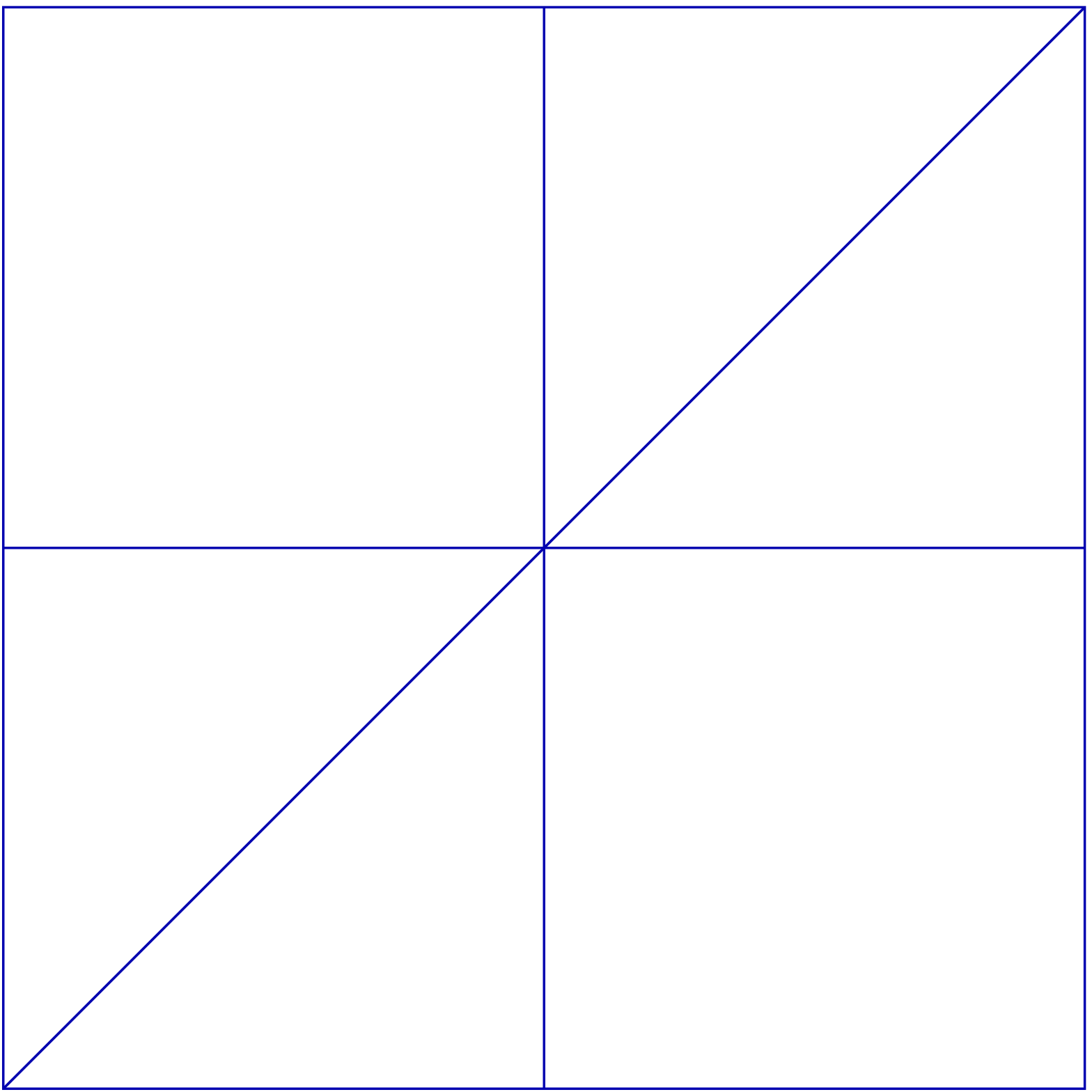}

    $\mesh_{0}$
  \end{minipage}%
  \begin{minipage}[c]{0.25\textwidth}\centering
    \includegraphics[width=0.95\textwidth]{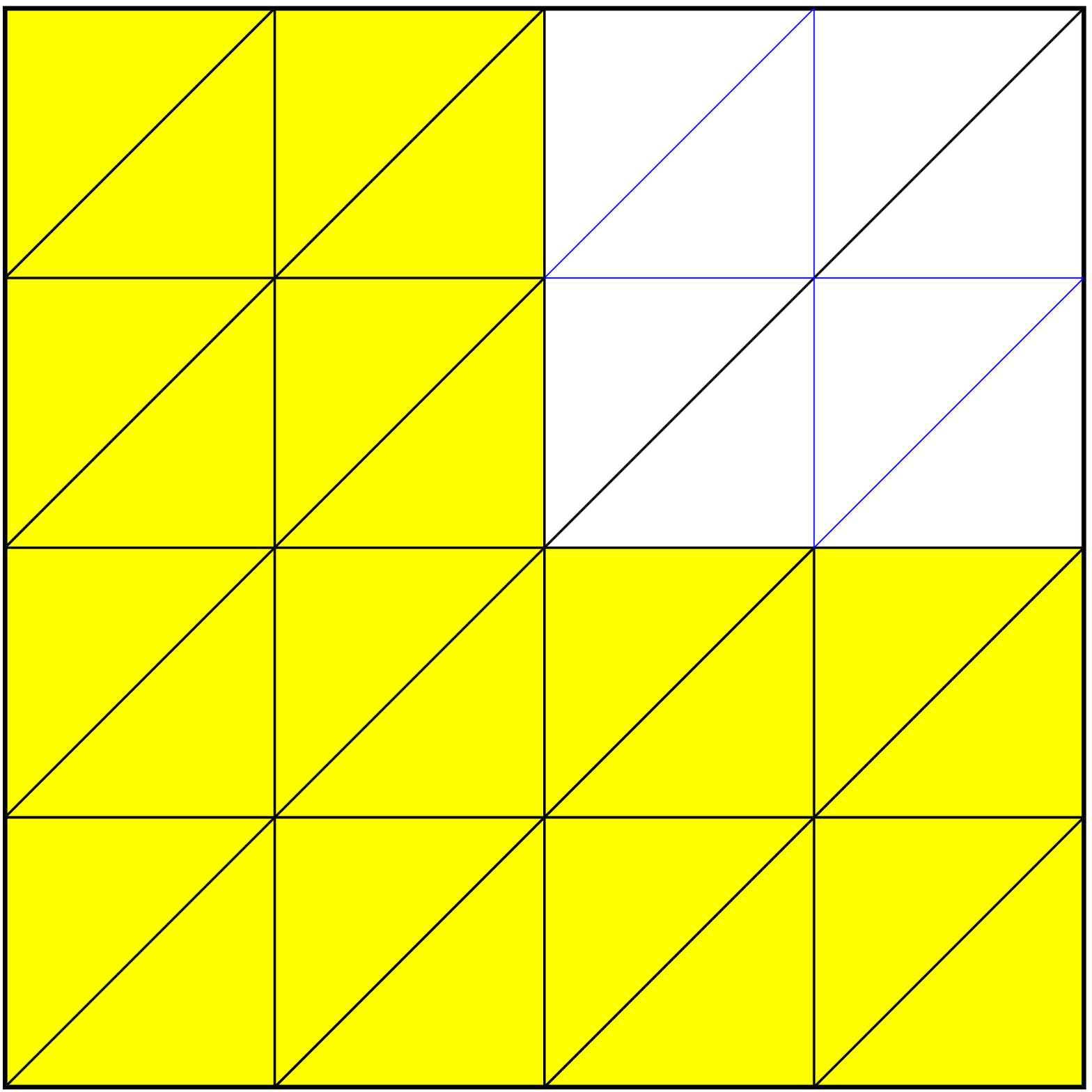}

    $\mesh_{1}$
  \end{minipage}%
  \begin{minipage}[c]{0.25\textwidth}\centering
    \includegraphics[width=0.95\textwidth]{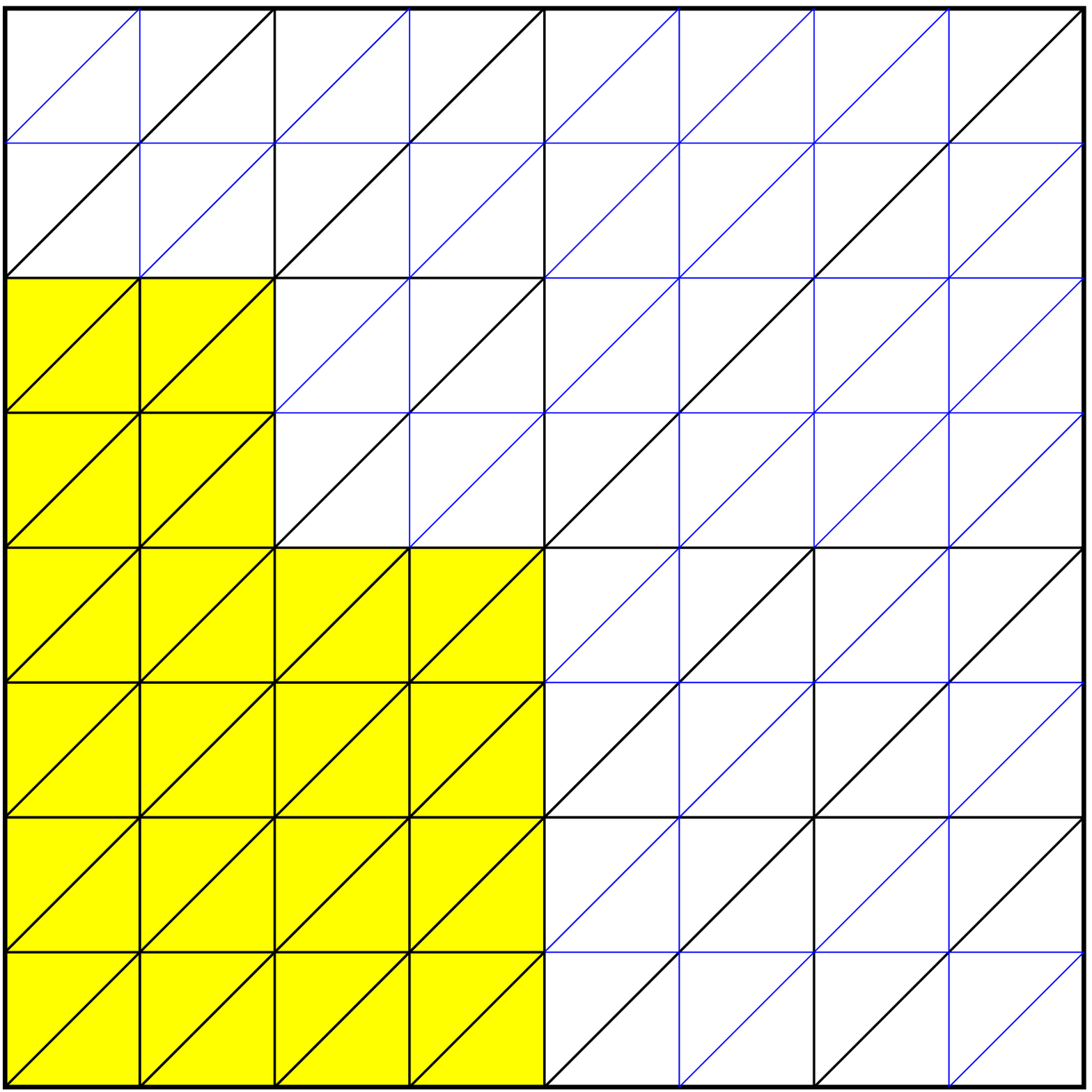}

    $\mesh_{2}$
  \end{minipage}%
  \begin{minipage}[c]{0.25\textwidth}\centering
    \includegraphics[width=0.95\textwidth]{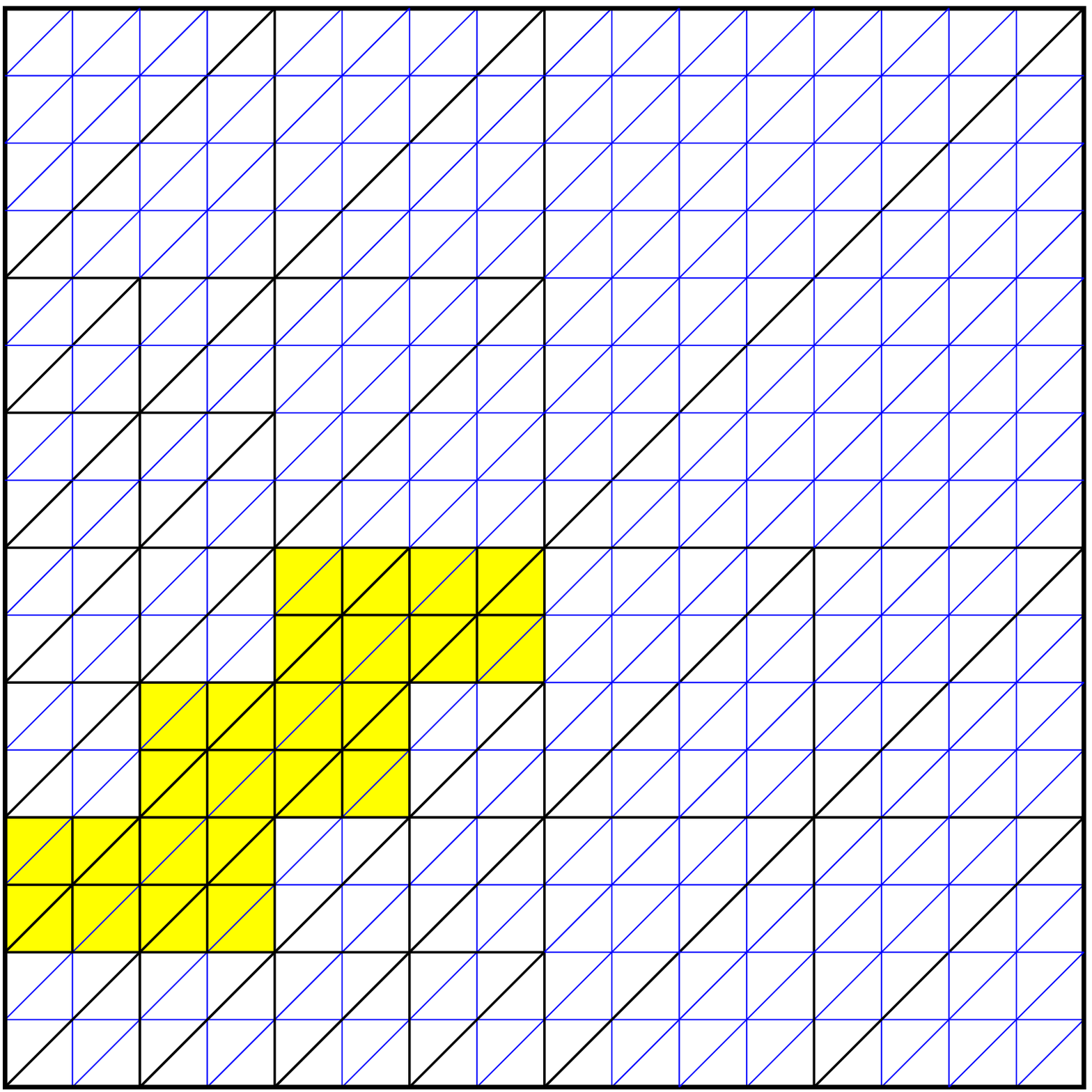}

    $\mesh_{3}=\mesh_{h}$
  \end{minipage}%
  \caption{Virtual refinement hierarchy for 2D triangular meshes.
    The quasi-uniform sequence $\{\wh{\mesh}_{l}\}_{0\leq l \leq L}$ is sketched in
    blue. Elements of $\mesh_{l}$ eligible for further subdivision are marked yellow.
  }
  \label{fig:2dref}
\end{figure}

\rh{The meshes occurring in the virtual refinement hierarchy need not agree with
the meshes that arise during adaptive refinement in an actual computation. Yet,
given $\mesh_{h}$, the virtual refinement hierarchy can always be found \emph{a posteriori}.
Write $\mesh_{\mathrm{hier}}$ for the union of all tetrahedra ever created during the
refinement process. Then, for $ 0<l<L$, define
\begin{gather}
  \label{eq:virt-mesh}
  \mesh_{l} := \left\{K\in\mesh_{\mathrm{hier}}:\; 
    \begin{array}[c]{l}
      \lev(K) \leq l\quad\text{and }K\;\text{does not contain a }\\
      K'\in \mesh_{\mathrm{hier}}\setminus\{K\}\;\text{with }
      \lev(K')\le l    
    \end{array}
    \right\}\;.
\end{gather}
}

\rh{
Using the construction of finite element spaces detailed in
Sect.~\ref{sec:meshes-with-hanging}, the local multigrid algorithms can handle any
kind of local regular refinement. Yet, convergence may degrade unless we curb extreme
jumps of local meshwidth. Thus, we assume the following throughout the remainder of 
this paper. 

\begin{assumption}
  \label{ass:hn}
  Any edge of $\mesh_{h}$ may contain at most one hanging node.
\end{assumption}

This will automatically be satisfied for all meshes $\mesh_{l}$ of the virtual
refinement hierarchy. Consequently, hanging nodes can occur only in a few geometric
configurations, one of which is depicted in Fig.~\ref{fig:tetsub}. This paves the way
for using mapping techniques and scaling arguments, see \cite[Sect.~3.6]{HIP02},
which confirm the following generalization of results of
Sect.~\ref{sec:conforming-meshes}. Of course, we rely on the constructions of finite
element spaces and interpolation operators described in
Sect.~\ref{sec:meshes-with-hanging}.

\begin{proposition}
  \label{prop:hn}
  Under Assumption \ref{ass:hn} the $L^{2}$-stability of bases, see \eqref{eq:fem17},
  \eqref{eq:fem18}, carries over uniformly to meshes created by local regular
  refinement. So do Lemmas~\ref{lem:42}, \ref{lem:VLFEdec}, and Estimates
  \eqref{eq:fem21}, \eqref{eq:fem22}.
\end{proposition}

Summing up, Assumption \ref{ass:hn} makes it possible to use the results obtained
in Sect.~\ref{sec:conforming-meshes} in the case of local regular refinement as
well. To avoid a proliferation of labels, we are going to quote the statements
from Sect.~\ref{sec:conforming-meshes} even when we mean their generalization
to meshes with hanging nodes. 
}

\subsection{\rh{Recursive bisection refinement}}
\label{sec:recurs-bisect-refin}
This procedure involves splitting a tetrahedron into two by promoting the midpoint
of the so-called refinement edge to a new vertex. Variants of bisection differ by the
selection of refinement edges: The iterative bisection strategy by B\"{a}nsch
\cite{BAE91,AMP98a} needs the intermediate handling of hanging nodes. The recursive
bisection strategies of \cite{KOS94a,MAU95,TRA97} do not create such hanging nodes
and, therefore, are easier to implement.  But for special $\mesh_0$, the two
recursive algorithms result in exactly the same tetrahedral meshes as the iterative
algorithm.  Since our implementation relies on the bisection algorithm of
\cite{KOS94a}, we outline its bisection policy in the following.  For more
information on bisection algorithms, we refer to \cite{ALBTA,STE08}.

\newcommand{\Type}{\operatorname{\textsf type}}
For the recursive bisection algorithm of \cite{KOS94a}, the
bisections of tetrahedra are totally determined by the local
vertex numbering of $\mesh_0$, plus a prescribed type for every
element in $\mesh_0$. Each tetrahedron $K$ is endowed with the
local indices 0, 1, 2, and 3 for its vertices. The refinement edge
of each element is always set to be the edge connecting vertex 0
and vertex 1. After bisection of $K$, the ``child  tetrahedron''
of $K$ which contains vertex 0 of $K$ is denoted by Child[0] and
the other one is denoted by Child[1]. The types of Child[0] and
Child[1] are defined by
$$\Type(\hbox{Child[0]})=\Type(\hbox{Child[1]})=(\Type(K)+1)\mod 3.$$

The new vertex at the midpoint of the refinement edge of $K$ is
always numbered by 3 in Child[0] and Child[1]. The four vertices
of $K$ are numbered in Child[0] and Child[1] as follows (see Fig.
\ref{fig:2tetra}):
\begin{eqnarray*}
\hbox{In Child[0]}: \quad && (0,2,3) \rightarrow (0,1,2),\\
\hbox{In Child[1]}: \quad && (0,2,3) \rightarrow (0,2,1)
\quad \hbox{, if}\quad \Type(K)=0,\\
\hbox{In Child[1]}: \quad && (0,2,3) \rightarrow (0,1,2) \quad
\hbox{, if}\quad \Type(K)>0.
\end{eqnarray*}
This recursive bisection creates only a small number of similarity
classes of tetrahedra, see \cite{KOS94a,ALBTA,TRA97}.

\begin{figure}[!htb]
\centering
\includegraphics*[width=5in,height=2.5in]{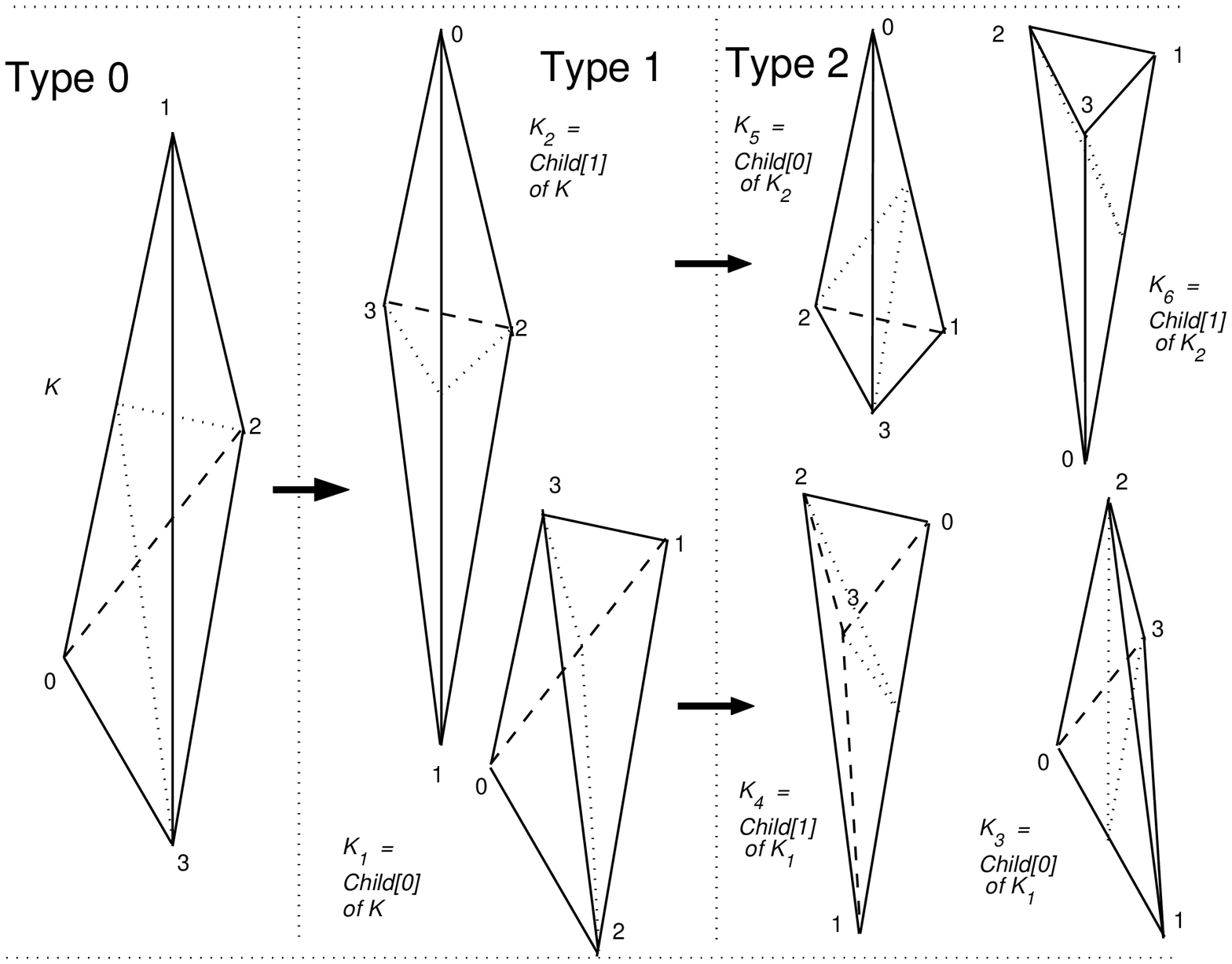}\\
\includegraphics*[width=5in,height=2in]{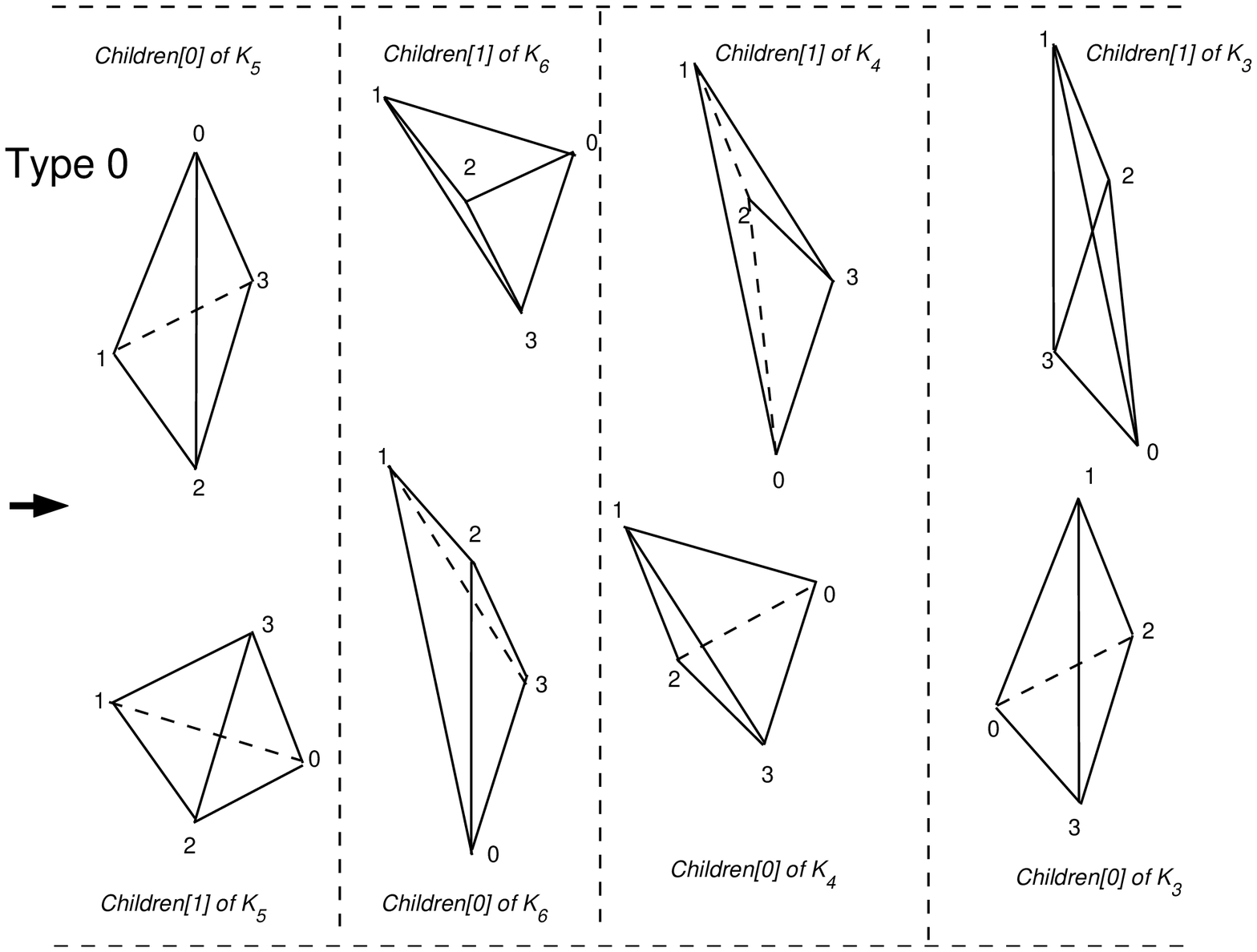}
\caption{Bisection of tetrahedra in the course of recursive bisection. Assignment of
  types to children}
\label{fig:2tetra}
\end{figure}

Fig~\ref{fig:nvb} shows a 2D example of the recursive bisection refinement (the
algorithm for 2D case is called ``the newest vertex bisection'' in \cite{MIT92}).
Similar to the 3D algorithm, for any element $K$, its three vertices are locally
numbered by 0, 1, and 2, its refinement edge is the edge between vertex 0 and 1. The
newly created vertex in the two children of $K$ are numbered by 2. In the child
element containing vertex 0 of $K$, vertex 0 and 2 of $K$ are renumbered by 1 and 0
respectively. In the other child element, vertex 1 and 2 of $K$ are renumbered by 0
and 1 respectively.

\begin{figure}[!htb]
  \centering

  \begin{minipage}[c]{0.33\textwidth}\centering
    \includegraphics[width=0.95\textwidth]{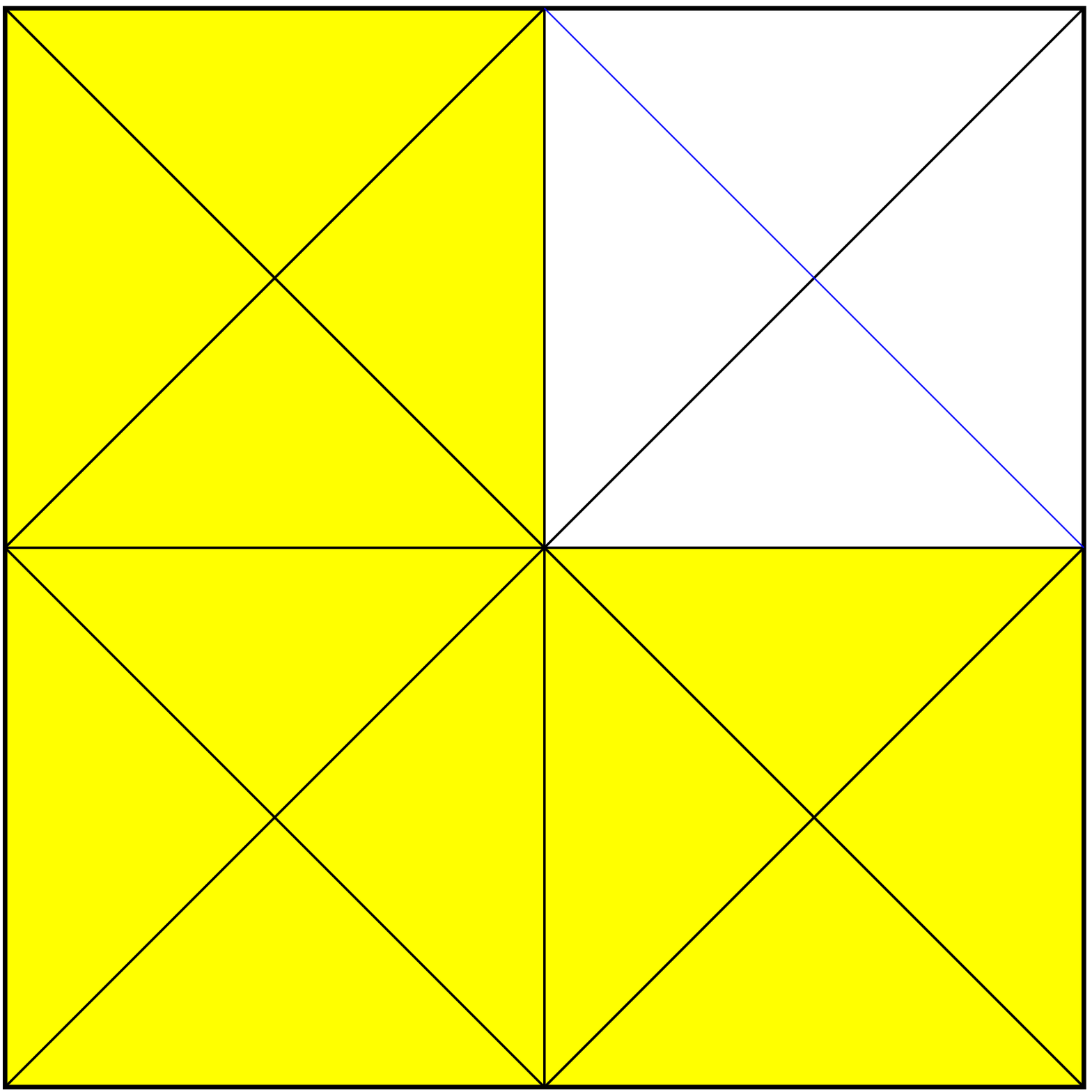}

    $\mesh_{1}$
  \end{minipage}%
  \begin{minipage}[c]{0.33\textwidth}\centering
    \includegraphics[width=0.95\textwidth]{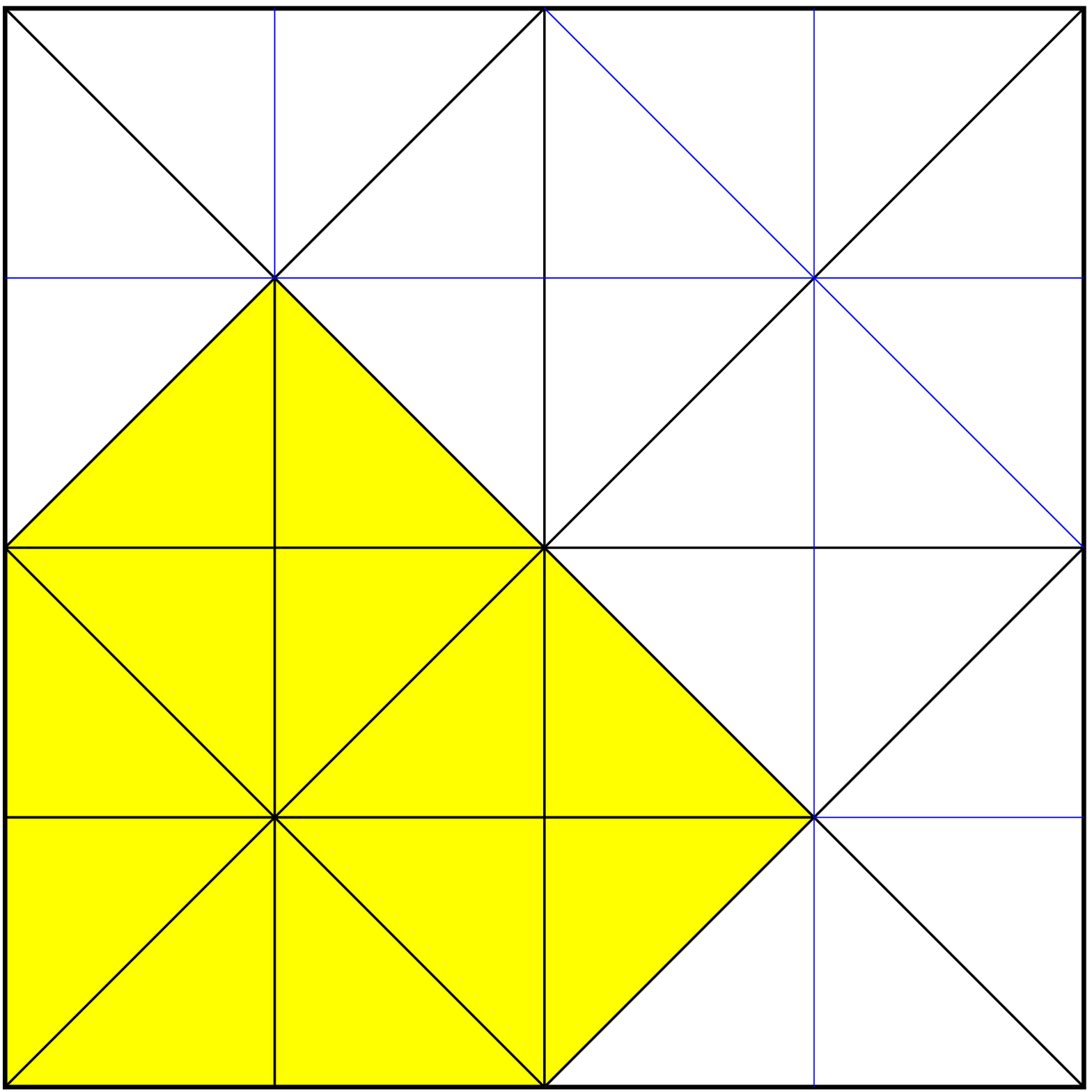}

    $\mesh_{2}$
  \end{minipage}%
  \begin{minipage}[c]{0.33\textwidth}\centering
    \includegraphics[width=0.95\textwidth]{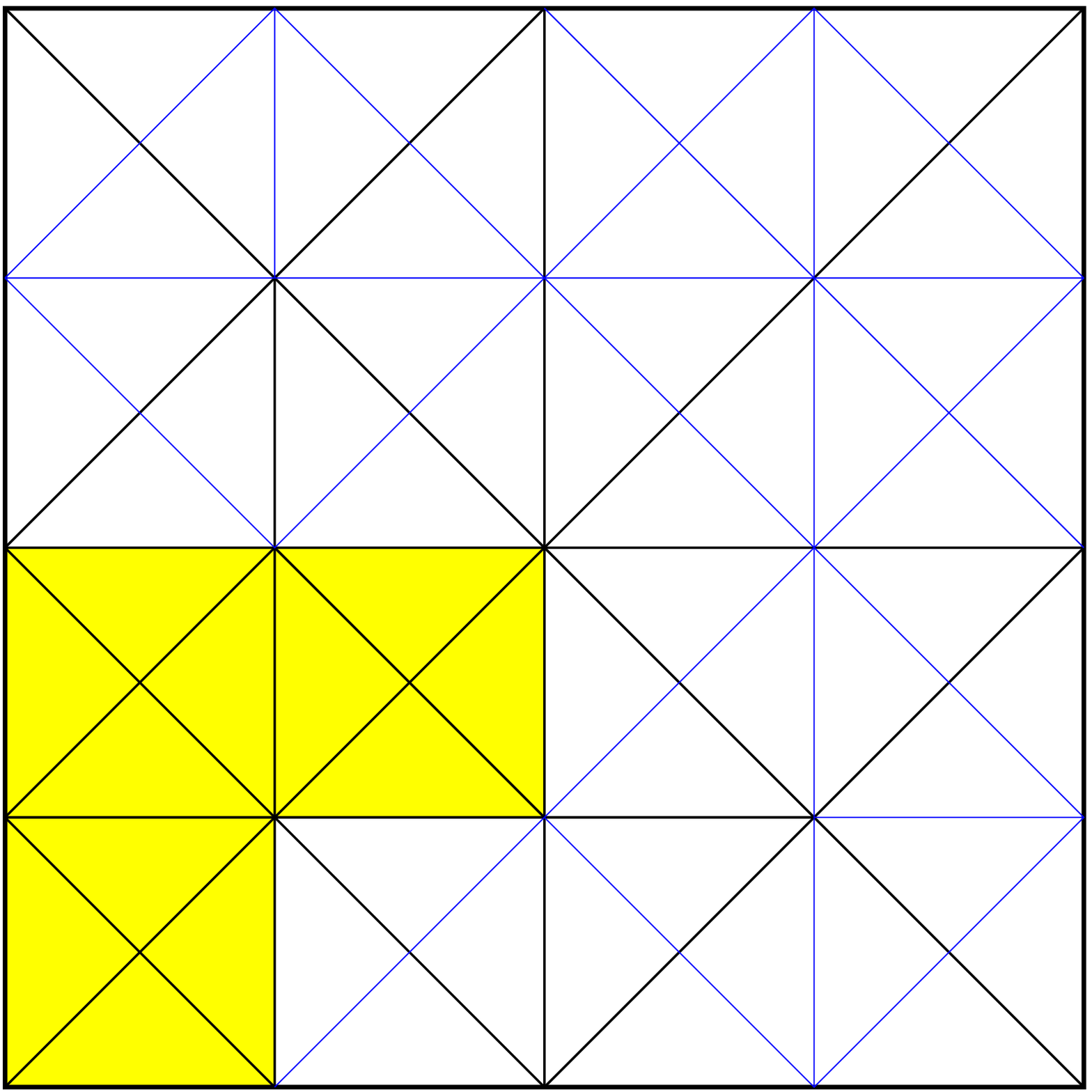}

    $\mesh_{3}$
  \end{minipage}%

  \begin{minipage}[c]{0.33\textwidth}\centering
    \includegraphics[width=0.95\textwidth]{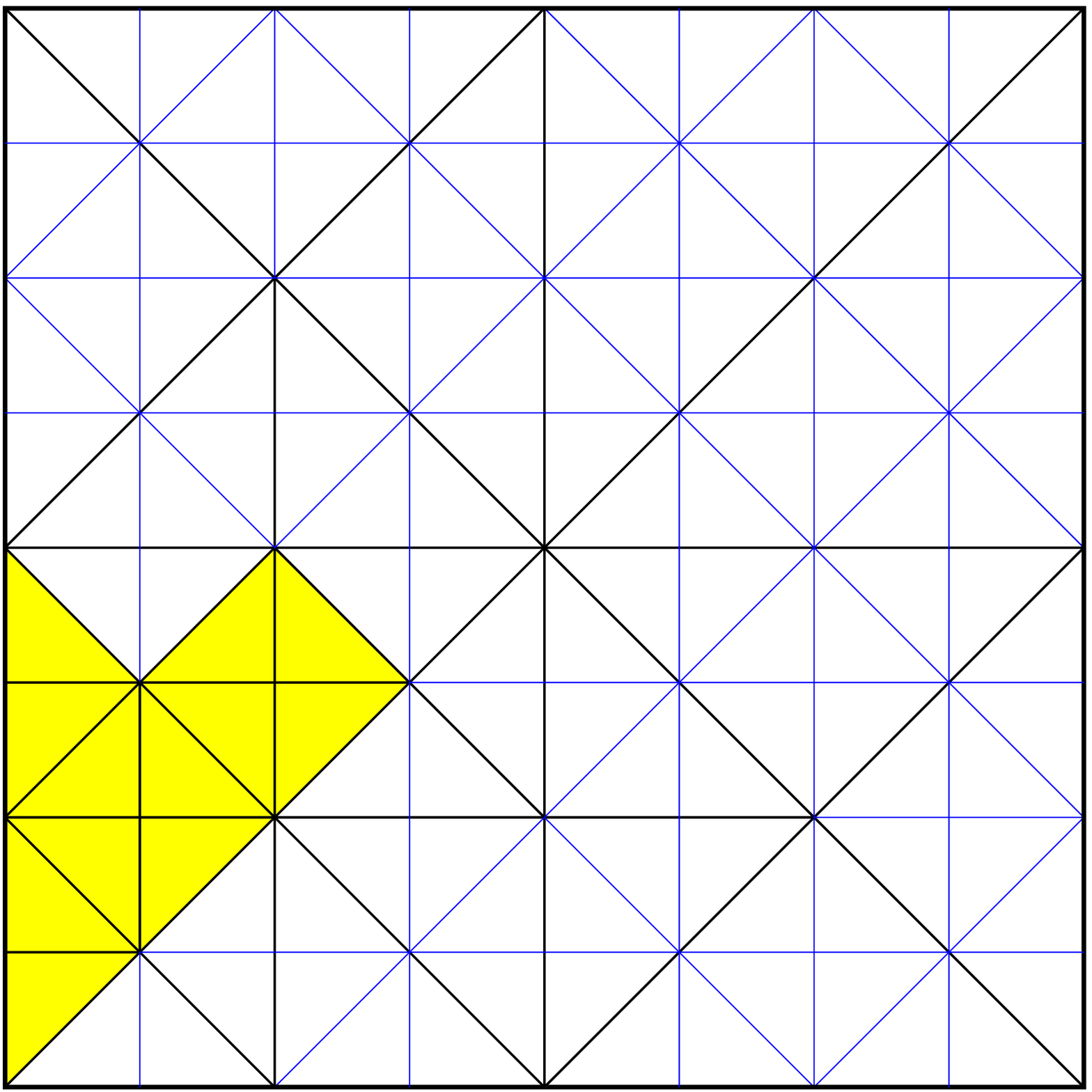}

    $\mesh_{4}$
  \end{minipage}%
  \begin{minipage}[c]{0.33\textwidth}\centering
    \includegraphics[width=0.95\textwidth]{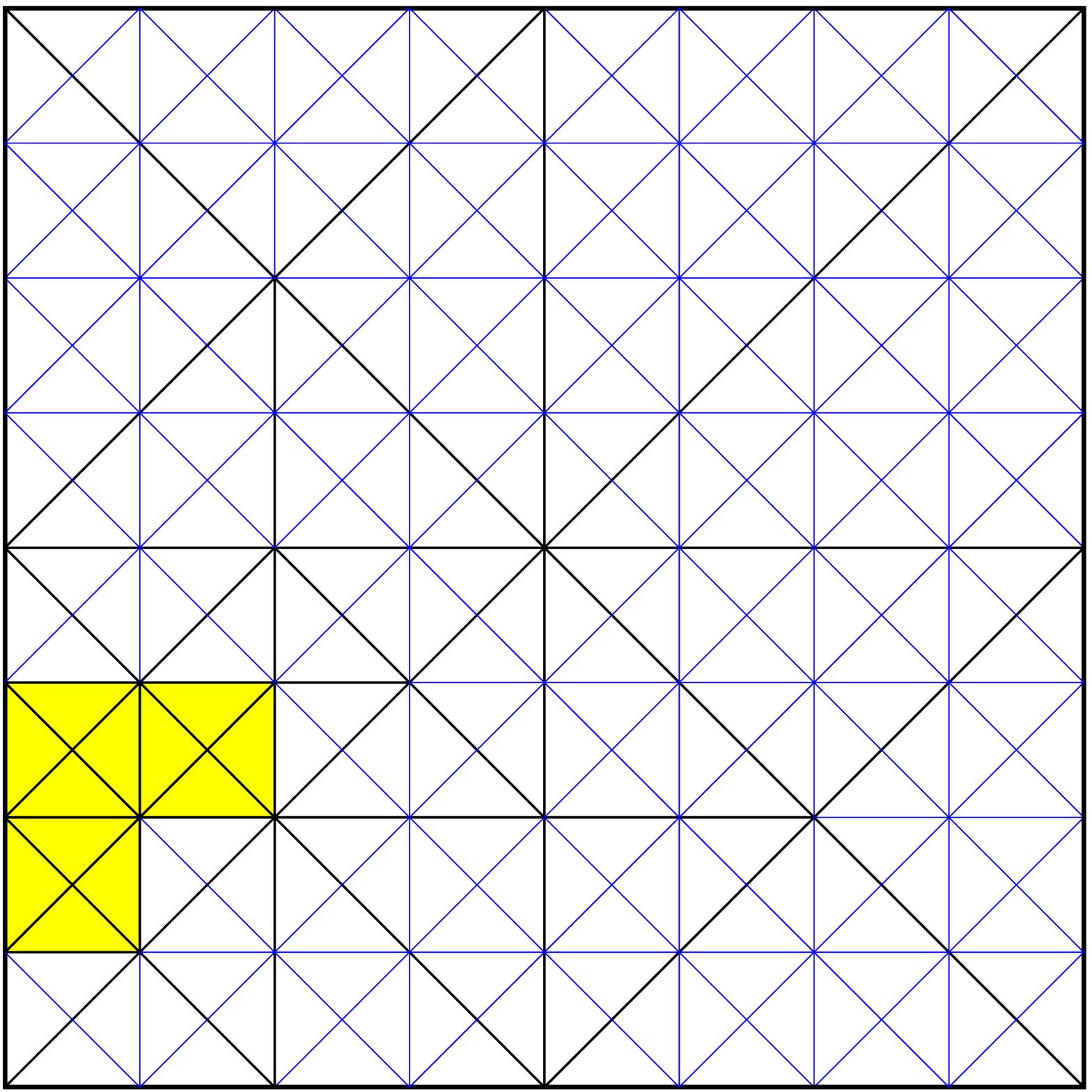}

    $\mesh_{5}$
  \end{minipage}%
  \begin{minipage}[c]{0.33\textwidth}\centering
    \includegraphics[width=0.95\textwidth]{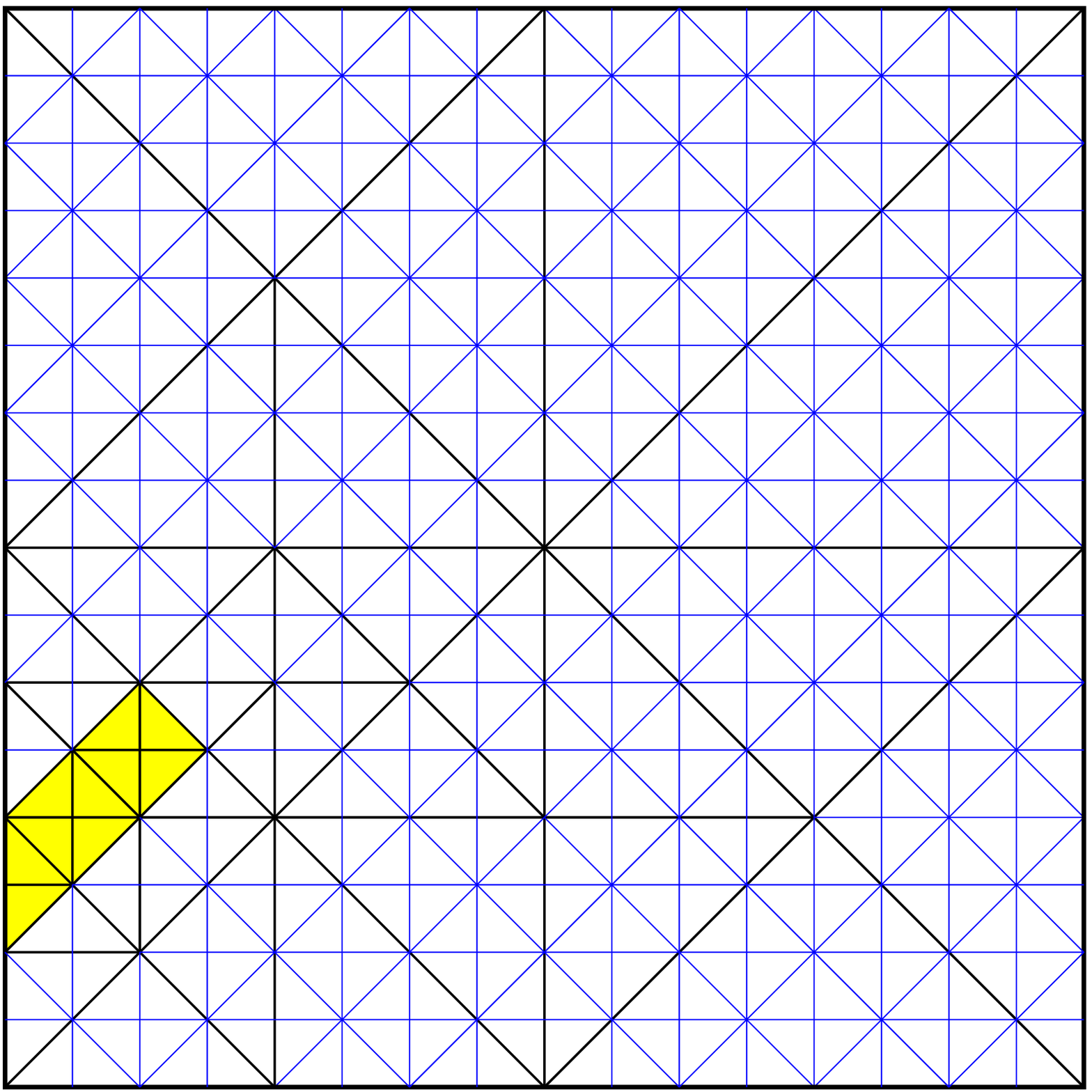}

    $\mesh_{6}$
  \end{minipage}%

  \caption{Virtual refinement hierarchy for 2D triangular meshes emerging in the
    course of successive local newest vertex bisection refinement of $\mesh_{0}$ from
    Fig.~\ref{fig:2dref}. Accompanying quasi-uniform
  meshes outlined in blue, maximally refined triangles marked yellow.}
  \label{fig:nvb}
\end{figure}

In order to keep the mesh conforming during refinements, the bisection of an edge is
only allowed when such an edge is the refinement edge for all elements which share
this edge. If a tetrahedron has to be refined, we have to loop around its refinement
edge and collect all elements at this edge to create an refinement patch. Then this
patch is refined by bisecting the common refinement edge. A more detailed discussion
can be found in \rh{\cite{KOS94a}}.

For any mesh $\mesh_l$ an associated ``quasi-uniform'' mesh $\wh{\mesh}_l$ according
to \eqref{eq:ref24}, $\mesh_l\prec\wh{\mesh}_l$, is obtained as follows: the elements
in $\{K\in\mesh_l:\;\lev(K)<l\}$ undergo bisection until $\lev(K)=l$ for any
$K\in\wh{\mesh}_l$.

We still have to make sure that the recursive bisection allows the definition of a
virtual refinement hierarchy. Thus, let $\mesh_h=\mesh_L$ be generated from the
initial mesh $\mesh_0$ by the bisection algorithm in \cite{KOS94a}. Denote by
$\mesh_{{\mathrm{hier}}}$ the set of all tetrahedra created during the bisection
process, i.e., for any $K\in\mesh_{\mathrm{hier}}$, there is a $K'\in\mesh_h$ such
that either $K'=K$ or $K'$ is created by refining $K$. \rh{Then, the virtual meshes
$\mesh_l$, $0<l<L$ can again be defined according to \eqref{eq:virt-mesh}.}

In the following, we are going to prove that each $\mesh_l$ is a
conforming mesh, that is, no hanging nodes occur in $\mesh_l$, $0\le l\le
L$. The proof depends on some mild assumptions on $\mesh_0$
(see assumptions (A1) and (A2) in \cite{KOS94a}), which will
be taken for granted.

\begin{lemma}\cite[Lemmas~2,3]{KOS94a}\label{lem:koss2}
  Let $T,T'\in \mesh_h$ be a pair of tetrahedra sharing a face $F=K\cap K'$.
  It holds true that
  \begin{enumerate}
  \item if $T$ contains the refinement edge of $T'$ and vice versa, then they have
    the same refinement edge,
  \item if $F$ contains the refinement edges of both $K$ and $K'$,
    then $\lev(K)=\lev(K')$,
  \item if $F$ contains the refinement edge of $K$, but does not
    contain the refinement edge of $K'$, then $\lev(K)=\lev(K')+1$,
  \item if $F$ does not contain the refinement edges of $K$ and
    $K'$, then $\lev(K)=\lev(K')$.
  \end{enumerate}
\end{lemma}

\begin{lemma}
  The meshes $\mesh_l$, $0\le l\le L$, according to \eqref{eq:virt-mesh} are
  conforming meshes.
\end{lemma}

\begin{figure}[htbp]
  \begin{center}
    \includegraphics[width=4in]{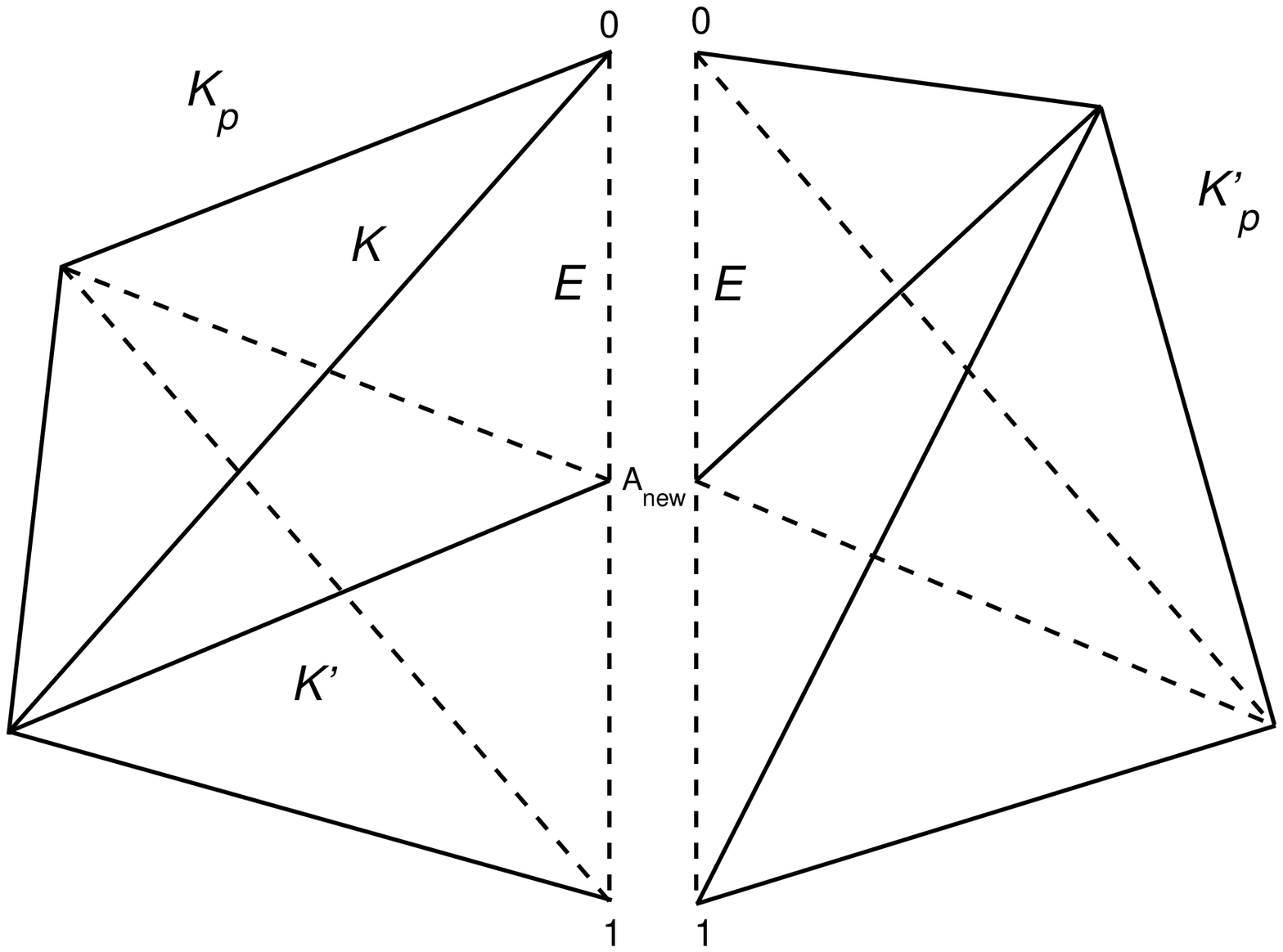}
    \caption{The patch around a refinement edge $E$ with vertex 0 and
      1. $\lev(K)=\lev(K')=L$ and
      $\lev(K_p)=\lev(K_p')=L-1$.}\label{fig:patch}
  \end{center}
\end{figure}

\begin{proof}
  We are going to prove the lemma by backward induction starting from $l=L$. Since
  $\mesh_L=\mesh_h$ is conforming, for any $K\in\mesh_L$ satisfying
  $\lev(K)=L$, there exists a brother of $K$, denoted by
  $K'\in\mesh_L$, such that $\lev(K')=L$ and $K_p:=K\cup K'\in
  \mesh_{L-1}$. Here $K_p$ is called the parent of $K$ and $K'$ with
  $\lev(K_p)=L-1$ (see Fig. \ref{fig:patch}).

  Let $E$ be the refinement edge of $K_p$. By the recursive bisection algorithm, $E$
  must be the common refinement edge of all tetrahedra in the refinement patch:
  $$P_E=\bigcup\{\,\ol{K_p'}:\;K_p'\in\mesh_{L-1} \hbox{ and } E\subset\ol{K_p'}\,\}.$$
  By Lemma \ref{lem:koss2}, $\lev(K_p')=L-1$ for any $K_p'\subset P_E$ and the midpoint
  of $E$, denoted by $A_{\mathrm{new}}$, is the unique new vertex of $\mesh_L$ in $P_E$. We
  conclude that
  $$P_E=\bigcup\{\,\ol{K}:\; K\in\mesh_L,\;\lev(K)=L,\hbox{ and } A_{\mathrm{new}}
  \hbox{ is a vertex of } K\,\}.$$ Coarsen the sub-mesh ${\mesh_L}_{|P_{E}}$ by
  removing the vertex $A_{\mathrm{new}}$ and all edges related to it and adding $E$ to this
  patch. Thus a conforming sub-mesh ${\mesh_{L-1}}_{|P_E}$ is obtained. Do the above
  coarsening process for every element $K\in\mesh_L$ with $\lev(K)=L$. This proves
  that $\mesh_{L-1}$ is conforming.

  Finally, an induction argument confirms that $\mesh_l$ is conforming, $l=L-2,\cdots,1$.
\end{proof}

\section{Local multigrid}
\label{sec:local-multigrid}

To begin with, we introduce nested refinement zones as open subsets of $\Omega$:
\begin{gather}
  \label{eq:3}
  \omega_{l} := \hbox{interior}\Big(
  \bigcup\{\overline{K}:\,K\in\mesh_{h},\, \lev(K)\ge l\}\Big)\subset {\Omega}\;,
\end{gather}
see Fig.~\ref{fig:2drefzone} and Fig.~\ref{fig:nvbrz}. The notion of refinement zones allows a concise
definition of the local multilevel decompositions of the finite element spaces
$\LFE(\mesh_{h})$ and $\EFE(\mesh_{h})$ that underly the local multigrid method.

\begin{figure}[!htb]
  \centering
  \begin{minipage}[c]{0.6\textwidth}
    \includegraphics[width=0.9\textwidth]{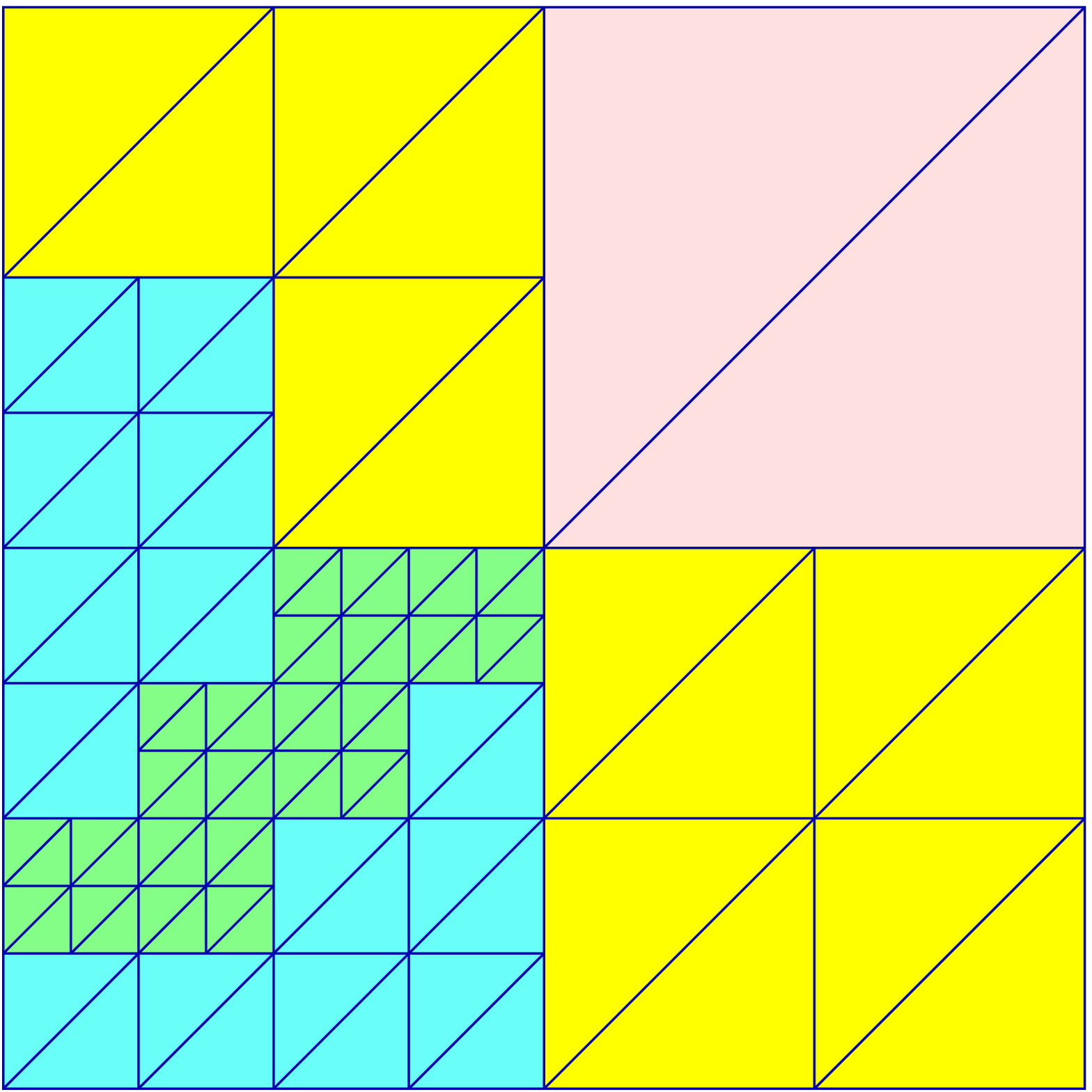}
  \end{minipage}%
  \begin{minipage}[c]{0.4\textwidth}
    ``Refinement strips'': set differences of refinement zones
    \vspace{2ex}

    \begin{tabular}[c]{cl}
      \includegraphics[width=2em]{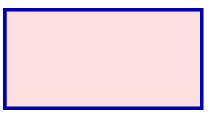} : & $\Sigma_{0}:=\omega_0\setminus\omega_{1}$ \\[2ex]
      \includegraphics[width=2em]{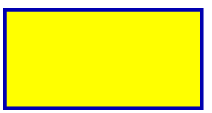} : & $\Sigma_{1}:=\omega_{1}\setminus\omega_{2}$ \\[2ex]
      \includegraphics[width=2em]{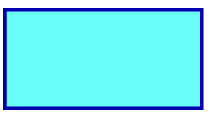} : & $\Sigma_{2}:=\omega_{2}\setminus\omega_{3}$ \\[2ex]
      \includegraphics[width=2em]{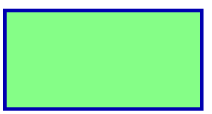} : & $\Sigma_{3}:=\omega_{3}$
    \end{tabular}
  \end{minipage}
  \caption{Refinement zones for the 2D refinement hierarchy of Figure~\ref{fig:2dref}.}
  \label{fig:2drefzone}
\end{figure}

\begin{figure}[!htb]
  \centering
  \begin{minipage}[c]{0.6\textwidth}
    \includegraphics[width=0.9\textwidth]{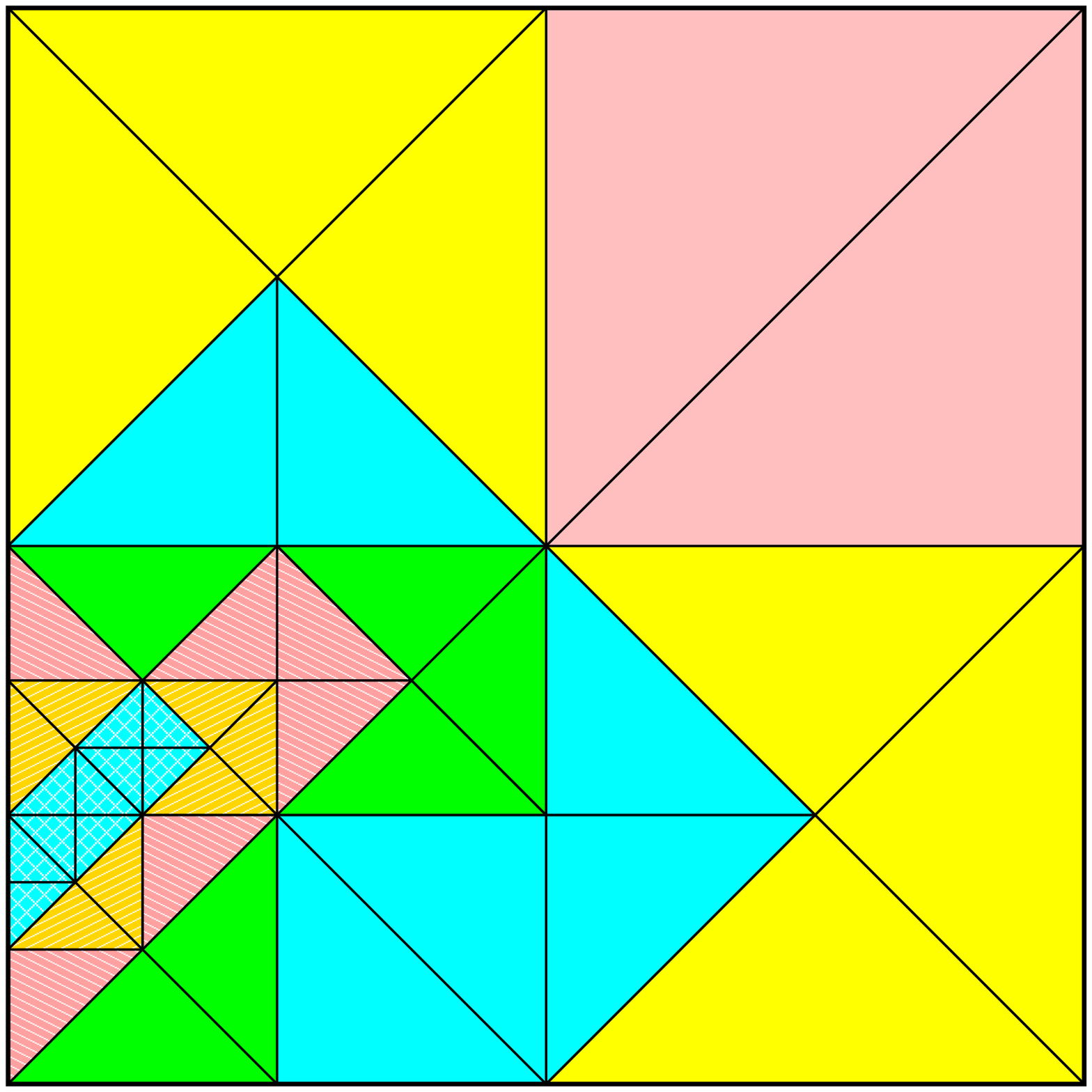}
  \end{minipage}%
  \begin{minipage}[c]{0.4\textwidth}
    ``Refinement strips'': set differences of refinement zones
    \vspace{2ex}

    \begin{tabular}[c]{cl}
      \includegraphics[width=2em]{zonesymb0} : & $\Sigma_{0}:=\omega_0\setminus\omega_{1}$ \\[2ex]
      \includegraphics[width=2em]{zonesymb1} : & $\Sigma_{1}:=\omega_{1}\setminus\omega_{2}$ \\[2ex]
      \includegraphics[width=2em]{zonesymb2} : & $\Sigma_{2}:=\omega_{2}\setminus\omega_{3}$ \\[2ex]
      \includegraphics[width=2em]{zonesymb3} : & $\Sigma_{3}:=\omega_{3}\setminus\omega_{4}$ \\[2ex]
      \includegraphics[width=2em]{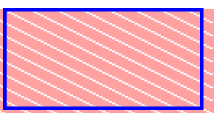} : & $\Sigma_{4}:=\omega_{4}\setminus\omega_{5}$ \\[2ex]
      \includegraphics[width=2em]{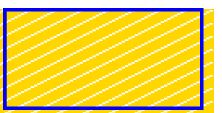} : & $\Sigma_{5}:=\omega_{5}\setminus\omega_{6}$ \\[2ex]
      \includegraphics[width=2em]{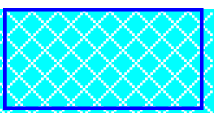} : & $\Sigma_{6}:=\omega_{6}$
    \end{tabular}
  \end{minipage}
  \caption{Refinement zones for the 2D refinement hierarchy of Figure~\ref{fig:nvb}.}
  \label{fig:nvbrz}
\end{figure}

We introduce local multigrid from the perspective of multilevel successive subspace
correction (SSC) \cite{JXU92,JXU97,XUZ00}. First, we give an abstract description for a linear
variational problem
\begin{gather}
  \label{LVP}
  u\in H:\quad \blf{a}(u,v) = f(v)\quad\forall v\in H\;,
\end{gather}
involving a positive definite bilinear form $\blf{a}$ on a Hilbert space $H$. The
method is completely defined after we have provided a finite subspace decomposition
\begin{gather}
  \label{ssc:dec}
  H = \sum\limits_{j=0}^{J}H_{j}\;,\quad H_{j}\subset H\;\text{closed subspaces},\;
  j=0,\ldots,J,\; J\in\bbN\;.
\end{gather}
Then the correction scheme implementation of one step of SSC acting on the iterate
$u^{m-1}$ reads:

\begin{itemize}
\item []\textsf{for} $m=1,2,\cdots$
\begin{itemize}
\item []$u_{-1}^{m-1}=u^{m-1}$

\item []\textsf{for} $j=0,1,\cdots,J$
\begin{itemize}
\item []\textsf{Let} $e_j\in H_j$ \textsf{solve}
$$\blf{a}(e_j,v_j)=f(v_j)-\blf{a}(u_{j-1}^{m-1},v_j)\quad \forall\,v_j\in H_j$$

\item [] $u_j^{m-1}=u^{m-1}_{j-1}+e_j$
\end{itemize}
\item []\textsf{endfor}

\item [] $u^m=u^{m-1}_J$
\end{itemize}
\item []\textsf{endfor}
\end{itemize}

This amounts to a stationary linear iterative method with error propagation operator
\begin{gather}
  \label{ssc:E}
  E=(I-P_J)(I-P_{J-1})\cdots(I-P_0)\;,
\end{gather}
where $P_{j}:H\mapsto H_{j}$ stands for the Galerkin projection
defined through
\begin{equation}
\blf{a}(P_jv,v_j)=\blf{a}(v,v_j)\quad\forall\,v_j\in
H_j.\label{eq:ssc-relax}
\end{equation}

The convergence theory of SSC for an inner product $\blf{a}$ and induced energy norm
$\N{\cdot}_{A}$ rests on two assumptions. The first one concerns the {\em stability
  of the space decomposition}. We assume that there exists a constant
$C_{\script{stab}}$ independent of $J$ such that
\begin{equation}
  \inf\Big\{\sum_{j=0}^J\N{v_j}_A^2:\; \sum_{j=0}^Jv_j=v\Big\} \le
  C_{\script{stab}}\N{v}_A^2 \quad \forall\,v\in H.
  \label{eq:C-stab}
\end{equation}
The second assumption is a {\em strengthened Cauchy-Schwartz
inequality}, namely, there exist two constants $0\le q< 1$ and
$C_{\script{orth}}$ independent of $j$ and $k$ such that
\begin{equation}
  \blf{a}(v_j,v_k)\le C_{\script{orth}}q^{|k-j|}\N{v_j}_A\N{v_k}_A\quad
  \forall\,v_j\in H_j,\; v_k\in H_k\;.
  \label{eq:C-orth}
\end{equation}
The above inequality states a kind of quasi-orthogonality between the subspaces. From
\cite[Theorem~4.4]{JXU92} and \cite[Theorem~5.1]{YSE93} we cite the following
central convergence theorem:

\begin{theorem}
  \label{thm:41}
  Provided that \eqref{eq:C-stab} and \eqref{eq:C-orth} hold,
  the convergence rate of Algorithm SSC is bounded by
  \begin{equation}
    \N{E}_A^2\le 1-\frac{1}{C_{\script{stab}}(1+\Theta)^2} \quad
    \hbox{with}\quad
    \Theta=C_{\script{orth}}\frac{1+q}{1-q},\label{eq:C-rat}
  \end{equation}
  where the operator norm is defined by
  $$\N{E}_A:=\sup_{v\in H,v\neq 0}\frac{\N{Ev}_{A}}{\N{v}_A}.$$
\end{theorem}

The bottom line is that the subspace splitting \eqref{ssc:dec} already provides a
full description of the method. Showing that both constants $C_{\script{stab}}$ from
\eqref{eq:C-stab} and $C_{\script{orth}}$ from \eqref{eq:C-orth} can be chosen
independently of the number $L$ of refinement levels is the challenge in asymptotic
multigrid analysis.

In concrete terms, the role of the linear variational problem \eqref{LVP} is played
by \eqref{eq:VPcurl} considered on the edge element space $\EFE(\mesh_{h})$, which
replaces the Hilbert space $H$.  To define the local multilevel decomposition of
$\EFE(\mesh_{h})$, we define ``sets of new basis functions'' on the various
refinement levels
\begin{gather}
  \label{eq:NB}
  \begin{aligned}
    & \Bas_{\LFE}^0 :=\Bas_{\LFE}(\mesh_0),\quad
    \Bas_{\LFE}^{l} := \{b_{h}\in\Bas_{\LFE}(\mesh_{l}):\;\supp
    b_{h}\subset\ol{\omega}_l\}\;,\\
    & \Bas_{\EFE}^0 :=\Bas_{\EFE}(\mesh_0),\quad
    \Bas_{\EFE}^{l} := \{\Vb_{h}\in\Bas_{\EFE}(\mesh_{l}):\;\supp
    \Vb_{h}\subset\ol{\omega}_l\}\;,
  \end{aligned} \quad
  1\leq l \leq L\;.
\end{gather}
A 2D drawing of the sets $\Bas_{\LFE}^{l}$ is given in
Fig.~\ref{fig:hnactvert} where $\Gamma_D=\partial\Omega$. Note
that we also have to deal with $\LFE(\mesh_{h})$, because, as
suggested by the reasoning in \cite{HIP99}, a local multilevel
decomposition of $\EFE(\mesh_{h})$ has to incorporate an
appropriate local multilevel decomposition of $\LFE(\mesh_{h})$.

\begin{figure}[!htb]
 \centering
 \begin{minipage}[c]{0.25\textwidth}\centering
   \includegraphics[width=0.95\textwidth]{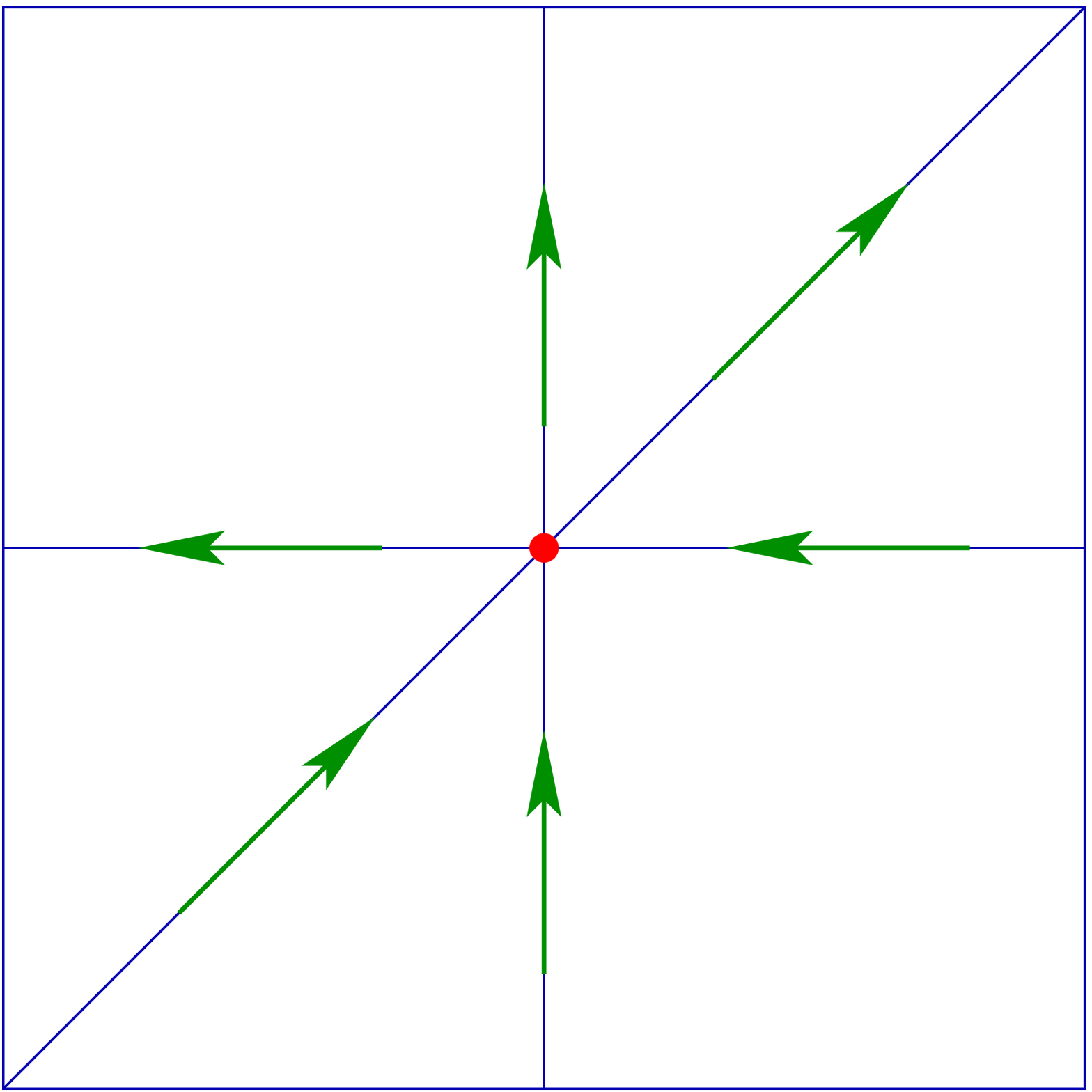}

   $l=0$
 \end{minipage}%
 \begin{minipage}[c]{0.25\textwidth}\centering
   \includegraphics[width=0.95\textwidth]{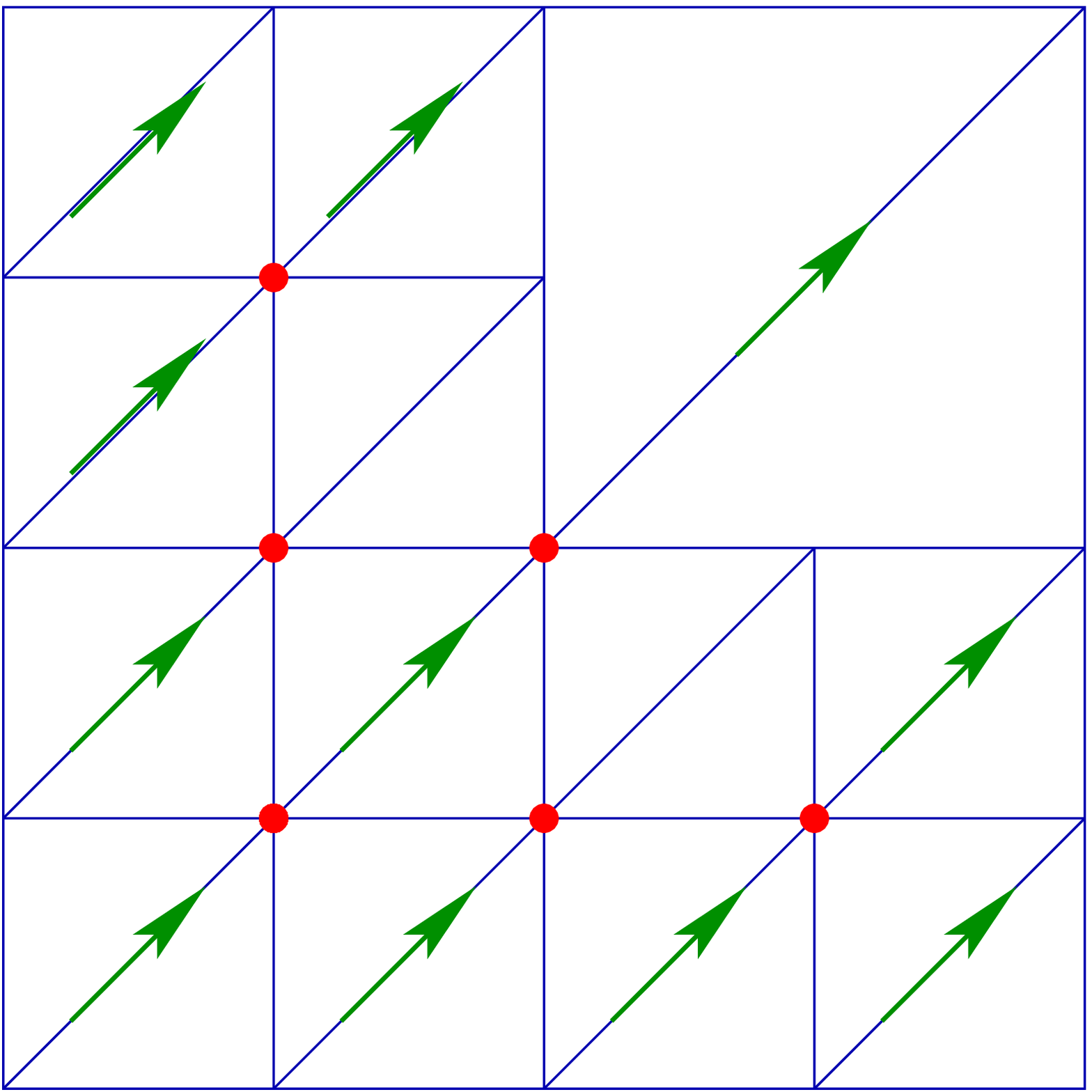}

   $l=1$
 \end{minipage}%
 \begin{minipage}[c]{0.25\textwidth}\centering
   \includegraphics[width=0.95\textwidth]{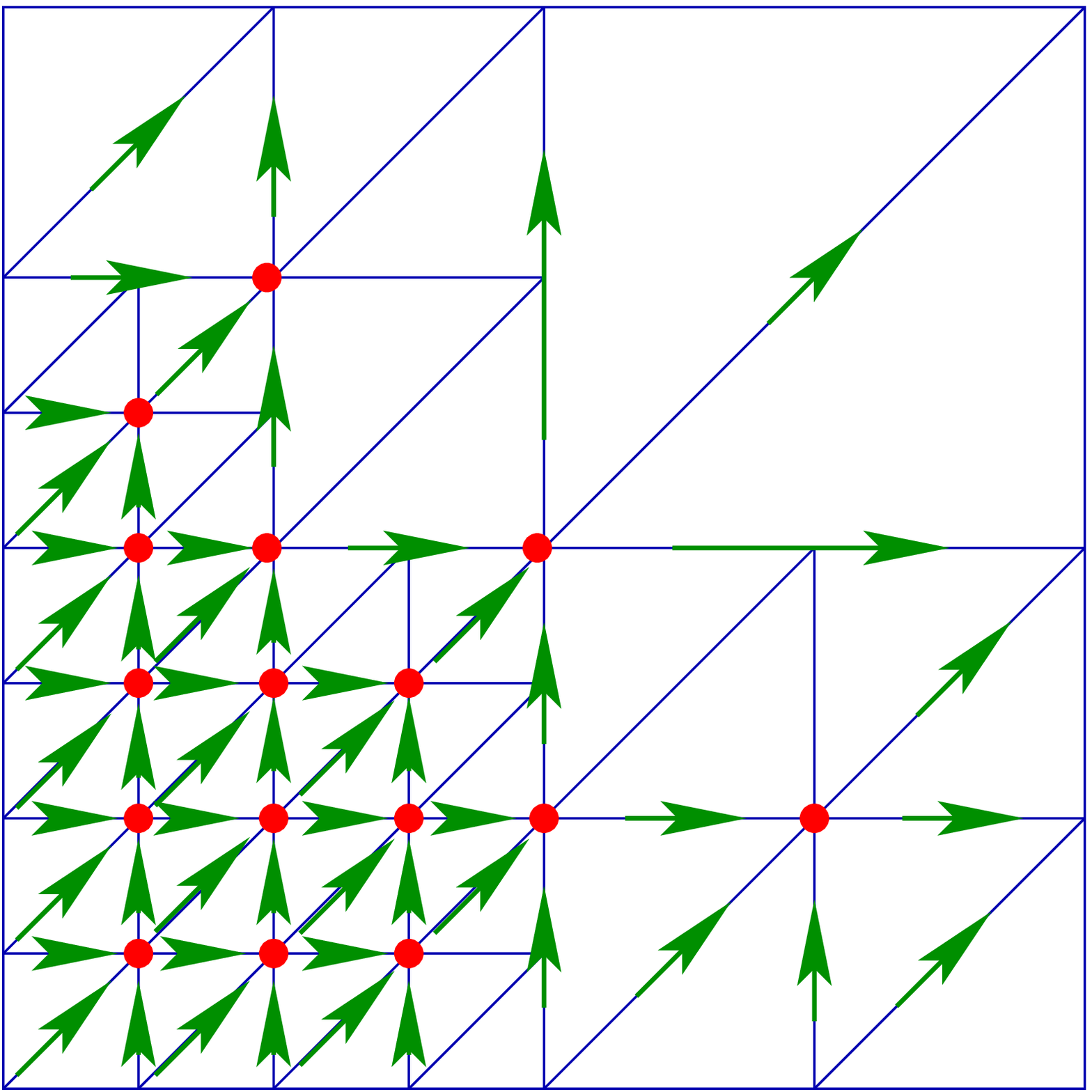}

   $l=2$
 \end{minipage}%
 \begin{minipage}[c]{0.25\textwidth}\centering
   \includegraphics[width=0.95\textwidth]{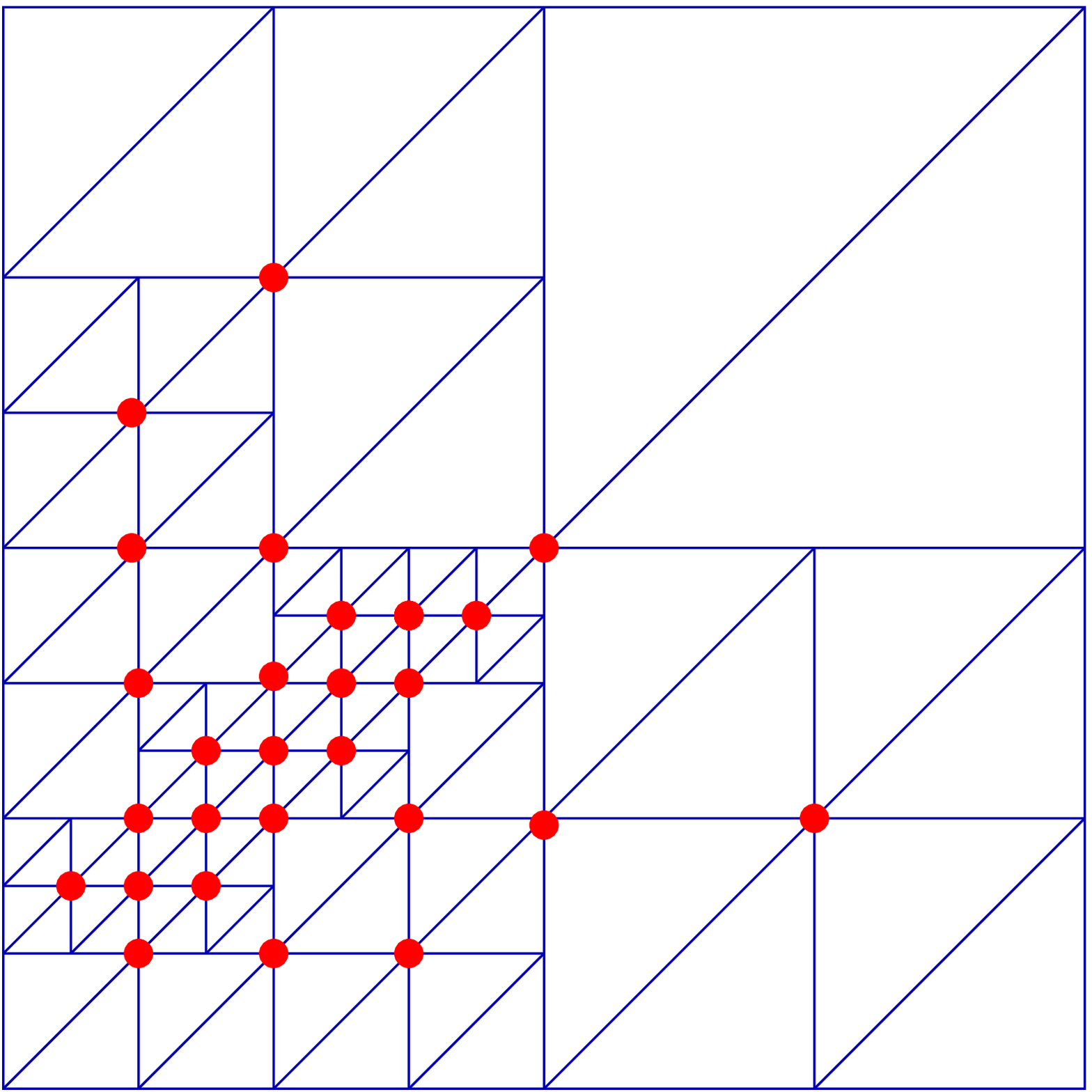}

   $l=3$
 \end{minipage}%
 \caption{Active vertices ({\color{red}red}) carrying ``tent functions'' in $\Bas_{\LFE}^{l}$,
 $\DBd=\partial\Omega$, refinement hierarchy of Fig.~\ref{fig:2dref}}
 \label{fig:hnactvert}
\end{figure}

Then, a possible local multigrid iteration for the linear system of equations arising
from a finite element Galerkin discretization of a $\bHone{\DBd}$-elliptic
variational problem boils down to a successive subspace correction method based on
the local multilevel decomposition
\begin{gather}
  \label{eq:LMGgrad}
  \LFE(\mesh_{h}) = \LFE(\mesh_{0}) +
  \sum\limits_{l=1}^{L} \sum\limits_{b_{h}\in\Bas_{\LFE}^{l} }\Span{b_{h}}\;.
\end{gather}

Similarly, the local multilevel splitting of $\EFE(\mesh_{h})$ is based on the
multilevel decomposition
\begin{gather}
  \label{eq:LMGcurl}
  \EFE(\mesh_{h}) = \EFE(\mesh_{0}) +
  \sum\limits_{l=1}^{L}
  \sum\limits_{b_{h}\in\Bas_{\LFE}^{l} }\Span{\grad b_{h}}
  +\sum\limits_{l=1}^{L}\sum\limits_{\Vb_{h}\in\Bas_{\EFE}^{l} }\Span{\Vb_{h}} \;.
\end{gather}
\rh{These splittings induce SSC iterations that can be implemented as non-symmetric
multigrid V-cycles with only one (hybrid) Gauss-Seidel post-smoothing step, see
\cite[Sect.~6]{HIP99}.  Duplicating components of \eqref{eq:LMGcurl} results in more
general multigrid cycles with various numbers of pre- and post-smoothing steps.}

The splitting \eqref{eq:LMGcurl} is motivated both by the design of multigrid methods
for \eqref{eq:VPcurl} and $\EFE(\mesh)$ in the case of uniform refinement and local
multigrid approaches to $\bHone{\DBd}$-elliptic variational problems after
discretization by means of linear finite elements \cite{MIT92,WUC05}. The occurrence
of gradients of ``tent functions'' $b_{h}$ in \eqref{eq:LMGcurl} is related to the
\emph{hybrid local relaxation}, which is essential for the performance of multigrid
in $\Hcurl$, see \cite{HIP99} for a rationale. A rigorous justification will emerge
during the theoretical analysis in the following sections. It will establish the
following main theorem.

\begin{theorem}[Asymptotic convergence of local multigrid for edge elements]
  \label{thm:main}
  Under the assumptions on the meshes made above and allowing at most one hanging
  node per edge, the decomposition \eqref{eq:LMGcurl} leads to an SSC iteration whose
  convergence rate is bounded away from $1$ uniformly in the number $L$ of refinement
  steps.
\end{theorem}


\section{Stability}
\label{sec:stability}

First we tackle the stability estimate \eqref{eq:C-stab} for the
local multilevel decomposition \eqref{eq:LMGgrad}, which is
implicitly contained in \eqref{eq:LMGcurl}.

\subsection{Local quasi-interpolation onto $\LFE(\mesh)$}
\label{sec:local-quasi-interp}

Quasi-interpolation operators are projectors onto finite element spaces that have
been devised to accommodate two conflicting goals: locality and boundedness in weak
norms \cite{CLE75,OSW94,SCZ90,SHO01,CHW07t}. \rh{As key tool they will be used in
  Sect.~\ref{sec:glob-mult-splitt} and the proof of Lemma~\ref{lem:disc-hem}. As in
  \cite[Sect.~2.1.1]{OSW94},} we resort to a construction employing local linear
$L^{2}$-dual basis functions.  \rh{We follow the analysis of \cite{SHO01} that
  permits us to take into account Dirichlet boundary conditions.}

For a generic tetrahedron $K$ define $\psi^{K}_{j}$, $j=1,2,3,4$,
by $L^{2}(K)$-duality to the barycentric coordinate functions
$\lambda_{i}$, $i=1,2,3,4$, of $K$:
\begin{gather}
  \label{eq:s2}
  \psi^{K}_{j}\in\bbP_{1}(K):\quad
  \int\nolimits_{K}\psi^{K}_{j}(\Bx)\lambda_{i}(\Bx)\,\mathrm{d}\Bx =
  \delta_{ij}\;,\quad i,j\in\{1,\ldots,4\}\;.
\end{gather}
Computing an explicit representation of the $\psi^{K}_{j}$ we
find
\begin{gather}
  \label{eq:s4}
  C^{-1} \leq |K|\NLtwo[K]{\psi^{K}_{j}}^{2} \leq C\quad,\quad
  C^{-1} \leq \NLp[K]{\psi^{K}_{j}}{1} \leq C\;,
\end{gather}
with an absolute constant $C>0$. We can regard $\psi^{K}_{j}$
as belonging to the $j$-th vertex of $K$. Thus, we will also
write $\psi^{K}_{\Bp}$, $\Bp\in \Cn(K)$, $\Cn(K)$ the set of
vetices of $K$.

\rh{Let $\mesh$ be one of the tetrahedral meshes $\mesh_{l}$ or $\wh{\mesh}_{l}$} of
$\Omega$. In order to introduce quasi-interpolation operators we take for granted
some ``node$\to$cell''--assignment, a mapping $\ol{\Cn}(\mesh)\mapsto \mesh$,
$\Bp\in\ol{\Cn}(\mesh)\mapsto K_{\Bp}\in\mesh$.

\begin{definition}
  \label{def:QIP}
  Writing ${\{b_{\Bp}\}}_{\Bp\in\Cn({\mesh})} :=
  \Bas_{\LFE}({\mesh})$, define the local quasi-interpolation operator
  \begin{equation}\label{eq:QIP}
    \QIOp :
    \left\{
      \begin{array}{ccl}
        \DS \Ltwo &\mapsto & \LFE({\mesh}) \\
        \DS u & \mapsto & 
        \sum_{\Bp\in\mcal{N}({\mesh})}\int_{K_{\Bp}}\psi^{K_{{\Bp}}}_{\Bp}(\Bx)
        u (\Bx)\,d\Bx\cdot b_{{\Bp}}\;.
      \end{array}
    \right.
  \end{equation}
  Analoguously, we introduce the local quasi-interpolation
  $\ol{\QIP}_{h}:\Ltwo\mapsto\ol{\LFE}(\mesh)$.
\end{definition}\vspace{2mm}

We point out that $\QIOp $ respects $u=0$ on $\DBd$, because the sum does not cover
basis functions attached to vertices on $\DBd$. From \eqref{eq:s2} it is also evident
that both $\QIOp $ and $\ol{\QIP}_{h}$ are projections, for instance,
\begin{equation}
  \QIOp  u_{h}= u_{h} \quad \forall\, u_{h}\in \LFE({\mesh})\;.
  \label{eq:Qproj}
\end{equation}%
Moreover, they satisfy the following strong continuity and
approximation properties:\vspace{2mm}

\begin{lemma}\label{lem:Qint}
  The quasi-interpolation operators from Def.~\ref{def:QIP} allow the estimates
  (set $\DBd=\emptyset$ for $\ol{\QIP}_{\mesh}$)
  \begin{align}
    \label{eq:L2stab}
    & \exists C=C(\rho_{{\mesh}}):\qquad\qquad\;\;\;
    \NLtwo{\QIOp  u}\le C\NLtwo{u}\;\;\,\forall\,u\in \Ltwo\;,\\
    \label{eq:H1stab}
    & \exists C=C(\rho_{{\mesh}},\Omega,\DBd):\quad\quad\;
    \SNHone{\QIOp  u}\le C\SNHone{u}\quad\;\forall\,u \in \bHone{\DBd}\;,\\
    \label{eq:int-err}
    & \exists C=C(\rho_{{\mesh}},k):\;
    \NLtwo{\mwf^{-k}(u-\QIOp u)}\le C \SNHm{u}{k}\quad
    \forall\, u\in \Hm{k}\cap \bHone{\DBd}\;,
  \end{align}
  and $k=1,2$.
\end{lemma}\vspace{2mm}

\begin{proof}[Part I]
  Continuity in $\Ltwo$ is a simple consequence of the stability \eqref{eq:fem18}
  of the nodal bases $\Bas_{\LFE}({\mesh})$ and of the Cauchy-Schwarz inequality:
  \begin{eqnarray*}
    \NLtwo{\QIOp u}^{2} &\leq&
    C \sum\limits_{\Bp\in\Cn({\mesh})}
    |\QIOp u(\Bp)|^{2}\NLtwo{b_{\Bp}}^{2} \\ &=&
    C \hspace*{-1ex}\sum\limits_{\Bp\in\Cn({\mesh})}
    \left|\int\nolimits_{K_{\Bp}}\psi^{K_{{\Bp}}}_{\Bp}(\Bx)
    u (\Bx)\,d\Bx\right|^{2}\NLtwo{b_{\Bp}}^{2} \\
  &\leq& C \sum\limits_{\Bp\in\Cn({\mesh})}
  \NLtwo[K_{\Bp}]{\psi^{K_{\Bp}}_{\Bp}}^{2}\NLtwo{b_{\Bp}}^{2}
  \NLtwo[K_{\Bp}]{u}^{2} \leq C \NLtwo{u}^{2}\;,
  \end{eqnarray*}
  with $C=C(\rho_{{\mesh}})>0$, because
  $\NLtwo[K_{\Bp}]{\psi^{K_{\Bp}}_{\Bp}}^{2}\NLtwo{b_{\Bp}}^{2} \leq C$, too.
\end{proof}

The following estimate is instrumental in establishing continuity of $\QIOp $
in $\bHone{\DBd}$:

\begin{theorem}[Generalized Hardy inquality]
  \label{thm:Hardy}
    \begin{gather*}
      \exists C=C(\Omega,\DBd)>0:\quad\int\limits_{\Omega}
      \left|\frac{u}{\operatorname{dist}(\Bx,\DBd)}\right|^{2}\,\mathrm{d}\Bx
      \leq C \SNHone{u}^{2}\quad\forall u\in\bHone{\DBd}\;.
    \end{gather*}
\end{theorem}

\begin{proof}
  By density it suffices to consider $u\in C^{\infty}(\ol{\Omega})$,
  $\operatorname{supp}(u)\cap\DBd=\emptyset$. Using a partition of unity, we can
  confine the estimate to neighborhoods of $\DBd$, in which $\partial\Omega$ is the
  graph of a Lipschitz-continuous function. Thus, after bi-Lipschitz transformations,
  we need only investigate three canonical situations, see Fig.~\ref{fig:hardy}:
  \begin{enumerate}
  \item $\DBd = \{z=0\}$, for which the 1D Hardy inequality gives the estimate,
    see the proof of Thm.~1.4.4.4 in \cite{GRI85}.
  \item $\DBd = \{z=0\;\wedge\; x>0\}$, which can be treated using polar coordinates
    in the $(x,z)$-plane and then integrating in $y$-direction:
    \begin{gather*}
      \mbox{$$}\hspace*{-2em}\int\limits_{0}^{\infty} \int\limits_{0}^{\pi}
      \left|\frac{u(r,\varphi)}{r}\right|^{2}\mathrm{d}\varphi\,r\,\mathrm{d}r
      \leq \int\limits_{0}^{\infty} \int\limits_{0}^{\pi}
      \left|\frac{\pi}{r}\frac{\partial u}{\partial\varphi}(r,\varphi)\right|^{2}
      \mathrm{d}\varphi\,r\,\mathrm{d}r \leq\pi^2
      \int\limits_{z>0}|\grad_{x,z}u|^2\,\mathrm{d}x\mathrm{d}z\;.
    \end{gather*}
  \item $\DBd= \{z=0\,\wedge x>0\,\wedge\,y>0\}$, for which we obtain a similar
    estimate using spherical coordinates.
  \end{enumerate}
  This ends the proof.
\end{proof}

  \begin{figure}[!htb]
    \centering

    \begin{minipage}[c]{0.33\textwidth}\centering
      \psfrag{x}{$x$}
      \psfrag{y}{$y$}
      \psfrag{z}{$z$}
      \psfrag{GD}{$\DBd$}
      \psfrag{G0}{$\partial\Omega\setminus\DBd$}
      \includegraphics[width=0.95\textwidth]{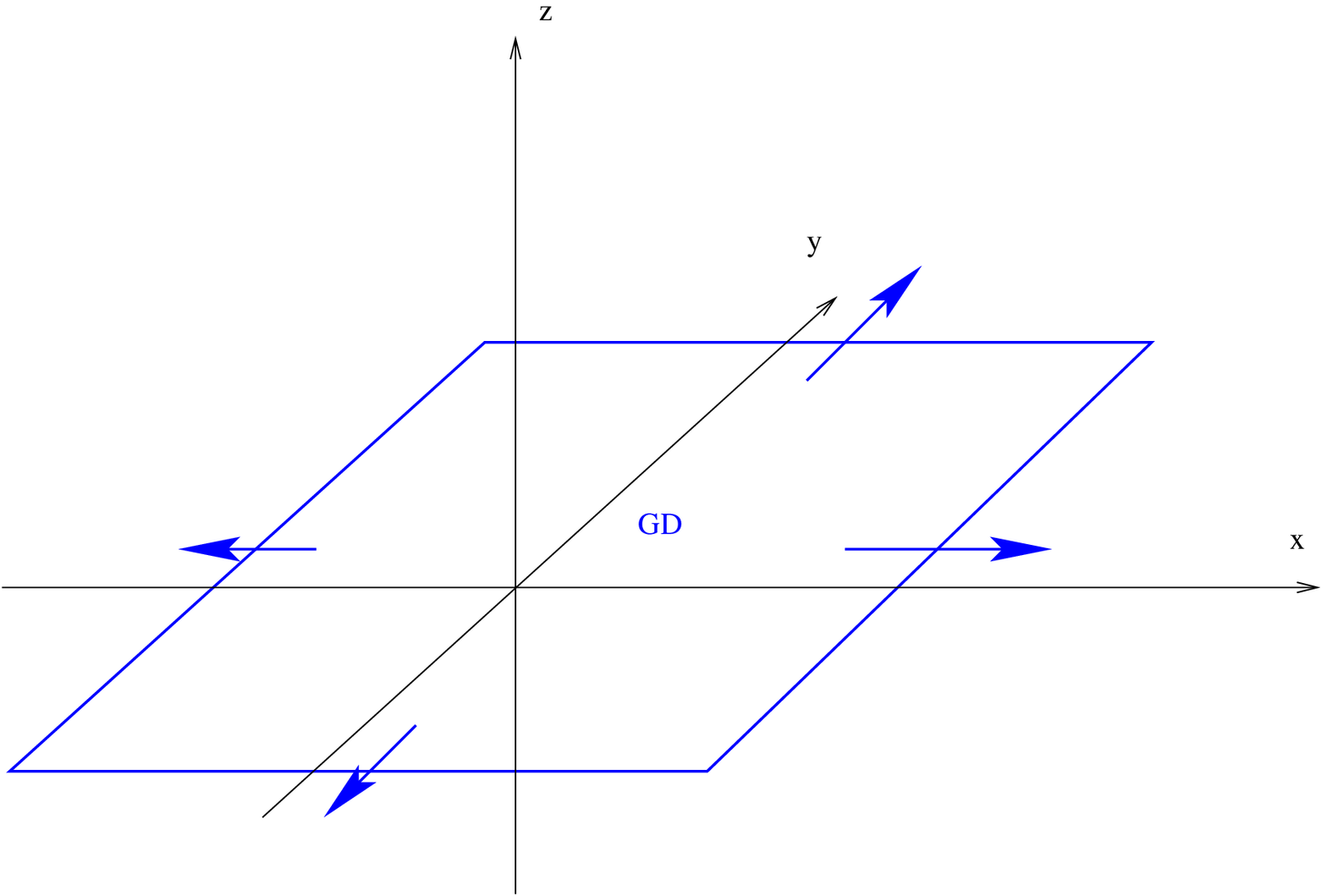}

      $\DBd = \{z=0\}$
    \end{minipage}%
    \begin{minipage}[c]{0.33\textwidth}\centering
      \psfrag{x}{$x$}
      \psfrag{y}{$y$}
      \psfrag{z}{$z$}
      \psfrag{GD}{$\DBd$}
      \psfrag{G0}[c][c][1][45]{$\partial\Omega\setminus\DBd$}
      \psfrag{phi}{$\varphi$}
      \psfrag{r}{$r$}
      \includegraphics[width=0.95\textwidth]{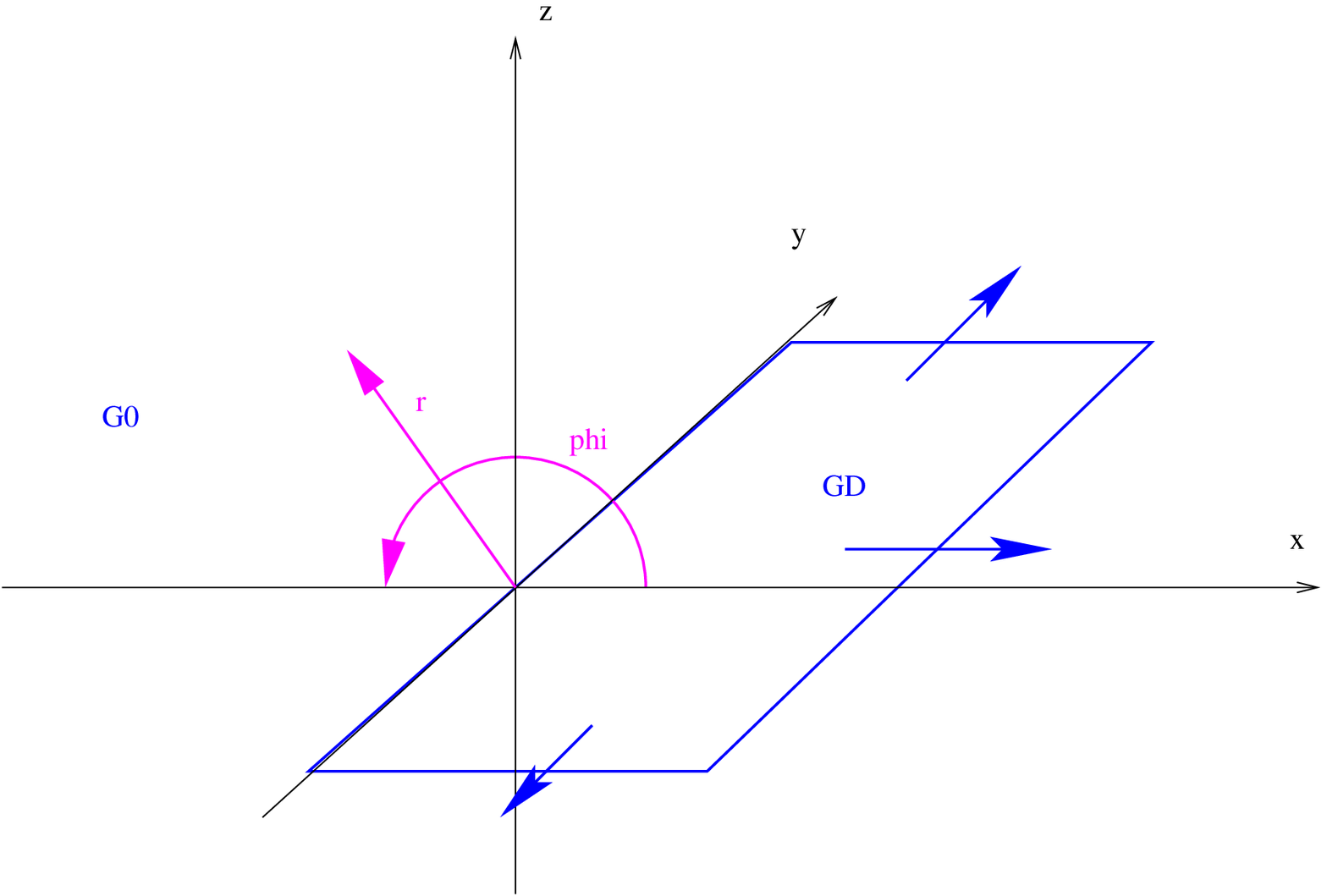}

      $\DBd = \{z=0\,\wedge\,x>0\}$
    \end{minipage}%
    \begin{minipage}[c]{0.33\textwidth}\centering
      \psfrag{x}{$x$}
      \psfrag{y}{$y$}
      \psfrag{z}{$z$}
      \psfrag{GD}{$\DBd$}
      \psfrag{G0}{$\partial\Omega\setminus\DBd$}
      \includegraphics[width=0.95\textwidth]{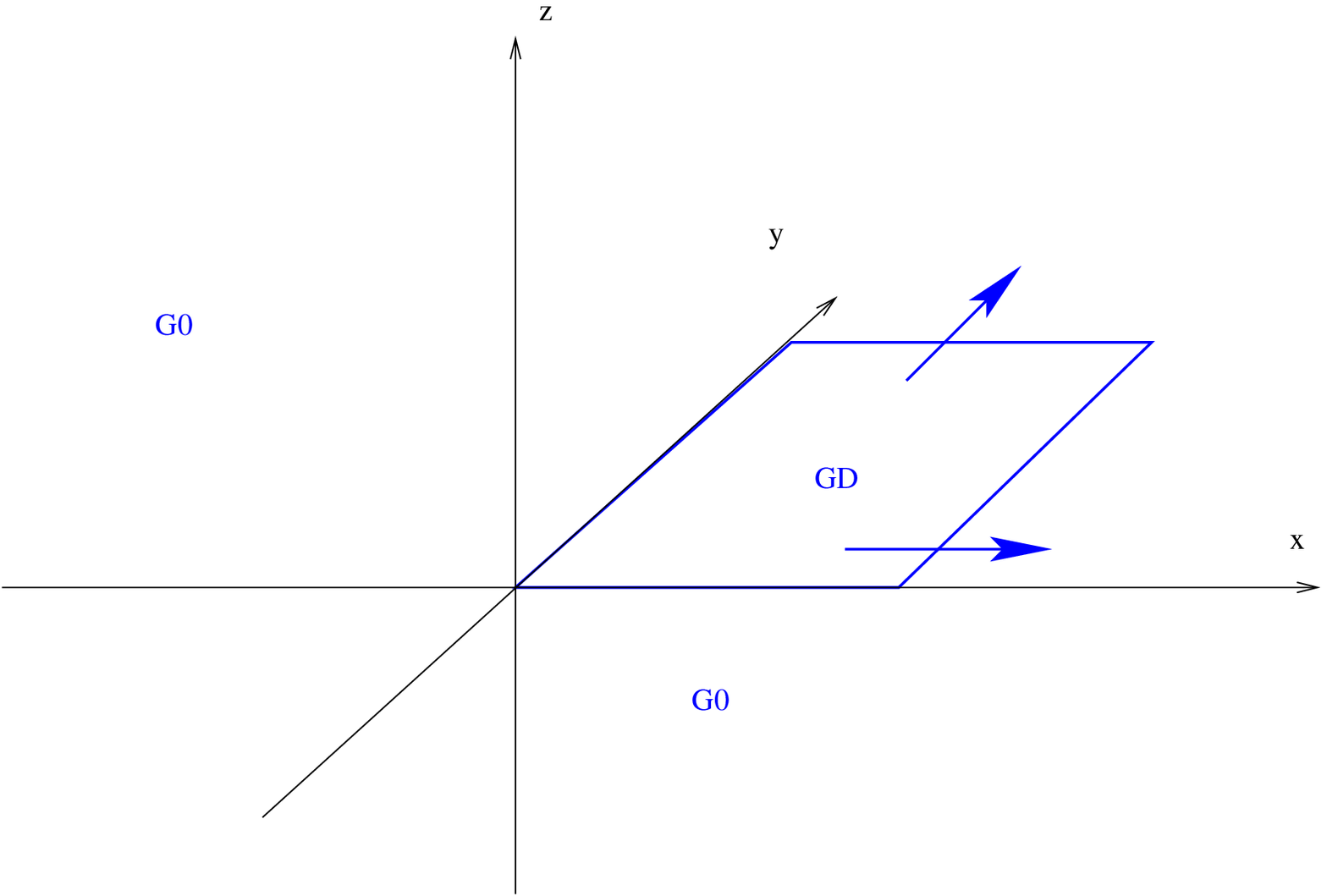}

      {\small $\DBd = \{z=0\,\wedge x>0\,\wedge\,y>0\}$}
    \end{minipage}%

    \caption{Canonical situations to be examined in the proof of Thm.~\ref{thm:Hardy}}
    \label{fig:hardy}
  \end{figure}

  \begin{proof}[of Lemma~\ref{lem:Qint}, part II]
    In order to tackle the $\Hone$-continuity of $\QIOp$, we use that $\grad
    \LFE(\mesh)\subset \EFE(\mesh)$ along with the stability estimate \eqref{eq:fem17}
    \begin{gather}
      \label{eq:s5}
      \NLtwo{\grad \QIOp u}^{2}\leq C \sum\limits_{E=[\Bp,\Bq]\in\Ce({\mesh})}
      (\QIOp u(\Bp) - \QIOp u(\Bq))^{2}\NLtwo{\Vb_{E}}^{2}\;,
    \end{gather}
    with the notation ${\{\Vb_{E}\}}_{E\in\Ce({\mesh})} :=
    \Bas_{\EFE}({\mesh})$.

    (i)\quad for the case {$E=[\Bp,\Bq]\in\Ce({\mesh})$},
    {$\Bp,\Bq\not\in\DBd$}, we adapt arguments from \cite{SHO01}. For any
    $u\in\bHone{\DBd}$, by \eqref{eq:s2}, we have the identity
    \begin{multline*}
      |(\QIOp  u)(\Bp)-(\QIOp  u)(\Bq)| =
      \Bigl|\int\limits_{K_{\Bp}}\int\limits_{K_{\Bq}}
      \psi_{\Bp}^{K_{\Bp}}(\Bx)\psi_{\Bq}^{K_{\Bq}}(\By)(u(\Bx)-u(\By))\,
      \mathrm{d}\By\mathrm{d}\Bx\Bigr| \\
      =  \Bigl|\int\limits_{K_{\Bp}}\int\limits_{K_{\Bq}}
      \psi_{\Bp}^{K_{\Bp}}(\Bx)\psi_{\Bq}^{K_{\Bq}}(\By)\,\int\limits_{0}^{1}
      \grad
      u(\By+\tau(\Bx-\By))\cdot(\Bx-\By)\,\mathrm{d}\tau\,\mathrm{d}\By\mathrm{d}\Bx
      \Bigr|\;.
    \end{multline*}
    Then split the innermost integral and transform
    \begin{gather*}
      \int\nolimits_{0}^{1} f(\By+\tau(\Bx-\By))\,\mathrm{d}\tau =
      \int\nolimits_{\hlb}^{1} f(\By+\tau(\Bx-\By))\,\mathrm{d}\tau +
      \int\nolimits_{\hlb}^{1} f(\Bx+\tau(\By-\Bx))\,\mathrm{d}\tau\;.
    \end{gather*}
    We infer
    \begin{multline*}
      |(\QIOp  u)(\Bp)-(\QIOp  u)(\Bq)| \\ \leq
      \begin{aligned}[t]
        & \int\limits_{\hlb}^{1} \int\limits_{K_{\Bp}}\int\limits_{K_{\Bq}}
        |\psi_{\Bp}^{K_{\Bp}}(\Bx)| |\psi_{\Bq}^{K_{\Bq}}(\By)|\,
        |\grad u(\By+\tau(\Bx-\By))||\Bx-\By|\,\mathrm{d}\By
        \mathrm{d}\Bx\mathrm{d}\tau \\
        + & \int\limits_{\hlb}^{1} \int\limits_{K_{\Bp}}\int\limits_{K_{\Bq}}
        |\psi_{\Bp}^{K_{\Bp}}(\Bx)| |\psi_{\Bq}^{K_{\Bq}}(\By)|\,
        |\grad u(\Bx+\tau(\By-\Bx))|\,|\Bx-\By|\,\mathrm{d}\By
        \mathrm{d}\Bx\mathrm{d}\tau
      \end{aligned}
    \end{multline*}
    The transformation formula for integrals reveals
    \begin{gather*}
      \int\nolimits_{K}f(\Bx+\tau(\By-\Bx))\,\mathrm{d}\By =
      \tau^{-3} \int\nolimits_{K'} f(\Bz)\,\mathrm{d}\Bz\;,\quad
      K' := \Bx+\tau(K-\Bx)\;.
    \end{gather*}
    Appealing to the bounds for {$\NLtwo[K]{\psi_{j}^{K}}$},
    {$\NLp[K]{\psi_{j}^{K}}{1}$}, $K\in{\mesh}$, from \eqref{eq:s4}, the
    {Cauchy-Schwarz inequality} yields
    \begin{align}
      \label{eq:qi1}
      &|(\QIOp  u)(\Bp)-(\QIOp  u)(\Bq)| \NLtwo{\Vb_{E}}\\
      &\qquad\qquad
      \leq C\underbrace{\left[\frac{|\Bp-\Bq| \NLtwo{\Vb_{E}}}
        {\min\{|K_{\Bq}|^{\frac{1}{2}},|K_{\Bp}|^{\frac{1}{2}}
         \}}\right]}_{{
          \leq C=C(\rho_{{\mesh}})}}
      \int\nolimits_{\hlb}^{1}\tau^{-3/2}\,\mathrm{d}\tau\cdot
      \SNHone[\left<\Omega_{E}\right>]{u}\;. \notag
    \end{align}
    Here {$\left<\Omega_{E}\right>$} stands for the convex hull of all tetrahedra
    adjacent to the edge $E$.

    (ii)\quad now consider {$E=[\Bp,\Bq]\in\Ce(\mesh)$},
    {$\Bp\in\DBd$}. Then, for any $u\in\bHone{\DBd}$
    \begin{align*}
    &|(\QIOp  u)(\Bq)-(\QIOp  u)(\Bp)|^{2} =
    |(\QIOp  u)(\Bq)|^{2} = \left|\int\nolimits_{K_{\Bq}}
      \psi_{\Bq}^{K_{\Bq}}(\Bx)u(\Bx)\,\mathrm{d}\Bx\right|^{2} \\
    = & \Bigl|\int\nolimits_{K_{\Bq}}
    \underbrace{\operatorname{dist}(\Bx,\DBd)}_{\leq C |\Bp-\Bq|}
    \psi_{\Bq}^{K_{\Bq}}(\Bx)
    \frac{u(\Bx)}{\operatorname{dist}(\Bx,\DBd)}\,\mathrm{d}\Bx\Bigr|^{2}
    \leq  \frac{C |\Bp-\Bq|^{2}}{|K_{\Bq}|}\cdot \int\limits_{K_{\Bq}}
    \left|\frac{u(\Bx)}{\operatorname{dist}(\Bx,\DBd)}\right|^{2}\,\mathrm{d}\Bx\\
    \leq &\frac{C}{|\Bp-\Bq|}  \int\limits_{\Omega_{E}}
    \left|\frac{u(\Bx)}{\operatorname{dist}(\Bx,\DBd)}\right|^{2}\,\mathrm{d}\Bx
    \;,
  \end{align*}
  with (different) constants {$C=C(\rho_{{\mesh}})>0$}.

  Combining \eqref{eq:s5}, \eqref{eq:qi1}, using the finite overlap property of
  ${\mesh}$ in the form
  \begin{gather*}
    \exists C=C(\rho_{{\mesh}}):\quad
    \sharp\{E\in\Ce({\mesh}):\, \Bx\in\left<\Omega_{E}\right>\} \leq C\quad
    \forall \Bx\in\Omega\;,
  \end{gather*}
  and appealing to Thm.~\ref{thm:Hardy} confirm $\SNHone{\QIOp  u} \leq C
  \SNHone{u}$. Observe that the Hardy inequality makes the constant depend on
  $\Omega$ and $\DBd$ in addition.

  The quasi-interpolation error estimate \eqref{eq:int-err} results from
  scaling arguments. Pick $K\in\mesh$, $u\in\Hm{2}\cap\bHone{\DBd}$, and write
  $\LIP_{K}u\in\bbP_{1}(K)$ for the linear interpolant of $u$ on $K$.
  Thanks to the projection property, we deduce as in Part I of the proof that,
  with $C=C(\rho_{\mesh})$,
  \begin{eqnarray*}
    \NLtwo[K]{(Id-\QIOp )u} &=&
    \NLtwo[K]{(Id-\QIOp )(u-\LIP_{K}u)} \leq
    C \NLtwo[\Omega_{K}]{u-\LIP_{K}u} \\ &\leq& C
    h_{K}^{2}\SNHm[\Omega_{K}]{u}{2}\;.
  \end{eqnarray*}
  Here, we wrote $\Omega_{K}:=\bigcup\{\ol{K'}:\,\ol{K'}\cap K\not=\emptyset\}$, and
  the final estimate can be shown by a simple scaling argument, \textit{cf.}
  \eqref{eq:fem15}. Estimate \eqref{eq:int-err} for $k=1$ follows by scaling
  arguments and interpolation between the Sobolev spaces $H^{2}(\Omega_{K})$ and
  $L^{2}(\Omega_{K})$.
\end{proof}

\subsection{Multilevel splitting of $\LFE({{\mesh}_{L}})$}
\label{sec:glob-mult-splitt}

In this section we first revisit the well-known \rh{\cite{JXU92,OSW92,ZHA92}} uniform
stability of multilevel splittings
of $\Hone$-conforming Lagrangian finite element functions in the case of mesh
hierarchies generated by uniform, \textit{i.e.} non-local, regular refinement.

We take for granted a virtual refinement hierarchy \eqref{eq:ref2} of tetrahedral
meshes as introduced in Sect.~\ref{sec:LMG} and its accompanying quasi-uniform
family of meshes \eqref{eq:ref24}.

Owing to the $\inf$ in \eqref{eq:C-stab}, it is enough to find a concrete family of
admissible ``candidate'' decompositions that enjoys the desired $L$-uniform stability.
We aim for candidates that fit the locally refined mesh hierarchy.

The principal idea, \rh{borrowed from \cite[Sect.~4.2.2]{OSW94},} is to use a sequence
  of quasi-interpolation operators $\QIP_{l}:\Ltwo\mapsto \LFE(\wh{\mesh}_{l})$ based
  on a judiciously chosen node$\to$element--assignments.  For $\wh{\mesh}_{l}$ we
  introduce a ``coarsest neighbor node$\to$element--assignment'': First, for any
  $\Bp\in\Cn(\wh{\mesh}_{l})$, $l=1,\ldots,L$, we pick $K\in{\mesh}_{l}$ such that
\begin{gather*}
  \lev(K) = \min\{\lev(K):\;\Bp\in \ol{K},\; K\in\mesh_{l}\}\;.
\end{gather*}
Secondly, we select a ``coarsest neighbor''
$K_{\Bp}\in\wh{\mesh}_{l}$ among those elements of
$\wh{\mesh}_{l}$ that are contained in $K$. This defines a mapping
$\Cn(\wh{\mesh}_{l})\mapsto\wh{\mesh_{l}}$, $\Bp\mapsto K_{\Bp}$.
We write $\ol{\QIP}_{l}:\Ltwo\mapsto \ol{\LFE}(\wh{\mesh}_{l})$ for the induced
quasi-interpolation operator according to Def.~\ref{def:QIP}.

Next, we examine the candidate multilevel splitting
\begin{gather}
  \label{eq:s6}
  u_{h} = \ol{\QIP}_{0}u_{h} + \sum\limits_{l=1}^{L}(\ol{\QIP}_{l}-\ol{\QIP}_{l-1})u_{h}\;,\quad
  u_{h}\in\LFE(\wh{\mesh}_{L})\;.
\end{gather}

\begin{lemma}
  \label{lem:H1stab}
  There holds, with a constant $C>0$ depending only on $\Omega$ and the uniform
  bound for the shape regularity measures $\rho_{\wh{\mesh}_{l}}$, $0\leq l\leq L$,
  \begin{equation}
    \SNHone{\ol{\QIP}_0 u_{h} }^2+
    \sum_{l=1}^Lh_l^{-2} \NLtwo{(\ol{\QIP}_l-\ol{\QIP}_{l-1}) u_{h} }^2\le
    C\SNHone{ u_{h} }^2\quad \forall u_{h} \in \LFE(\wh{\mesh}_L)\;.
    \label{eq:Ustab}
  \end{equation}
\end{lemma}
\newcommand{\ufn}{u}

\begin{proof}
  We take the cue from the elegant approach of Bornemann and Yserentant in
  \cite{BOY93}, who discovered how to bring techniques of real interpolation theory
  of Sobolev spaces \cite{LIM72}, \cite[Appendix B]{MCL00} to bear on \eqref{eq:s6}.
  The main tools are the so-called $K$-functionals given by
\begin{eqnarray*}
  K(t,u)^2&:=&\inf_{w\in H^2(\Omega)}\left\{\NLtwo{u-w}^2
    +t^2\SN{w}^2_{H^2(\Omega)}\right\}\;,\\
  K_{\bbR^3}(t,u)^2&:=&\inf_{w\in
    H^2(\bbR^3)}\left\{\N{u-w}^2_{L^2(\bbR^3)}
    +t^2\SN{w}^2_{H^2(\bbR^3)}\right\}\;.
\end{eqnarray*}
The estimates \eqref{eq:H1stab} and \eqref{eq:int-err} of Lemma~\ref{lem:Qint}
create a link between the terms in \eqref{eq:Ustab} and $K(t,u)$: owing to
\eqref{eq:L2stab} and \eqref{eq:int-err} there holds for any
$u\in\Ltwo$
\begin{eqnarray*}
  \NLtwo{(\ol{\QIP}_l-\ol{\QIP}_{l-1}) u } &\leq&
  \NLtwo{(\ol{\QIP}_l-\ol{\QIP}_{l-1})(u-w)} +
  \NLtwo{(\ol{\QIP}_l-\ol{\QIP}_{l-1})w} \\
  &\leq& C\bigl(\NLtwo{u-w} +  h_{l}^{2}\SNHm{w}{2}\bigr)\quad\forall w\in\Hm{2}\;.
\end{eqnarray*}
Here and below the generic constants $C$ may depend on shape regularity
$\max\limits_{0\leq l\leq L}\rho_{\wh{\mesh}_{l}}$ and the (quasi-uniformity)
constants in \eqref{eq:ref25}. We conclude
\begin{gather}
  \label{eq:s7}
  \NLtwo{(\ol{\QIP}_l-\ol{\QIP}_{l-1}) u }^2 \leq C\,K(h_{l}^{2},u)^{2}\quad
  \forall u\in\Ltwo\;,
\end{gather}
which implies
\begin{equation}
  \label{eq:Ksum}
  \SNHone{\ol{\QIP}_0\ufn}^2+\sum_{l=1}^Lh_l^{-2} \NLtwo{(\ol{\QIP}_l-\ol{\QIP}_{l-1})\ufn}^2\le
  C \Big\{\SNHone{\ufn}^2+\sum_{l=1}^Lh_l^{-2}
  K(h^2_l,\ufn)\Big\}\;.
\end{equation}
Let $\wt{\ufn}\in H^1(\bbR^3)$ be the Sobolev extension of $\ufn$ such
that, with $C=C(\Omega)>0$,
$$\wt{\ufn}_{|\Omega}=\ufn\quad\hbox{and}\quad \SN{\ufn}_{H^1(\bbR^3)}\le C\SNHone{\ufn}.$$
Define the Fourier Transform of $\wt{\ufn}$ by
$$\wh{\ufn}(\xibf)=\frac{1}{(2\pi)^{3/2}}\int_{\bbR^3}
\wt{\ufn}(\Bx) e^{-\imath\,\Bx\cdot\xibf}\,\mathrm{d}\Bx.$$ By the
equivalent definition of Sobolev-norms on $\bbR^{3}$
$$|\wt{\ufn}|^2_{H^i(\bbR^3)}\approx \int_{\bbR^3}|\xibf|^{2i}
|\wh{\ufn}(\xibf)|^2\,\mathrm{d}\xibf\;,\quad i=0,1,$$ we have
\begin{eqnarray}
K_{\bbR^3}(t,\wt{\ufn})^2&\le&C\inf_{w\in
H^2(\bbR^3)}\int_{\bbR^3}\Big\{\SN{\wh{\ufn}(\xibf)-\wh{w}(\xibf)}^2
+t^2|\xibf|^4\SN{\wh{w}}^2\Big\} \,\mathrm{d}\xibf\notag\\
&=&C\int_{\bbR^3}\frac{t^2|\xibf|^4}{1+t^2|\xibf|^4}|\wh{\ufn}(\xibf)|^2
\,\mathrm{d}\xibf\;,\notag
\end{eqnarray}
because the infimum is attained for \cite[Thm.~B7]{MCL00}
$\wh{w}(\xibf)=\wh{\ufn}(\xibf)/(1+t^2|\xibf|^4)$.
Since
\begin{eqnarray}
K(t,\ufn)^2&=&\inf_{w\in H^2(\Omega)}\left\{\NLtwo{\ufn-w}^2
+t^2\SN{w}^2_{H^2(\Omega)}\right\},\notag\\
&=&\inf_{w\in H^2(\bbR^3)}\left\{\NLtwo{\ufn-w}^2
+t^2\SN{w}^2_{H^2(\Omega)}\right\}\le
K_{\bbR^3}(t,\wt{\ufn})^2,\notag
\end{eqnarray}
we deduce that
\begin{eqnarray}
\sum_{l=1}^Lh_l^{-2} K(h^2_l,\ufn)^2&\le& C\sum_{l=1}^L
\int_{\bbR^3}\frac{h_l^2|\xibf|^4}{1+h_l^4|\xibf|^4}|\wh{\ufn}(\xibf)|^2
\,\mathrm{d}\xibf\label{eq:Ksum-stab}\\
&\le& C\sup_{\xibf\in\bbR^3}\Big\{\sum_{l=1}^L
\frac{\theta^{2l}|\xibf|^2}{1+\theta^{4l}|\xibf|^4}\Big\}
\int_{\bbR^3}|\xibf|^2|\wh{\ufn}(\xibf)|^2 \,\mathrm{d}\xibf\notag\\
&\le& C\SN{\wh{\ufn}}^2_{H^1(\bbR^3)}\le C\SNHone{\ufn}^2,\notag
\end{eqnarray}
where we have used assumption \eqref{eq:ref25}. The proof is finished by combining
\eqref{eq:Ksum} and \eqref{eq:Ksum-stab}.
\end{proof}\vspace{2mm}

\rh{Now we restrict ourselves to $u_{h}\in\LFE(\mesh_{h})$. Then,} thanks to the
particular design of the node$\to$element--assignment underlying $\ol{\QIP}_{l}$, the
terms in the decomposition \eqref{eq:s6} turn out to be localized.\vspace{2mm}

\begin{lemma}
  For all $u_h\in \LFE(\mesh_h)$ and $0\le l \leq j \le L$,
  \begin{eqnarray}
    \ol{\QIP}_j \rh{u_h} = \rh{u_h}\quad\text{in}\;\;\Omega\setminus\omega_{l+1}.
    \label{eq:Qluh=uh}
  \end{eqnarray}
\end{lemma}

\begin{proof}
  If $\Bp\in\ol{\Cn}(\wh{\mesh}_{j})$ and $\Bp\not\in \omega_{l+1}$ (open set !),
  then $K_{\Bp}\not\subset \omega_{l+1}$ ($K_{\Bp}\in \wh{\mesh}_{j}$). Recall that
  $K_{\Bp}$ was deliberately chosen such that there is $K\in\mesh_{l}$ with
  $K_{\Bp}\subset K$. Since $u_{h}$ is linear on $K$, the same holds for $K_{\Bp}$
  and \eqref{eq:s2} guarantees
  \begin{gather*}
    (\ol{\QIP}_{j}u_{h})(\Bp) = u_{h}(\Bp)\;.
  \end{gather*}
  When restricted to $\Omega\setminus\omega_{l+1}$, the mesh $\wh{\mesh}_{j}$ is a
  refinement of $\mesh_{h}$. Hence, agreement of the $\mesh_{h}$-piecewise linear
  function $u_{h}$ with $\ol{\QIP}_ju_h$ in all nodes of $\wh{\mesh}_{j}$ outside
  $\omega_{l+1}$ implies ${\ol{\QIP}_ju_h}_{|\Omega\setminus\omega_{l+1}} =
  {u_{h}}_{|\Omega\setminus\omega_{l+1}}$.
\end{proof}\vspace{2mm}

Consequently, for any $u_{h}\in\LFE(\mesh_{h})$, outside
$\omega_{l}$ both $\ol{\QIP}_{l}u_{h}$ and $\ol{\QIP}_{l-1}u_{h}$
agree with $u_{h}$. \vspace{2mm}

\begin{corollary}
  \label{cor:supp}
  For any $u_h\in \LFE(\mesh_h)$ and $1\le l\le L$,
  \begin{gather*}
    \operatorname{supp}((\ol{\QIP}_{l}-\ol{\QIP}_{l-1})u_{h}) \subset \ol{\omega}_{l}\;.
  \end{gather*}
\end{corollary}

In other words, the components of \eqref{eq:s6} are localized inside refined
regions of $\Omega$. In light of the definition \eqref{eq:3} of the refinement
zones, we also find
\begin{gather}
  \label{eq:s8}
  (\ol{\QIP}_{l}-\ol{\QIP}_{l-1})u_{h} \in \ol{\LFE}(\mesh_{l})\;!
\end{gather}

However, having used $\ol{\QIP}_{l}$ we cannot expect the
splitting to match potential homogeneous Dirichlet boundary
conditions. This can be remedied using Oswald's trick
\cite[Cor.~30]{OSW90}. We fix $u_{h}\in\LFE(\mesh_{h})$ and
abbreviate $u_{0}=\ol{\QIP}_{0}u_{h}\in\ol{\LFE}(\mesh_{0})$,
$u_{l}:=(\ol{\QIP}_{l}-\ol{\QIP}_{l-1})u_{h}\in\ol{\LFE}(\mesh_{l})$,
$l\ge 1$. Then, we consider the partial sums
\begin{gather}
  \label{eq:s9}
  \ol{s}_{l} := \sum\limits_{j=0}^{l} u_{j} \in \ol{\LFE}(\mesh_{l})
  \quad l\ge 0\;.
\end{gather}
Dropping those basis functions in
$\ol{\Bas}_{\LFE}(\mesh_{l})$ that belong to vertices in
$\ol{\DBd}$ in the representation of $\ol{s}_{l}$ we arrive at
$s_{l}\in\LFE(\mesh_{l})\in\bHone{\DBd}$.

Due to Cor.~\ref{cor:supp}, we observe that
\begin{gather}
  \label{eq:s10}
  \text{$\ol{s}_{l}$ and $\ol{s}_{l-1}$ agree on $\Omega\setminus\omega_{l}$.}
\end{gather}
Hence, away from $\ol{\omega}_{l}\cap\ol{\DBd}$ the same basis contribution are
removed from both functions when building $s_{l}$ and $s_{l-1}$, respectively.
This permits us to conclude
\begin{gather}
  \label{eq:s11}
  \text{${s}_{l}$ and ${s}_{l-1}$ agree on $\Omega\setminus\omega_{l}$.}
\end{gather}
Putting it differently,
\begin{gather}
  \label{eq:s13}
  \operatorname{supp}(s_{l}-s_{l-1}) \subset \overline{\omega}_{l}\;.
\end{gather}
Hence, for all $1\leq l \leq L$ we can estimate
\begin{gather}
  \label{eq:s12}
  \begin{aligned}
  \NLtwo{s_{l}-s_{l-1}} =&
  \NLtwo[\omega_{l}]{s_{l}-s_{l-1}} \\ \leq &
  \NLtwo[\omega_{l}]{s_{l}-\ol{s}_{l}}+
  \NLtwo[\omega_{l}]{s_{l-1}-\ol{s}_{l-1}} +
  \NLtwo[\omega_{l}]{u_{l}}\;.
  \end{aligned}
\end{gather}
The benefit of zeroing in on $\omega_{l}$ is that on this subdomain $\ol{s}_{l}$ has
the same ``uniform scale'' $h_{l}$ as $u_{l}$. Thus, repeated application of
uniform $L^{2}$-stability estimates \eqref{eq:fem18} for basis representations and
elementary Cauchy-Schwarz inequalities make possible the estimates (for arbitrary
$0<\epsilon<\frac{1}{2}$)
\begin{eqnarray*}
  \NLtwo[\omega_{l}]{s_{l}-\ol{s}_{l}}^{2} &\leq& C
  h_{l}^{3} \sum\limits_{\Bp\in\Cn(\Gamma_{l})} \ol{s}_{l}(\Bp)^{2} \leq
  C h_{l} \NLtwo[\Gamma_{l}]{{\ol{s}_{l}}_{|\partial\Omega}}^{2} =
  C h_{l} \NLtwo[\Gamma_{l}]{{\sum\limits_{j=0}^{L}u_{j}-\sum\limits_{j=0}^{l}u_{j}
    }}^{2}  \\ &\leq&
  Ch_{l} \Bigl(\sum\limits_{j=l+1}^{L}\NLtwo[\Gamma_{l}]{u_{j}}\Bigr)^{2} \leq
  Ch_{l}
  \Bigl(\sum\limits_{j=l+1}^{L}h_{j}^{-\frac{1}{2}}\NLtwo[\omega_{l}]{u_{j}}\Bigr)^{2}\\
  &\leq& Ch_{l}\cdot
  \sum\limits_{j=l+1}^{L} h_{j}^{1-2\epsilon}\cdot
  \sum\limits_{j=l+1}^{L}
  h_{j}^{2\epsilon-2}\NLtwo[\omega_{l}]{u_{j}}^{2}\\
  &\leq& Ch_{l}^{2-2\epsilon}\cdot
  \sum\limits_{j=l+1}^{L} h_{j}^{2\epsilon-2}\NLtwo[\omega_{l}]{u_{j}}^{2}\;.
\end{eqnarray*}
Here the set $\Cn(\Gamma_{l})$ comprises the nodes of $\ol{\Cn}(\wh{\mesh}_{l})$
that lie on $\ol{\omega}_{l}\cap\ol{\DBd}$ and we make heavy use of the geometric
decay of $h_{l}$. The latter also yields
\begin{eqnarray*}
  \sum\limits_{l=1}^{L}h_{l}^{-2}\NLtwo[\omega_{l}]{s_{l}-\ol{s}_{l}}^{2} &\leq&
  C \sum\limits_{l=1}^{L} h_{l}^{-2\epsilon}
  \sum\limits_{j=l+1}^{L} h_{j}^{2\epsilon-2}\NLtwo[\omega_{l}]{u_{j}}^{2} \\ & =&
  C \sum\limits_{j=2}^{L}\left(\sum\limits_{l=1}^{j}h_{l}^{-2\epsilon}\right)
  h_{j}^{-2+2\epsilon}\NLtwo[\omega_{l}]{u_{j}}^{2} \\
  &\leq& C \sum\limits_{j=2}^{L} h_{j}^{-2}\NLtwo{u_{j}}^{2}
  \leq C \SNHone{u_{h}}^{2}\;,
\end{eqnarray*}
by virtue of Lemma~\ref{lem:H1stab}. Except for the last line, all constants
depend only on $\rho_{\wh{\mesh}_{l}}$ and the constants in \eqref{eq:ref25}. Merging the
last estimate with \eqref{eq:s12} gives us
\begin{gather}
  \label{eq:s14}
  \sum\limits_{l=1}^{L} h_{l}^{-2}\NLtwo{s_{l}-s_{l-1}}^{2} \leq C \SNHone{u_{h}}^{2}\;.
\end{gather}
Thus, in light of \eqref{eq:s13} and the following identity
\begin{equation*}
s_0+\sum\limits_{l=1}^{L} (s_{l}-s_{l-1})=s_L=\ol{s}_L=u_h,
\end{equation*}
we have accomplished the proof of the following theorem:

\begin{theorem}
  \label{thm:H1main}
  For any $u_{h}\in \LFE(\mesh_{h})$ we can find $u_{l}\in \LFE(\mesh_{l})$ such that
  \begin{gather}
    \label{eq:r2}
    u_{h}=\sum\limits_{l=0}^{L}u_{l},
    \qquad\operatorname{supp}(u_{l}) \subset \ol{\omega}_{l}\;,
  \end{gather}
  and
  \begin{gather*}
    \SNHone{u_{0}}^{2} + \sum\limits_{l=1}^{L}h_{l}^{-2}\NLtwo{u_{l}}^{2}
    \leq C \SNHone{u_{h}}^{2}\;,
  \end{gather*}
  with $C>0$ independent of $L$.
\end{theorem}

Notice that in combination with the $L^{2}$-stability \eqref{eq:fem18} of nodal bases
and inverse inequalities, this theorem asserts an $L$-uniform estimate of the form
\eqref{eq:C-stab} for the splitting \eqref{eq:LMGgrad} w.r.t. the energy norm
$\SNHone{\cdot}$. From \eqref{eq:r2} it is clear that the basis functions admitted in
\eqref{eq:LMGgrad} can represent the functions $u_{l}$ of Thm.~\ref{thm:H1main}.

\begin{remark}
  \label{rem:58}
  It is interesting to note that, in contrast to other analyses \cite{AIM01,BOY93},
  the above proof does not \rh{hinge on Assumption~\ref{ass:hn}.
    Thm.~\ref{thm:H1main} remains valid for an arbitrary number of hanging nodes on
    an active edge. Howver, this does not translate into asymptotically optimal
    convergence of local $H^{1}$-multigrid in this case, because, in order to infer it from
    Thm.~\ref{thm:H1main}, we also need uniform $L^{2}$-stability of the bases.}
\end{remark}


\subsection{Helmholtz-type decompositions}
\label{sec:helmh-type-decomp}

\newcommand{\EQ}{\mcal{Q}} 
\newcommand{\EQH}{\EQ_h} 

Helmholtz-type decompositions, \rh{also called \emph{regular decompositions}}, have
emerged as a powerful tool for answering questions connected with $\Hcurl$. In
particular, they have paved the way for a rigorous multigrid theory for
$\Hcurl$-elliptic problems \cite{HIP99,HIX06,HWZ05,PAZ02,GPD03,HIP00b,ZCW05a,KOV08}. We
refer to \cite[Sect.~2.4]{HIP02} for more information.

We will need a very general version provided by the following theorem.

\begin{theorem}
  \label{thm:hdec}
  Let $\Omega$ meet the requirements stated in Sect.~\ref{sec:introduction}. Then,
  for any $\Vv\in \bHcurl{\DBd}$, there exists a $p\in \bHone{\DBd}$
  and $\BPsi\in (H^1_{\DBd}(\Omega))^3$ such that
  \begin{eqnarray}
    &\Vv=\nabla p + \BPsi,&\label{eq:hdec}\\
    &\SNHone{p}\le C\|\Vv\|_{\Hcurl}, \quad\NHone{\BPsi}\le
    C\|\curl\Vv\|_{\Ltwo},&\label{eq:hstab}
  \end{eqnarray}
  where the constant $C$ depends only on $\Omega$.
\end{theorem}

  \begin{figure}[!htb]
    \centering
    \psfrag{O}{$\Omega$}
    \psfrag{O1}{$\Omega_{1}$}
    \psfrag{O2}{$\Omega_{2}$}
    \psfrag{O3}{$\Omega_{3}$}
    \includegraphics[width=0.7\textwidth]{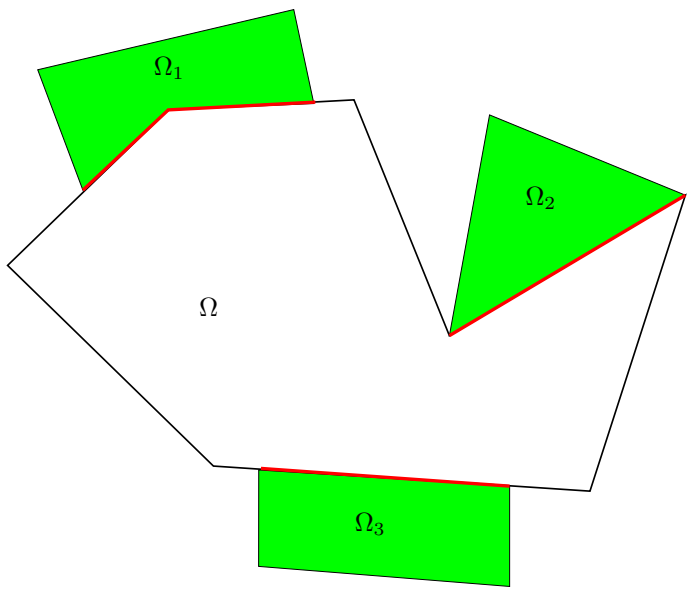}
    \caption{Buffer zones attached to connected components of (red) Dirichlet
      boundary part $\DBd$}
    \label{fig:attachdom}
  \end{figure}

  \begin{proof} Given $\Vu\in\bHcurl{\DBd}$, we define
    $\wt{\Vu}\in\Hcurl[\wt{\Omega}]$, $\wt{\Omega} :=
    \operatorname{interior}(\ol{\Omega}\cup \ol{\Omega}_{1}\cup
    \ol{\Omega}_{2}\cup\dots)$ (see Sect.~\ref{sec:introduction} and
    Fig.~\ref{fig:attachdom} for the meaning of $\Omega_{i}$), by
  \begin{gather}
    \label{eq:pr19}
    \wt{\Vu}(\Bx) =
    \begin{cases}
      \Vu(\Bx) & \text{for } \Bx\in \Omega\;,\\
      0 & \text{for } \Bx\in \Omega_{i}\;\text{for some } i\;.
    \end{cases}
  \end{gather}
  Notice that the tangential components of $\wt{\Vu}$ are continuous across
  $\partial\Omega$, which ensures $\wt{\Vu}\in\Hcurl[\wt{\Omega}]$. Then extend
  $\wt{\Vu}$ to $\ol{\Vu}\in\Hcurl[\bbR^{3}]$, see \cite{BCS00}.

  Since {$\curl\ol{\Vu}\in\kHdiv[\bbR^{3}]$}, Fourier techniques \cite[Sect.~3.3]{GIR86} yield
  a $\Phibf\in (\Hone[\bbR^{3}])^{3}$ that fulfills
  \begin{gather}
    \label{eq:pr20}
    \curl\Phibf = \curl\ol{\Vu}\;,\quad
    \NHone[\bbR^{3}]{\Phibf} \leq C \NLtwo[\bbR^{3}]{\curl\ol{\Vu}}\;,
  \end{gather}
  with $C=C(\Omega)>0$. As a consequence
  \begin{gather}
    \label{eq:pr21}
    \curl(\ol{\Vu}-\Phibf) = 0 \quad\Rightarrow\quad \ol{\Vu}-\Phibf=\grad
    q\quad\text{in }\bbR^{3}\;.
  \end{gather}
  On every $\Omega_{i}$, by definition $\ol{\Vu} = 0$, which implies
  $q_{|\Omega_{i}} \in \Hm[\Omega_{i}]{2}$. As the attached domains $\Omega_{i}$
  are well separated Lipschitz domains, see Fig.~\ref{fig:attachdom}, the
  $H^{2}$-extension of $q_{|\bigcup_{i}\Omega_{i}}$ to $\ol{q}\in \Hm[\bbR^{3}]{2}$ is
  possible. Moreover, it satisfies
  \begin{gather}
    \label{eq:pr24}
     \NHm[\bbR^{3}]{\ol{q}}{2} \leq C \NHm[\bigcup_{i}\Omega_{i}]{{q}}{2} \leq
     C \NHone[\bbR^{3}]{\Phibf} \leq \NLtwo{\curl \Vu}\;.
  \end{gather}
  \begin{gather}
    \label{eq:pr22}
    \ol{\Vu} = \Phibf - \grad\ol{q} + \grad(q+\ol{q})\;.
  \end{gather}
  Finally, set $\Psibf:=(\Phibf - \grad\ol{q})_{|\Omega}$, $p:= q+\ol{q}$, and observe
  \begin{gather}
    \label{eq:pr23}
    \NHone{\Psibf} \leq \NHone[\bbR^{3}]{\Phibf} + \NHm[\bbR^{3}]{\ol{q}}{2} \leq
    C \NLtwo{\curl \Vu}\;.
  \end{gather}
  The constants may depend on $\Omega$, $\DBd$, and the chosen $\Omega_{i}$.
\end{proof}

The stable Helmholtz-type decomposition \eqref{eq:hdec}
immediately suggests the following idea: when given
$\Vv_{h}\in\EFE(\mesh_{h})$, first split it according to
\eqref{eq:hdec} and then attack both components by the uniformly
$H^{1}$-stable local multilevel decompositions explored in the
previous section. Alas, the idea is flawed, because neither of the
terms in \eqref{eq:hdec} is guaranteed to be a finite element
function, even if this holds for $\Vv_h$.

Fortunately, the idea can be mended by building a purely discrete counterpart of
\eqref{eq:hdec} as in \cite[Lemma~5.1]{HIX06} (\rh{called there ``discrete regular
  decomposition''}). For the sake of completeness we elaborate the proof below.

\begin{lemma}
  \label{lem:disc-hem}
  For any $\Vv_h\in\EFE(\mesh_h)$, there is $\BPsi_h\in (\LFE(\mesh_h))^3$,
  $p_h\in \LFE(\mesh_h)$, and $\wt{\bf v}_h\in\EFE(\mesh_h)$ such that
  \begin{eqnarray}
    &\Vv_h=\wt{\bf v}_h+\BPi_h\BPsi_h+\nabla p_h\;,&\label{eq:dec-vh}\\
    &\qquad\;\N{p_h}_{\Hone}\le C\N{\Vv_h}_{\bHcurl{}}\;,& \label{eq:hd15}\\
    & \N{h^{-1}\wt{\Vv}_h}_{\Ltwo}+\N{\BPsi_h}_{\Hone}\le
    C\N{\curl\Vv_h}_{\Ltwo}\;,&\label{eq:stab-vh}
  \end{eqnarray}
  where the constant $C$ depends only on $\Omega$, $\DBd$, and the shape
  regularity of $\mesh_h$.
\end{lemma}

\begin{proof} (\textit{cf.} \cite[Lemma~5.1]{HIX06})
  We fix a $\Vv_{h}\in
  \EFE(\mesh_{h})$ and use the stable regular decomposition of Thm.~\ref{thm:hdec} to
  split it according to
  \begin{gather}
    \label{eq:hd1}
    \Vv_{h} = \Psibf + \grad p\;,\quad
    \Psibf\in(\bHone{\DBd})^{3}\;,\quad
    p \in \bHone{\DBd}\;.
  \end{gather}
  We have already known that the functions $\Psibf$ and $p$ satisfy
  \begin{gather}
    \label{eq:hd2}
    \NHone{\Psibf} \leq C \NLtwo{\curl\Vv_{h}}\quad,\quad
    \NLtwo{\grad p} \leq C \NHcurl{\Vv_{h}}\;,
  \end{gather}
  with constants depending only on $\Omega$ and $\DBd$.

  Next, note that in \eqref{eq:hd1} $\curl\Psibf=\curl\Vv_{h}\in \curl\EFE(\mesh_{h})$,
  and, owing to Lemma~\ref{lem:42}, $\EIP_{h}\Psibf$ is well defined.
  Further, a commuting diagram property together with Lemma~\ref{lem:notop} implies
\begin{gather}
  \label{eq:hd3}
  \curl(Id-\EIP_{h})\Psibf = 0\quad \Rightarrow \quad
  \exists q\in \bHone{\DBd}:\quad
  (Id-\EIP_{h})\Psibf = \grad q\;.
\end{gather}
The estimate of Lemma~\ref{lem:42} together with \eqref{eq:hd2} yields
\begin{gather}
  \label{eq:hd6}
  \NLtwo{h^{-1}\grad q} = \NLtwo{h^{-1}(Id-\EIP_{h})\Psibf}
  \leq C \SNHone{\Psibf} \leq C\NLtwo{\curl\Vv_{h}} \;.
\end{gather}

In order to push $\Psibf$ into a finite element space, a quasi-interpolation operator
$\QIP_{h}:(\Ltwo)^{3}\mapsto (\LFE(\mesh_{h}))^{3}$ is the right tool.  We simply get
it from componentwise application of an operator according to Def.~\ref{def:QIP}
where any node$\to$element--assignment will do. Thus, we can define the terms in the
decomposition \eqref{eq:dec-vh} as
\begin{align}
  \label{eq:hd8}
  \widetilde{\Vv}_{h} & := \EIP_{h}(\Psibf-\QIP_h\Psibf) \in \EFE(\mesh_{h})\;,\\
  \label{eq:hd9}
  \Psibf_{h} & := \QIP_h\Psibf \in (\LFE(\mesh_{h}))^{3}\;,\\
  \label{eq:hd10}
  \grad p_{h} & := \grad(p+q)\;,\quad p_{h} \in \LFE(\mesh_{h})\;.
\end{align}
Indeed, $\grad(p+q)\in \EFE(\mesh_{h})$ such that $p+q\in\LFE(\mesh_{h})$.
The stability of the decomposition \eqref{eq:dec-vh} can be established as
follows: first, make use of Lemma~\ref{lem:42} and \eqref{eq:int-err} to obtain,
with $C=C(\rho_{\mesh_{h}})>0$,
\begin{gather}
  \label{eq:hd11}
  \begin{aligned}
    \NLtwo{h^{-1}\widetilde{\Vv}_{h}} &
    \leq
    \NLtwo{h^{-1}(Id-\EIP_{h})(\Psibf-\QIP_h\Psibf)} +
    \NLtwo{h^{-1}(Id-\QIP_h)\Psibf} \notag \\
    & \leq C \SNHone{(Id-\QIP_h)\Psibf} + \SNHone{\Psibf} \notag \\
    & \leq C \SNHone{\Psibf} \leq C \NLtwo{\curl\Vv_{h}}\;.\\
  \end{aligned}
\end{gather}
Due to the definition \eqref{eq:hd9}, the next estimate is a simple consequence of
\eqref{eq:H1stab} and Thm.~\ref{thm:hdec}
\begin{gather}
  \label{eq:hd12}
  \NHone{\Psibf_{h}}  \leq C \NHone{\Psibf} \leq C  \NLtwo{\curl\Vv_{h}}\;.
\end{gather}
Finally, the estimates established so far plus the triangle inequality yield
\begin{gather}
\label{eq:hd13}
  \NLtwo{\grad p_{h}}  \leq C \NHcurl{\Vv_{h}} \;.
\end{gather}

\end{proof}


\subsection{Local multilevel splitting of $\VU(\mesh_h)$}
\label{sec:local-mult-splitt}

With the discrete Helmholz-type decomposition of Lemma~\ref{lem:disc-hem}
at our disposal, we can now tackle its piecewise linear and continuous
components with Thm.~\ref{thm:H1main}.

\begin{lemma}
  \label{lem:main}
  For any $\Vv_h\in\EFE(\mesh_h)$, there exists a constant $C$ only depending on the
  domain, the Dirichlet boundary part $\DBd$, the shape regularity of the meshes
  $\mesh_{l}$, $\wh{\mesh}_{l}$, $0\leq l \leq L$, and the constants in
  \eqref{eq:ref25}, such that
  \begin{gather}
    \label{eq:mdec-vh}
    \Vv_h=\sum_{l=0}^L\Big(\Vv_l+\nabla p_l\Big),\quad \Vv_l\in
    \Span{\Bas_{\EFE}^l},
    \;p_l\in \Span{\Bas_{\LFE}^l}\;,
  \end{gather}
  and
  \begin{multline}
    \label{eq:mstab-vh}
    \N{\Vv_0}^{2}_{\Hcurl}+\SNHone{p_0}^2+
    \sum\limits_{l=1}^{L}
    h_{l}^{-2}\left(\NLtwo{\Vv_l}^{2}+\NLtwo{p_l}^2\right) \leq C
    \N{\Vv_h}^2_{\Hcurl}\;,
  \end{multline}
  where $\Bas_{\LFE}^l$ and $\Bas_{\EFE}^l$ are defined in
  \eqref{eq:NB}.
\end{lemma}\vspace{2mm}

\begin{proof}
  We start from the discrete Helmholtz-type decomposition of $\Vv_h$ in \eqref{eq:dec-vh}:
  \begin{eqnarray}
    \Vv_h=\wt{\bf v}_h+\BPi_h\BPsi_h+\nabla p_h,\quad \BPsi_h\in
    (V(\mesh_h))^3,\;p_h\in V(\mesh_h),\;\wt{\bf
      v}_h\in\EFE(\mesh_h).\notag
  \end{eqnarray}
  We apply the result of Thm.~\ref{thm:H1main} about the existence of stable local
  multilevel splittings of $V(\mesh_h)$ componentwise to $\BPsi_{h}$: this gives
  \begin{eqnarray}
    &\BPsi_{h} = \sum\limits_{l=0}^{L}\BPsi_{l}\;,
    \quad \BPsi_l\in \Span{\mathfrak{B}_V^l}^3\;, &\label{eq:mdec-Psih}\\
    &\SNHone{\BPsi_{0}}^{2} + \sum\limits_{l=1}^{L}
    h_{l}^{-2}\NLtwo{\BPsi_{l}}^{2} \leq C \SNHone{\BPsi_{h}}^2\;.&
    \label{eq:mstab-Psih}
  \end{eqnarray}
  Observe that the functions $\BPsi_{l}$ do not belong to $\EFE(\mesh_{l})$. Thus, we
  target them with edge element interpolation operators $\EIP_{l}$ onto
  $\EFE(\mesh_{l})$, see \eqref{eq:fem14}, and obtain the splitting
  described in Lemma~\ref{lem:VLFEdec}:
\begin{equation}
  \label{eq:dec-Psi}
  \BPsi_{l}=\BPi_l\BPsi_l+\nabla w_l\;,\quad
  w_l\in\wt{V}_2(\mesh_l)\;.
\end{equation}
The gradient terms introduced by \eqref{eq:dec-Psi} are well under
control: writing $s_{h}:=\sum\nolimits_{l=0}^{L} w_{l}$, the $L^{2}$-stability
of \eqref{eq:dec-Psi}, see Lemma~\ref{lem:VLFEdec}, yields
\begin{eqnarray*}
  \NLtwo{\EIP_{l}\Psibf_{l}} &\leq& C \NLtwo{\Psibf_{l}}\;,\\
  \SNHone{s_h}^2 & \leq &
  C \Big(\sum_{l=0}^L\NLtwo{\BPsi_l}\Big)^2
  \leq C \sum_{l=0}^Lh_l^2\cdot
  \sum_{l=0}^Lh_l^{-2}\NLtwo{\BPsi_l}^2
  \overset{\text{\eqref{eq:mstab-Psih}}}{\leq}
  C\SNHone{\BPsi_h}^2\;.
\end{eqnarray*}
Because of $\curl\EIP_{0}\Psibf_{0} = \curl\Psibf_{0}$, we infer
from \eqref{eq:mstab-Psih}
\begin{gather}
  \label{eq:sp6}
  \N{\BPi_0\BPsi_{0}}^{2}_{\BH(\curl,\Omega)} + \sum\limits_{l=1}^{L}
  h_{l}^{-2}\NLtwo{\BPi_l\BPsi_{l}}^{2} \leq C \SNHone{\BPsi_{h}}^2\;.
\end{gather}
Above and throughout the remainder of the proof, constants are
independent of $L$.

By the projector property $\EIP_{h}\circ \EIP_{l} = \EIP_{l}$, $l=0,\ldots,L$,
and the commuting diagram property \eqref{CDP}, we arrive at
\begin{gather}
  \label{eq:sp4}
  \Vv_{h} = \wt{\Vv}_{h} + \sum\limits_{l=0}^{L}\EIP_{l}\Psibf_{l}
  + \grad(\LIP_{h} s_{h} + p_{h})\;,
\end{gather}
where $\LIP_h$ is the nodal linear interpolation operator onto $\LFE(\mesh_h)$.
Recall \eqref{eq:fem21} to see that
\begin{gather*}
  \SNHone{\LIP_{h} s_{h} + p_{h}} \leq
  C \SNHone{s_{h}} + \SNHone{p_{h}} \leq C \NHcurl{\Vv_{h}}\;.
\end{gather*}

The local multilevel splitting of $\LIP_h s_h+p_{h}$ according
to Thm.~\ref{thm:H1main} gives
\begin{gather}
\LIP_{h} s_{h} + p_{h} = \sum\limits_{l=0}^{L}p_{l}\;,
\quad p_l\in \Span{\mathfrak{B}_V^l}\;,\label{eq:mdec-ph}\\
\SNHone{p_{0}}^{2} + \sum\limits_{l=1}^{L}
h_{l}^{-2}\NLtwo{p_{l}}^{2} \leq C \SNHone{\LIP_{h} s_{h} +
p_{h}}^2\leq C \NHcurl{\Vv_{h}}^{2}\;. \label{eq:mstab-ph}
\end{gather}

Still, the contribution $\wt{\Vv}_{h}$ does not yet match
\eqref{eq:LMGcurl}. The idea is to distribute $\wt{\Vv}_{h}$ to
the terms $\EIP_{l}\Psibf_{l}$ by scale separation. To that end, we
assign a level to each active edge of $\mesh_{h}$
\begin{gather}
  \label{eq:sp8}
  \lev(E) := \min\{\lev(K):\; K\in\mesh_{h},\; E\;\text{is edge of }{K}\}\;,\quad
  E\in\Ce(\mesh_{h})\;.
\end{gather}
Thus, we distinguish parts of $\wt{\Vv}_{h}$ on different levels:
given the basis representation
\begin{gather}
  \label{eq:sp9}
  \wt{\Vv}_{h} = \sum\limits_{E\in\Ce(\mesh_{h})} \alpha_{E}\Vb_{E}\;,\quad
  {\{\Vb_{E}\}}_{E\in\Ce(\mesh_{h})} = \Bas_{\EFE}(\mesh_{h})\;,
\end{gather}
we split
\begin{gather}
  \label{eq:sp10}
  \wt{\Vv}_{h} = \sum\limits_{l=0}^{L} \wt{\Vv}_{l}\;,\quad
  \wt{\Vv}_{l} := \sum\limits_{{E\in\Ce(\mesh_{h})}\atop{\lev(E)=l}}
  \alpha_{E}\Vb_{E}\;,
  \quad \operatorname{supp}(\wt{\Vv}_{l}) \subset \ol{\omega}_{l}\;.
\end{gather}
The estimate $\NLtwo{\mwf^{-1}\wt{\Vv}_{h}} \leq C\NLtwo{\curl\Vv_{h}}$ from
Lemma~\ref{lem:disc-hem} means that $\wt{\Vv}_{h}$ is ``small on fine scales''.
Thanks to the $L^{2}$-stability \eqref{eq:fem17} of the edge bases, this carries over
to $\wt{\Vv}_{l}$:
\begin{eqnarray}
  \sum\limits_{l=0}^{L}h_{l}^{-2}\NLtwo{\wt{\Vv}_{l}}^{2} &\leq&
  C \sum\limits_{l=0}^{L}h_{l}^{-2}\sum\limits_{E\in\Ce(\mesh_{h}),
    \lev(E)=l}\alpha_{E}^{2}\NLtwo{\Vb_{E}}^{2} \notag \\
  \label{eq:spA}
  &\leq&  C \sum\limits_{l=0}^{L}h_{l}^{-2}\sum\limits_{E\in\Ce(\mesh_{h}),
    \lev(E)=l}\alpha_{E}^{2}\NLtwo[T_{E}]{\Vb_{E}}^{2} \\ &\leq&
  C \sum\limits_{l=0}^{L}h_{l}^{-2}\NLtwo[\Sigma_{l}]{\wt{\Vv}_{h}}^{2}
  \leq C \NLtwo{\mwf^{-1}\wt{\Vv}_{h}}^{2}\;,\notag
\end{eqnarray}
where $T_{E}\in\mesh_{h}$ is coarsest element adjacent to $E$, \textit{cf.}
\eqref{eq:sp8}, and refinement strips are defined by
\begin{gather}
  \label{eq:sp7}
  \Sigma_l := {\omega}_{l}\setminus \ol{\omega_{l+1}}\;,\;0\leq l < L,\quad
  \Sigma_{L} := \omega_{L}\;,
\end{gather}
see Figs.~\ref{fig:2drefzone} and \ref{fig:nvbrz}.

Yet, in the case of bisection refinement, $\wt{\Vv}_{l}$ may not be spanned by basis
functions in $\Bas_{\EFE}^{l}$, because the basis functions of $\EFE(\mesh_{h})$
attached to each edge on $\ol{\Sigma_l}\bigcap\ol{\omega_{l+1}}$, $0\le l<L$ do not
belong to any $\Bas_{\EFE}^{l}$!

\begin{figure}[!htb]
  \centering
  \begin{minipage}[t]{0.33\textwidth}\centering
    \psfrag{E}{$E$}
    \includegraphics[width=0.95\textwidth]{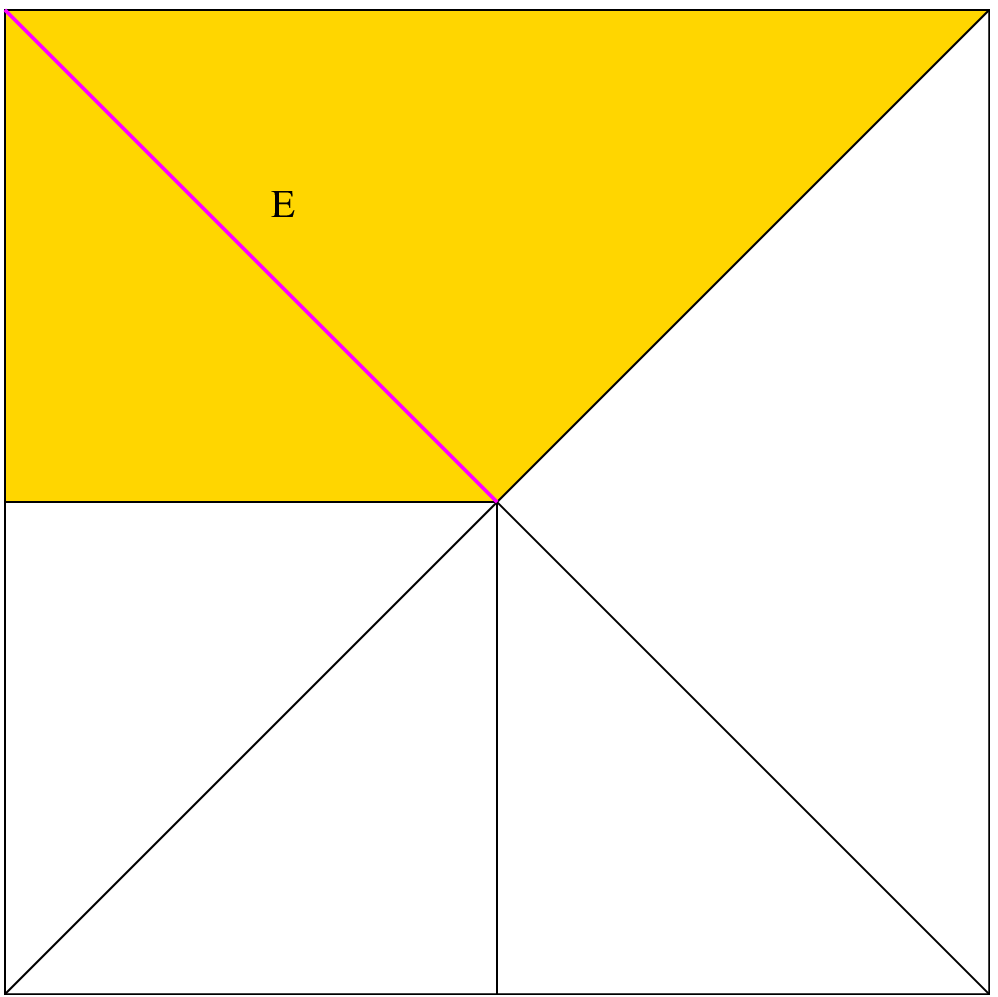}

    Edge $E$, support of basis function $\Vb_{E}$
  \end{minipage}%
  \begin{minipage}[t]{0.33\textwidth}\centering
    \psfrag{E}{$E$}
    \includegraphics[width=0.95\textwidth]{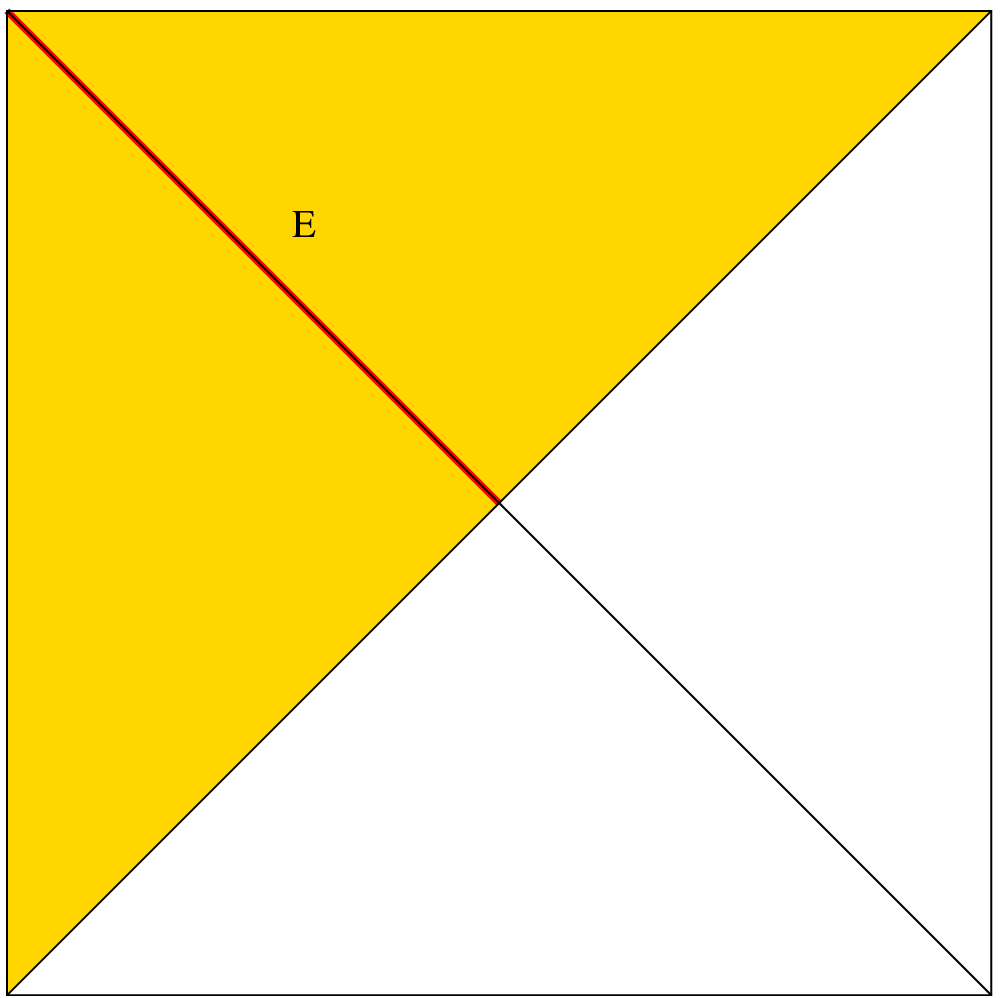}

    Support of $\Vb_{E}^{l}$
  \end{minipage}%
  \begin{minipage}[t]{0.33\textwidth}\centering 
    \psfrag{E1}{$E_{1}$}
    \psfrag{E2}{$E_{2}$}
    \includegraphics[width=0.95\textwidth]{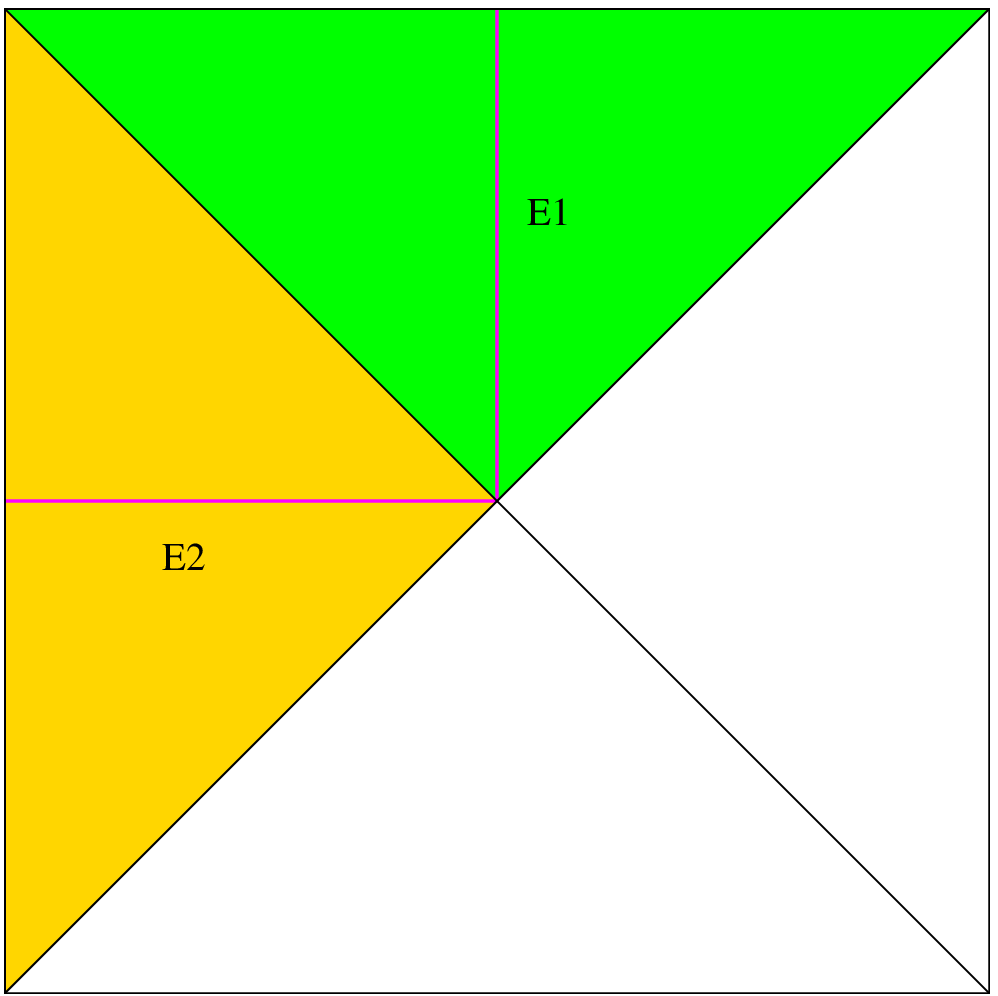}

    Edges supporting $\Vb^{l+1}_{E_{1}}$, $\Vb^{l+1}_{E_{2}}$
  \end{minipage}
  \caption{Basis function with which $\Vb_E$ can be represented}
  \label{fig:be}
\end{figure}

Take any $E\subset \ol{\Sigma_l}\bigcap\ol{\omega_{l+1}}$. Let $\Vb_E$, $\Vb_E^l$,
and $\Vb_E^{l+1}$ be the basis functions of $\EFE(\mesh_h)$, $\EFE(\mesh_l)$, and
$\EFE(\mesh_{l+1})$ associated with $E$, see Fig.~\ref{fig:be} for a 2D illustration.
Denote by $K_1,\ldots,K_n$ all elements in $\omega_{l+1}$ and $\mesh_l$ which contain
$E$, and by $E_1,\ldots,E_m$ their new edges connecting $E$ but not contained in the
refinement edges of $K_1,\ldots,K_n$ (see Fig.  \ref{fig:elm}). Supposing the
orientations of each $E_i$ and $E$ point to their common endpoint, we have
\begin{gather}
  \label{be:8}
  \Vb_E = \Vb^l_E+\frac{1}{2}\sum_{i=1}^m\Vb_{E_i}^{l+1}\;.
\end{gather}
This decomposition is $L^{2}$-stable with constants merely depending
on shape regularity.

\begin{figure}[h!]
\centerline{
\includegraphics*[width=3in]{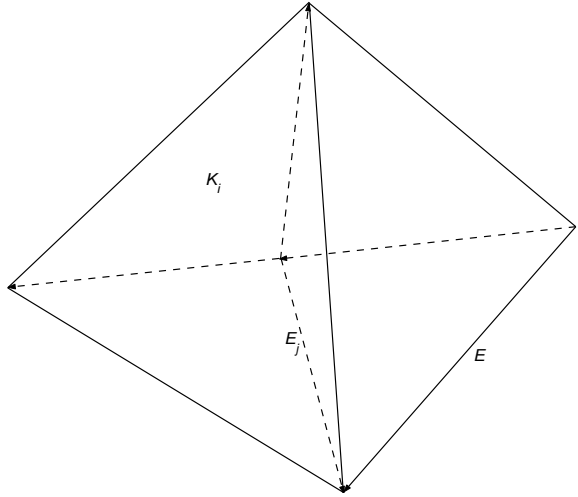}}
\caption {Situation at an edge $E$ lying on the interface between $\Sigma_l$ and
$\omega_{l+1}$.\label{fig:elm}}
\end{figure}

Since $\sum_{i=1}^m\Vb_{E_i}^{l+1}\in\Bas_{\EFE}^{l+1}$, we may move the component of
$\wt{\Vv}_l$ associated with this term to $\wt{\Vv}_{l+1}$ for any $E$. Then the
decomposition \eqref{eq:sp10} and the stability estimate \eqref{eq:spA} remain valid.

Summing up, the stability estimate \eqref{eq:sp6} is preserved after replacing
$\EIP_{l}\Psibf_{h}$ with $\EIP_{l}\Psibf_{h}+\wt{\Vv}_{l}\in \EFE(\mesh_{l})$.
\end{proof}

Eventually, the proof of Thm.~\ref{thm:main} is readily accomplished. With
Lemma~\ref{lem:main} at our disposal, we merely appeal to the $L^{2}$-stabilities
expressed in \eqref{eq:fem17} and \eqref{eq:fem18} and inverse inequalities to see
that all components in \eqref{eq:mdec-vh} can be split into local contributions of
basis functions in $\Bas_{\EFE}^{l}$ and $\Bas_{\LFE}^{l}$, respectively.

\section{Quasi-orthogonality}
\label{sec:quasi-orthogonality}

The strengthened Cauchy-Schwartz inequality \eqref{eq:C-orth} has been established in
\cite{JXU92,YSE86} for $H^1$-conforming linear Lagrangian finite element spaces, in
\cite[Sect.~6]{HIP97} for $\VH(\Div)$-elliptic variational problems and so-called
face elements. It is discussed in \cite[Sect.~4]{HIP00b} for \eqref{eq:VPcurl}, edge
elements, and geometric multigrid with global refinement.  The considerations for locally
refined meshes are fairly similar, but will be elaborated for the sake of
completeness. 

The trick is, not to consider the one-dimensional spaces spanned by individual basis
functions as building blocks of the splitting \eqref{ssc:dec}, but larger aggregates.
Thus, we put the nodal basis functions in $\Bas_{\LFE}^{l}$ and $\Bas_{\EFE}^{l}$
into a small number of classes, such that the supports of any two basis functions in
the same class do not overlap. Since these basis functions are attached to vertices
and edges respectively, the definition of those classes can be based on a
partitioning the vertices/edges of ${{\mesh}_l}_{|\omega_{l}}$ into disjoint sets such
that any two vertices/edges of the same set do not belong to the same tetrahedron.
The is formally stated in the following ``colouring lemma'':

\begin{lemma}
  \label{lem:subsplit}
  There exist $P_{\Cn},P_{\Ce}\in\mathbb{N}$ depending only on shape regularity such
  that the sets ${\Cn({\mesh}_{l})}_{|\omega_{l}}$ and
  ${\Ce(\mesh_{l})}_{\omega_{l}}$ of vertices and edges of $\mesh_{l}$ inside the
  refinement zone $\omega_{l}$ can be partitioned into subsets 
  \begin{align*}
    {\Cn({\mesh}_{l})}_{|\omega_{l}} := 
    \{\Bp\in \Cv(\mesh_{l}),\;\supp b_{\Vp}\subset\overline{\omega}_{l}\} =\; & 
    \Cn_{l}^{1}\cup\dots\cup \Cn_{l}^{P_{\Cn}} \;,\\
    {\Ce({\mesh}_{l})}_{|\omega_{l}} := 
    \{E\in \Ce(\mesh_{l}),\;\supp b_{E}\subset\overline{\omega}_{l}\} =\; & 
    \Ce_{l}^{1}\cup\dots\cup \Ce_{l}^{P_{\Ce}}\;,
  \end{align*}
  and for any $K\in{\mesh}_{l}$, $K\subset\omega_{l}$, two of its vertices/edges
  will belong to different subsets.
\end{lemma}
Here, $b_{\Bp}$ is the nodal basis function of $\LFE(\wh{\mesh}_l)$ attached to the
vertex $\Bp$, and $\Vb_E$ is the nodal basis function of $\EFE(\wh{\mesh}_l)$
associated with the edge $E$, see \eqref{eq:fem13}.

\begin{proof}
  A crude argument cites the fact that each vertex and each edge belongs to only a
  finite number of elements. A bound for this number can be deduced from the shape
  regularity measure. The rest is elementary combinatorial arguments. 
\end{proof}

Next, define subspaces of $\LFE(\wh{\mesh}_l)$ and $\EFE(\wh{\mesh}_l)$ by
\begin{align*}
  \LFE^i_l:= & \Span{b_{\Bp},\;\Bp\in\Cn_l^i}\subset \Span{\Bas_{\LFE}^{l}},\;
  i=1,\ldots,P_{\Cn}\;,\\
  \EFE^i_l:= & \Span{\Vb_E,\;E\in\Ce_l^i}\subset\Span{\Bas_{\EFE}^{l}},\;i=1,\ldots,P_{\Ce}\;.
\end{align*}%
Note that the basis functions spanning both $\LFE^{i}_{l}$ and $\EFE^{i}_{l}$ are
mutually orthogonal (w.r.t. $\aform$ and the $H^{1}(\Omega)$-inner product). Thus, it
suffices to establish the strengthened Cauchy-Schwarz inequality \eqref{eq:C-orth}
for the family of subspaces ${\{H_{j}\}}_{j} = {\{\EFE^i_l\}}_{l,i}\cup{\{\grad
  \LFE^{i}_{l}\}}_{l,i}$ of $\EFE({\mesh_{L}})$. This will yield the relevant constants
in \eqref{eq:C-rat}. In other words, we analyze the quasi-orthogonality property of
the multilevel decomposition
\begin{gather}
  \label{eq:LMGcurllump}
  \EFE(\mesh_{h}) = \EFE(\mesh_{0}) +
  \sum\limits_{l=1}^{L}
  \sum\limits_{i=1}^{P_{\Cn}} \grad \LFE^{i}_{l} 
  +\sum\limits_{l=1}^{L}\sum\limits_{i=1}^{P_{\Ce}} \EFE_{l}^{i} \;.
\end{gather}
\rh{Note that \eqref{eq:LMGcurllump} gives rise to a multigrid algorithm, for which
Thm.~\ref{thm:41} gives exactly the same convergence estimate as for the method
induced by \eqref{eq:LMGcurl}!}

\begin{lemma}
   \label{lem:orth-curl}
   For all $\Vv_{m}\in\EFE({\mesh}_{m})$ and $\Vu_{l}^{i}\in\EFE_l^i$, $0\leq m
   \leq l \leq L$, $i=1,\ldots,P_{\Ce}$, it holds that, with $C>0$ depending only on
   the bound for the shape regularity measures of the meshes ${\mesh}_{l}$,
   \begin{align}
     \label{eq:orth-curl}
     \SP{\curl\Vv_{m}}{\curl\Vu_{l}^{i}}_{\Ltwo} & \leq C 
     h_{l}^{\frac{1}{2}}h_{m}^{-\frac{1}{2}}
     \N{\curl\Vv_{m}}_{\Ltwo}\N{\curl\Vu_{l}^{i}}_{\Ltwo}\;,\\ 
     \label{eq:orth-l2}
     \SP{\Vv_{m}}{\Vu_{l}^{i}}_{\Ltwo} & \leq C 
     h_{l} \N{\Vv_{m}}_{\Ltwo}\N{\curl\Vu_{l}^{i}}_{\Ltwo}\;.
   \end{align}
 \end{lemma}

 \begin{proof} Pick any (open) tetrahedron $K\in\wh{\mesh}_{m}$. Use the basis
   representation of $\Vu_l^i$ to isolate ``interior'' and ``boundary'' parts
   \begin{equation*}
     {\Vu_l^i}_{|K} =\sum_{E\subset\bar{K}, 
       E\in\Ce_l^i}\Bigl(\int_E\Vu_l^i\cdot\mathrm{d}\vec{s}\Bigr)\cdot\Vb_E
     = \Vu_{l,bd}^i+\Vu_{l,int}^i,
   \end{equation*}
   where
   \begin{equation*}
     \Vu_{l,bd}^i :=\sum_{E\subset\partial K,E\in\Ce_l^i}
     \Bigl(\int_E\Vu_l^i\cdot\mathrm{d}\vec{s}\Bigr)\cdot\Vb_E \quad\hbox{and}\quad
     \Vu_{l,int}^i=\sum_{E\subset K,E\in\Ce_l^i}
     \Bigl(\int_E\Vu_l^i\cdot\mathrm{d}\vec{s}\Bigr)\cdot\Vb_E\;.
   \end{equation*}  
   Since $\curl\Vv_m$ is a constant vector in $K$ and $\Vu_{l,int}^i\times\Bn={\bf
     0}$ on $\partial K$, by Green's formula, it is easy to see
   \begin{equation*}
     \int_K\curl\Vu_{l}^i\cdot\curl\Vv_m\mathrm{d}\Bx
     =\int_K\curl\Vu_{l,bd}^i\cdot\curl\Vv_m\mathrm{d}\Bx
     =\int_{\Sigma}\curl\Vu_{l,bd}^i\cdot\curl\Vv_m\mathrm{d}\Bx,
   \end{equation*}
   where 
   \begin{gather*}
     \Sigma:=\bigcup\{\supp{\Vb_E}:\;E\subset\partial K,E\in\Ce_l^i\}
   \end{gather*}
   is contained in a narrow strip along the boundary of $K$ of width $\approx
   h_{l}$. Hence, we arrive at the area ratio
   \begin{gather*}
     |\Sigma| \leq C h_{l}h_{m}^{-1}|K|\;.
   \end{gather*}
   Here and throughout the remainder of the proof, $C>0$ depends on
   shape regularity only. Thus, using the Cauchy-Schwartz inequality
   and noting that the basis functions in $\EFE_l^i$ are mutually orthogonal, we have
   \begin{eqnarray}
     \label{eq:orth2}
     \int_K\curl\Vu_{l}^i\cdot\curl\Vv_m\mathrm{d}\Bx &\le&
     \N{\curl\Vu_{l,bd}^i}_{L^2(\Sigma)} |\Sigma|^{1/2}
     \SN{\curl\Vv_m}\\
     &\le& C\,\sqrt{\frac{h_l}{h_m}}\,\N{\curl\Vu_{l}^i}_{L^2(K)}
     |K|^{1/2}\SN{\curl\Vv_m}\notag\\
     &=& C\,\sqrt{\frac{h_l}{h_m}}\,\N{\curl\Vu_{l}^i}_{L^2(K)}
     \N{\curl\Vv_m}_{L^2(K)}\;.\notag
   \end{eqnarray}
   To estimate the $L^2$-inner product, we recall the following simple fact about the
   norms of edge basis functions on level $l$:
   \begin{eqnarray*}
     \N{\Vb}_{L^2(K)} \le Ch_l\N{\curl\Vb}_{L^2(K)}\quad
     \forall\,\Vb\in\Bas_{\EFE}^{l}\;.
   \end{eqnarray*}
   Since the basis functions of $\EFE_l^i$ do not interact, we have
   \begin{eqnarray}
     \label{eq:orth0}
     \int_{K}\Vu_{l}^i\cdot \Vv_m\mathrm{d}\Bx
     \le\N{\Vu_{l}^i}_{L^2(K)} \N{\Vv_m}_{L^2(K)} \le
     Ch_l\N{\curl\Vu_{l}^i}_{L^2(K)}
     \N{\Vv_m}_{L^2(K)}.
   \end{eqnarray}
   Now \eqref{eq:orth-curl} and  \eqref{eq:orth-l2} follows by summation
   over all elements of ${\mesh}_{m}$ and another Cauchy-Schwarz inequality.
 \end{proof}

After replacing $\EFE_l^i$ with $\LFE_l^i$ in the proof of Lemma
\ref{lem:orth-curl}, similar arguments establish the following estimate:

\begin{lemma}
  \label{lem:orth-grad}
  For all $\Vv_{m}\in\EFE({\mesh}_{m})$ and $u_{l}^{i}\in\LFE_l^i$, $0\leq m
  \leq l \leq L$, $i=1,\ldots,P_{\Cn}$, it holds that, with $C>0$ depending only on
  the bound for the shape regularity measures of the meshes ${\mesh}_{l}$,
  \begin{align}
    \label{eq:orth-grad0}
    \SP{\Vv_{m}}{\grad u_{l}^{i}}_{\Ltwo} & \leq C 
    h_{l}^{\frac{1}{2}}h_{m}^{-\frac{1}{2}}
    \N{\Vv_{m}}_{\Ltwo}\N{\grad u_{l}^{i}}_{\Ltwo}\;.
  \end{align}
\end{lemma}

\begin{proof} Again, pick $K\in {\mesh}_{m}$. By separating
  interior and boundary parts of $u_{l}^{i}$ as above and noting
  $\Div{\Vv_{m}}_{|K}=0$ on $K$, we find by Green's formula
  \begin{gather*}
    \int\nolimits_{K}\Vv_{m}\cdot\grad u_{l}^{i}\,\mathrm{d}\Bx = 
    \int\nolimits_{\Sigma}\Vv_{m}\cdot\grad u_{l}^{i} \,\mathrm{d}\Bx\;.
  \end{gather*}
  As above, we infer
  \begin{gather*}
    \int\nolimits_{K}\Vv_{m}\cdot\grad u_{l}^{i}\,\mathrm{d}\Bx \leq 
    C\,\sqrt{\frac{h_l}{h_m}}\,\N{\Vv_{m}}_{L^2(K)}
    \N{\grad u_l^{i}}_{L^2(K)}\;.
  \end{gather*}
  Summation over all $K$ and a Cauchy-Schwarz inequality finish the proof.
\end{proof}

Because of the geometric decay of the meshwidths $h_{l}$ of the (uniformly refined) meshes
$\wh{\mesh}_{l}$, these estimates clearly imply the desired quasi-orthogonality
for \eqref{eq:LMGcurllump}.

\begin{theorem} (Strengthened Cauchy-Schwartz inequality) 
  For any $\Vu_{l}^{i}\in\EFE_{l}^{i}$ or $\Vu_{l}^{i}\in\grad \LFE_{l}^{i}$ and any
  $\Vv_{l}^{j}\in\EFE_{l}^{j}$ or $\Vv_{l}^{j}\in\grad \LFE_{l}^{j}$, $0\le i,j\le
  P_{\Ce}$ or $0\le i,j\le P_{\Cn}$, resp., the estimate
\begin{equation}
  \label{eq:SCSI}
  \aform(\Vu_l^i,\Vv_m^j)\le C\theta^{|l-m|/2}\N{\Vu_l^i}_A\N{\Vv_m^j}_A \quad
  0\le i,j\le N_l,\; 0\le l,\,m\le L
\end{equation}
holds, where $C=C(\max_{l}{\rho}_{l})>0$ and $0<\theta<1$ is the decrease rate of
the meshwidths defined in \eqref{eq:ref25}.
\end{theorem}


\section{Numerical experiments}
\label{sec:numer-exper}

In the reported numerical experiments the implementation of adaptive mesh refinement
was based on the adaptive finite element package ALBERTA \cite{ALBT}, which uses the
bisection strategy of \cite{KOS94a}, see Sect.~\ref{sec:LMG}.

Let $\mesh_0$ be an initial mesh satisfying the two assumptions (A1)
and (A2) in \cite[P.~282]{KOS94a}, the adaptive mesh refinements are
governed by a residual based a posteriori error estimator. In the
experiments we assume the current density $\Vf\in\Hdiv$ and use the
estimator given by \cite[\S 5]{ZCW05a}: given a finite element
approximation $\Vu_{h}\in\EFE(\mesh_{h})$, for any $T\in\mesh_{h}$
\begin{equation}
\eta_T^2:=h_T^2\|\Vf-\Vu_h\|_{\BH(\Div,T)}^2
+\frac{h_T}{2}\sum_{F\subset\partial T}\Big\{\|[\Vu_h]_F\|_{0,F}^2
+\|[\curl\Vu_h\times\nubf]_F\|_{0,F}^2\Big\}, \nonumber
\end{equation}
where $F$ is a face of $T$, $\nubf$ is the unit normal of $F$, and
$[\Vu_h]_F$ is the jump of $\Vu_h$ across $F$. The global a
posteriori error estimate and the maximal estimated element error
on $\mesh_h$ are defined by
\begin{equation}
  \eta_h:=\left(\sum_{T\in\mesh_h}\eta_T^2 \right)^{1/2},\qquad
  \eta_{\script{max}}=\max_{T\in\mesh_h}\eta_T.
\end{equation}
Using $\eta_h$ and $\eta_{\script{max}}$, we use \cite[Algorithm~5.1]{ZCW05a} to mark
and refine $\mesh_h$ adaptively.

In the following, we report two numerical experiments to demonstrate the competitive
behavior of the local multigrid method and to validate our convergence theory.

\setcounter{theorem}{0}
\begin{example}\label{exp1}
We consider the Maxwell equation on the three-dimensional
``L-shaped'' domain
$\Omega=(-1,\;1)^3\setminus\{(0,\;1)\times(-1,\;0)\times
(-1,\;1)\}$. The Dirichlet boundary condition and the righthand
side $\Vf$ are chosen so that the exact solution is
$$\Vu:=\nabla\left\{r^{1/2}\sin(\phi/2)\right\}$$
in cylindrical coordinates $(r,\phi,z)$.
\end{example}

Table \ref{exp1:tab} shows the numbers of multigrid iterations required to reduce the
initial residual by a factor $10^{-8}$ on different levels. We observe that the
multigrid algorithm converges in almost the same small number of steps, though the
number of elements varies from 156 to 100,420.

\begin{table}[h!]
\centering \caption{The number of adaptive iterations
$N_{\script{it}}$, the number of elements $N_{\script{el}}$, the
number of multigrid iterations $N_{itrs}$ required to reduce the
initial residual by a factor $10^{-8}$, the relative error between
the true solution $\Vu$ and the discrete solution $\Vu_h$:\quad
$E_{\script{rel}}=\N{\Vu-\Vu_h}_{\VH(\curl,\Omega)}/
\N{\Vu}_{\VH(\curl,\Omega)}$ (Example
\ref{exp1}).}\label{exp1:tab}
\begin{tabular}{*{9}{c}}
\hline  $N_{\script{it}}$
& 2 & 5 & 10 &  15 & 20 & 25 & 30 & 35  \\
\hline  $N_{\script{el}}$ &  156 &  388 &  1,900 & 4,356 &
 9608  &  19,424 &  48,088  &  100,420 \\
\hline  $E_{\script{rel}}$  & 0.4510 & 0.3437&   0.2456  & 0.1919
& 0.1600 &  0.1350 & 0.1094 &   0.0915   \\
\hline $N_{itrs}$ &  11 & 21 & 19 &  19 & 19 & 19 & 19 & 19 \\
\hline\\
\end{tabular}
\end{table}

Fig. \ref{exp1:cpu} (left) plots the CPU time versus the number of degrees of freedom
on different adaptive meshes. It shows that the CPU time of solving the algebraic
system increases roughly linearly with respect to the number of elements.  Fig.
\ref{exp1:cpu} (rught) depicts a locally refined mesh of 100,420 elements created by
the adaptive finite element algorithm.

\begin{figure}[h!]
  \subfigure{%
    \includegraphics*[width=0.45\textwidth]{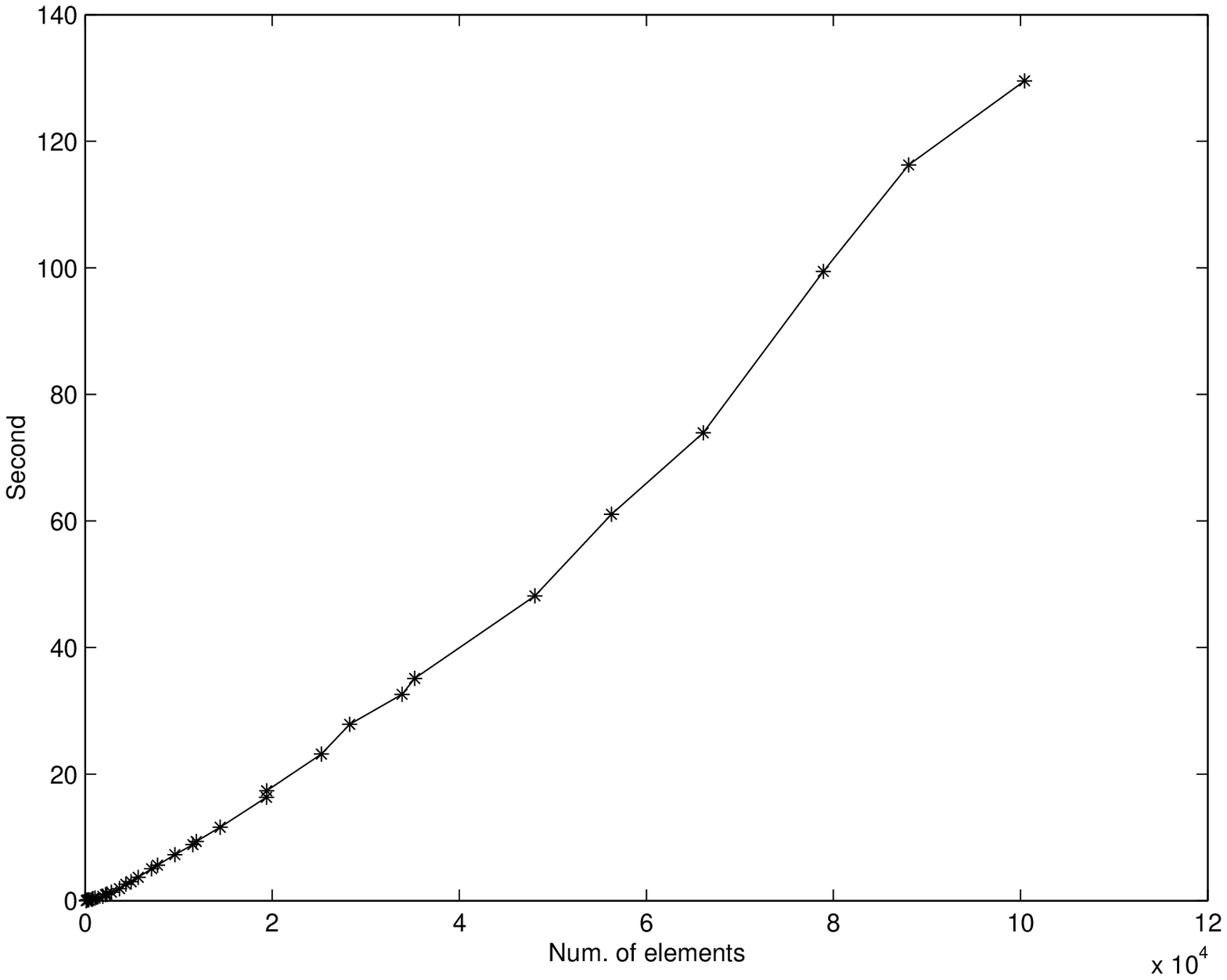}}
  \subfigure{%
    \includegraphics*[width=0.5\textwidth]{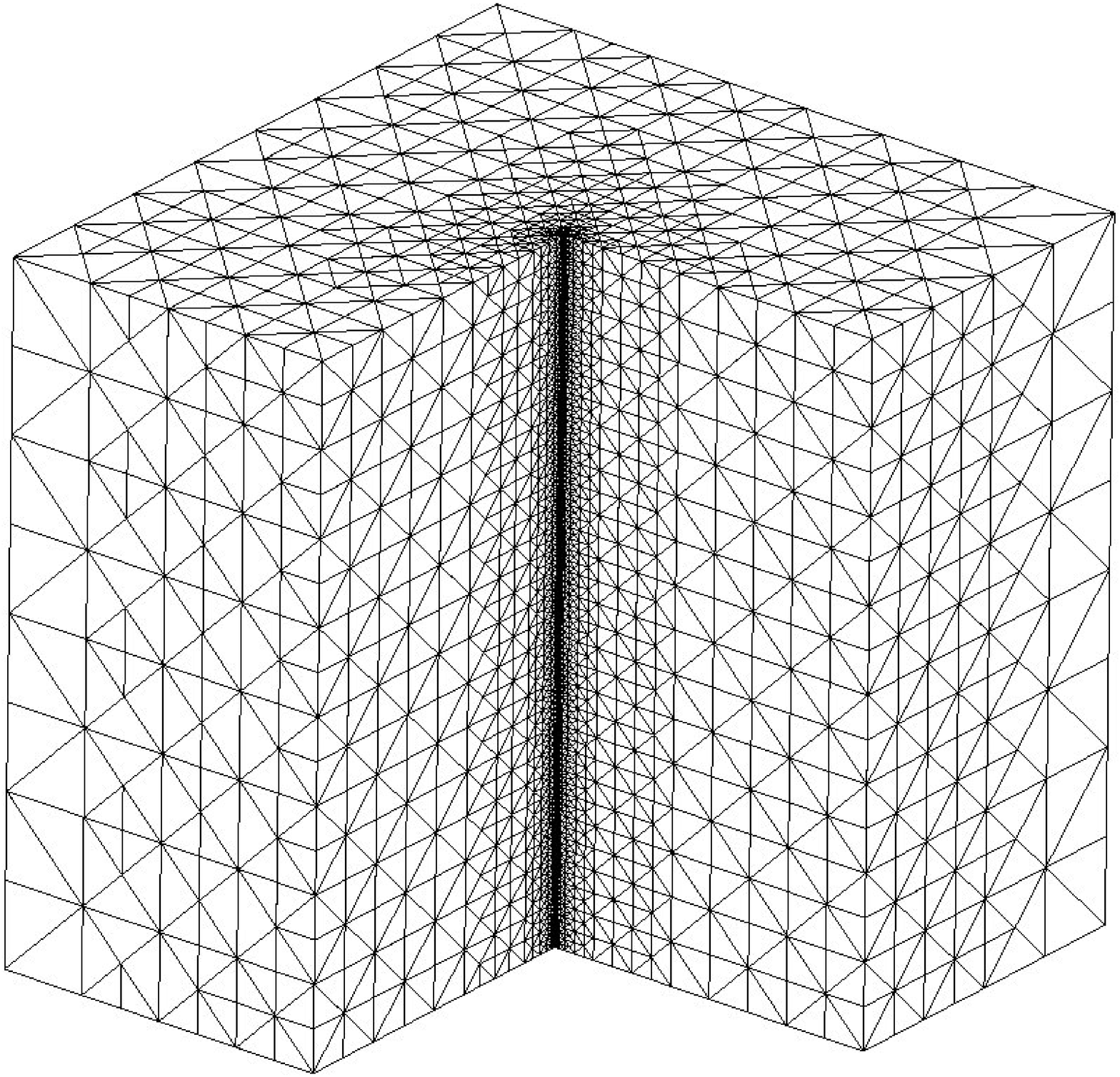}}
  \caption {Example \ref{exp1}, left:
    execution time for local multigrid method,
    right: instance of a locally refined mesh (100,420 elements)
    \label{exp1:cpu}}
\end{figure}

\begin{example}
  \label{exp2}
This example uses the same solution as Example \ref{exp1}
$$\Vu:=\nabla\left\{r^{1/2}\sin(\phi/2)\right\}$$
in cylindrical coordinates $(r,\phi,z)$. But the computational domain is changed to a
three-dimensional non-Lipschitz domain with an inner crack-type boundary, which is
defined by
$$\Omega=(-1,\;1)^3\setminus\{(x,0,z):\;0\le x< 1,\;-1<z<1\}.$$
The Dirichlet boundary condition and the source function $\Vf$ are
the same as above.
\end{example}

Table \ref{exp2:tab} records the numbers of multigrid iterations required to reduce the
initial residual by a factor $10^{-8}$ on different levels. We observe that the
multigrid algorithm converges in less than 30 steps, with the number of elements
soaring from 128 to 135,876.

\begin{table}[h!]
  \centering \caption{The number of adaptive iterations
    $N_{\script{it}}$, the number of elements $N_{\script{el}}$, the
    number of multigrid iterations $N_{itrs}$ required to reduce the
    initial residual by a factor $10^{-8}$, the relative error between
    the true solution $\Vu$ and the discrete solution $\Vu_h$:\quad
    $E_{\script{rel}}=\N{\Vu-\Vu_h}_{\VH(\curl,\Omega)}/
    \N{\Vu}_{\VH(\curl,\Omega)}$ (Example
    \ref{exp2}).}\label{exp2:tab}
\begin{tabular}{*{9}{c}}
\hline  $N_{\script{it}}$
& 2 & 5 & 10 &  15 & 20 & 25 & 30 & 33  \\
\hline  $N_{\script{el}}$ &  128 &  404 &  1,236 & 3,416 &
12,420 &  29,428 &  81,508  &  135,876 \\
\hline  $E_{\script{rel}}$  & 0.4616 & 0.3762&   0.2992  & 0.2347
& 0.1752 &  0.1394 & 0.1095 &   0.0958   \\
\hline $N_{itrs}$ &  14 & 30 & 25 &  26 & 26 & 27 & 27 & 27 \\
\hline\\
\end{tabular}
\end{table}

Fig. \ref{exp2:mesh} (left) shows the CPU time versus the number of degrees of freedom on
different adaptive meshes. Obviously, the CPU time for solving the algebraic system
increases nearly linearly with respect to the number of elements.

Fig. \ref{exp2:mesh} (right) displays a locally refined mesh of 135,876 elements using
adaptive finite element algorithm. In addition, the restriction of the mesh to the
cross-section $\{y=0\}$, which contains the inner boundary, is drawn. This reveals
strong local refinement.

This experiment bears out that the local multigrid is also efficient for the problems
in non-Lipschitz doamins, which are outside the scope of our theory.

\begin{figure}[h!]
  \subfigure{%
    \includegraphics*[width=0.45\textwidth]{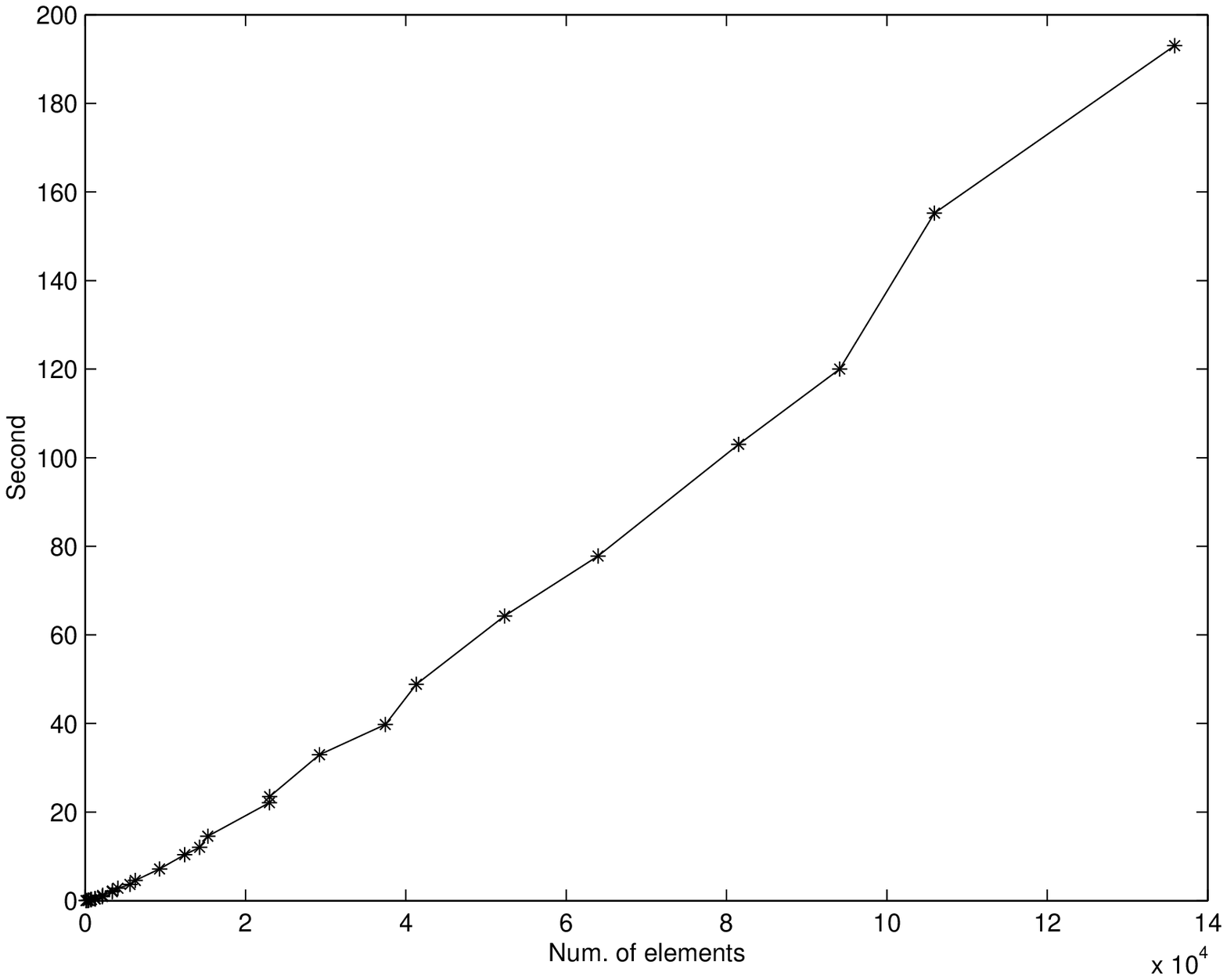}}
  \subfigure{%
    \includegraphics*[width=0.5\textwidth]{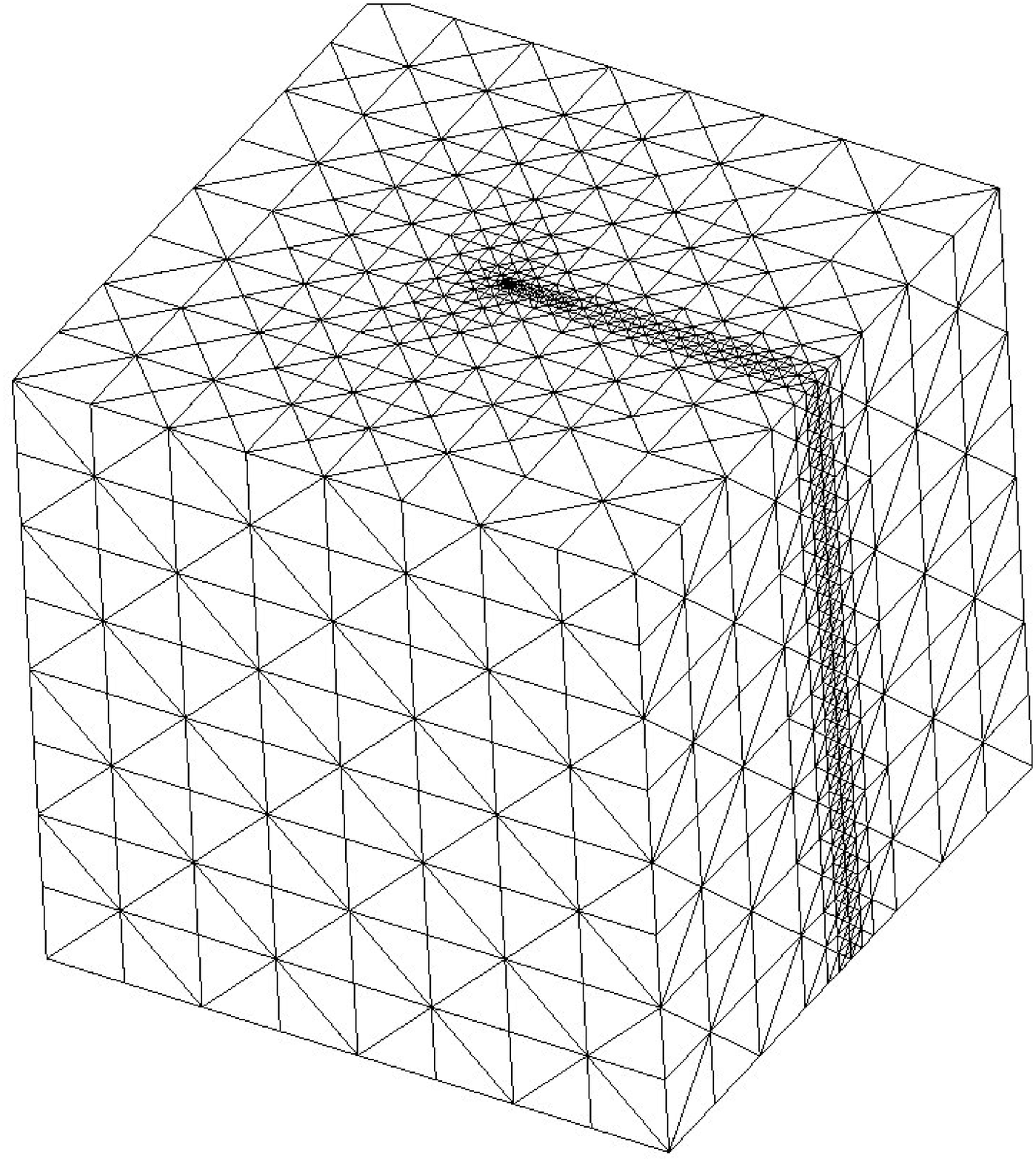}}
  \caption{Example \ref{exp2},
    left: CPU time for solving the algebraic system by multigrid method,
    right: a locally refined mesh (135,876 elements)
    \label{exp2:mesh}}
\end{figure}



\section*{Acknowledgement}

The authors would like to thank Dr. L. Wang of Computer Network
Information Center, Prof. Z. Chen and Prof. L. Zhang of the Institute
of Computational Mathematics, Chinese Academy of Sciences, for
their support in the implementation of the local multigrid method. They
are grateful to one referees who detected an error in an earlier version
of the manuscript. 

\bibliographystyle{siam}
\bibliography{lit}

\begin{thebibliography}{10}

\bibitem{AIM01}
{\sc M.~Ainsworth and W.~McLean}, {\em Multilevel diagonal scaling
  preconditioners for boundary element equations on locally refined meshes},
  Numer. Math., 93 (2003), pp.~387--413.

\bibitem{AFW99}
{\sc D.~Arnold, R.~Falk, and R.~Winther}, {\em Multigrid in {$H({\rm div})$}
  and {$H({\rm\bf curl})$}}, Numer. Math., 85 (2000), pp.~175--195.

\bibitem{AFW06}
\leavevmode\vrule height 2pt depth -1.6pt width 23pt, {\em Finite element
  exterior calculus, homological techniques, and applications}, Acta Numerica,
  15 (2006), pp.~1--155.

\bibitem{AMP98a}
{\sc D.~Arnold, A.~Mukherjee, and L.~Pouly}, {\em Locally adapted tetrahedral
  meshes using bisection}, SIAM Journal on Scientific Computing, 22 (2000),
  pp.~431--448.

\bibitem{BAB87}
{\sc D.~Bai and A.~Brandt}, {\em Local mesh refinement multilevel techniques},
  SIAM J. Sci. Stat. Comput., 8 (1987), pp.~109--134.

\bibitem{BAE91}
{\sc E.~B\"ansch}, {\em Local mesh refinement in 2 and 3 dimensions}, IMPACT
  Comput. Sci. Engrg., 3 (1991), pp.~181--191.

\bibitem{BDH97}
{\sc R.~Beck, P.~Deuflhard, R.~Hiptmair, R.~Hoppe, and B.~Wohlmuth}, {\em
  Adaptive multilevel methods for edge element discretizations of {M}axwell's
  equations}, Surveys on Mathematics for Industry, 8 (1998), pp.~271--312.

\bibitem{BHH98}
{\sc R.~Beck, R.~Hiptmair, R.~Hoppe, and B.~Wohlmuth}, {\em Residual based
  a-posteriori error estimators for eddy current computation},
  $\mathrm{M}^2\mathrm{AN}$, 34 (2000), pp.~159--182.

\bibitem{BEY94}
{\sc J.~Bey}, {\em Tetrahedral grid refinement}, Computing, 55 (1995),
  pp.~355--378.

\bibitem{BDD04}
{\sc P.~Binev, W.~Dahmen, and R.~DeVore}, {\em Adaptive finite element methods
  with convergence rates}, Numerische Mathematik, 97 (2004), pp.~219--268.

\bibitem{BOY93}
{\sc F.~Bornemann and H.~Yserentant}, {\em A basic norm equivalence for the
  theory of multilevel methods}, Numer. Math., 64 (1993), pp.~455--476.

\bibitem{BOS88a}
{\sc A.~Bossavit}, {\em Whitney forms: {A} class of finite elements for
  three-dimensional computations in electromagnetism}, IEE Proc. A, 135 (1988),
  pp.~493--500.

\bibitem{BOS98b}
\leavevmode\vrule height 2pt depth -1.6pt width 23pt, {\em Computational
  Electromagnetism. {V}ariational Formulation, Complementarity, Edge Elements},
  vol.~2 of Electromagnetism Series, Academic Press, San Diego, CA, 1998.

\bibitem{BRP93}
{\sc J.~Bramble and J.~Pasciak}, {\em New estimates for multilevel methods
  including the {V}--cycle}, Math. Comp., 60 (1993), pp.~447--471.

\bibitem{BPW91}
{\sc J.~Bramble, J.~Pasciak, J.~Wang, and J.~Xu}, {\em Convergence estimates
  for product iterative methods with applications to domain decomposition},
  Math. Comp., 57 (1991), pp.~1--21.

\bibitem{BCS00}
{\sc A.~Buffa, M.~Costabel, and D.~Sheen}, {\em On traces for
  {$\mathbf{H}(\mathbf{curl},\Omega)$} in {L}ipschitz domains}, J. Math. Anal.
  Appl., 276 (2002), pp.~845--867.

\bibitem{ZCW05a}
{\sc Z.-M. Chen, L.~Wang, and W.-Y. Zheng}, {\em An adaptive multilevel method
  for time-harmonic {M}axwell equations with singularities}, SIAM J. Sci.
  Comp.,  (2006).
\newblock To appear.

\bibitem{CHW07t}
{\sc S.~Christiansen and R.~Winther}, {\em Smoothed projections in finite
  element exterior calculus}, E-print 25-06, Department of Mathematics,
  University of Oslo, Oslo, Norway, 2006.
\newblock http://www.math.uio.no/eprint/pure\_math/2006/25-06.html.

\bibitem{CIA78}
{\sc P.~Ciarlet}, {\em The Finite Element Method for Elliptic Problems}, vol.~4
  of Studies in Mathematics and its Applications, North-Holland, Amsterdam,
  1978.

\bibitem{CFW04}
{\sc M.~Clemens, S.~Feigh, and T.~Weiland}, {\em Geometric multigrid algorithms
  using the conformal finite integration technique}, IEEE Trans. Magnetics, 40
  (2004), pp.~1065--1068.

\bibitem{CLE75}
{\sc P.~Cl\'ement}, {\em Approximation by finite element functions using local
  regularization}, RAIRO Anal. Num\'er., 2 (1975), pp.~77--84.

\bibitem{COD00}
{\sc M.~Costabel and M.~Dauge}, {\em Singularities of electromagnetic fields in
  polyhedral domains}, Arch. Rational Mech. Anal., 151 (2000), pp.~221--276.

\bibitem{CDN03}
{\sc M.~Costabel, M.~Dauge, and S.~Nicaise}, {\em Singularities of eddy current
  problems}, ESAIM: Mathematical Modelling and Numerical Analysis, 37 (2003),
  pp.~807--831.

\bibitem{GIR86}
{\sc V.~Girault and P.~Raviart}, {\em Finite element methods for Navier--Stokes
  equations}, Springer, Berlin, 1986.

\bibitem{GPD03}
{\sc J.~Gopalakrishnan, J.~Pasciak, and L.~Demkowicz}, {\em Analysis of a
  multigrid algorithm for time harmonic {M}axwell equations}, SIAM J. Numer.
  Anal., 42 (2003), pp.~90--108.

\bibitem{GRH99}
{\sc V.~Gradinaru and R.~Hiptmair}, {\em Whitney elements on pyramids},
  Electron. Trans. Numer. Anal., 8 (1999), pp.~154--168.

\bibitem{GRI85}
{\sc P.~Grisvard}, {\em Elliptic Problems in Nonsmooth Domains}, Pitman,
  Boston, 1985.

\bibitem{HIP97}
{\sc R.~Hiptmair}, {\em Multigrid method for {M}axwell's equations}, Tech. Rep.
  374, Institut f\"ur Mathematik, Universit\"at Augsburg, 1997.
\newblock USE HIP 99.

\bibitem{HIP99}
\leavevmode\vrule height 2pt depth -1.6pt width 23pt, {\em Multigrid method for
  {M}axwell's equations}, SIAM J. Numer. Anal., 36 (1999), pp.~204--225.

\bibitem{HIP02}
\leavevmode\vrule height 2pt depth -1.6pt width 23pt, {\em Finite elements in
  computational electromagnetism}, Acta Numerica, 11 (2002), pp.~237--339.

\bibitem{HIP00b}
\leavevmode\vrule height 2pt depth -1.6pt width 23pt, {\em Analysis of
  multilevel methods for eddy current problems}, Math. Comp., 72 (2003),
  pp.~1281--1303.

\bibitem{HWZ05}
{\sc R.~Hiptmair, G.~Widmer, and J.~Zou}, {\em Auxiliary space preconditioning
  in {${\mathbf H}_{0}(\mathbf{curl},\Omega)$}}, Numer. Math., 103 (2006),
  pp.~435--459.

\bibitem{HIX06}
{\sc R.~Hiptmair and J.~Xu}, {\em Nodal auxiliary space preconditioning in
  {H(curl)} and {H(div)} spaces}, SIAM J. Numer. Anal., 45 (2007),
  pp.~2483--2509.

\bibitem{HOS07}
{\sc R.~Hoppe and J.~Sch\"oberl}, {\em Convergence of adaptive edge element
  methods for the {3D} eddy currents equation}, J. Comp. Math.,  (2008).

\bibitem{KOV08}
{\sc T.~Kolev and P.~Vassilevski}, {\em Parallel auxiliary space {AMG} for
  {$\bf{H}(\mathbf{curl})$} problems}, J. Comp. Math.,  (2008).

\bibitem{KOS94a}
{\sc I.~Kossaczk\'y}, {\em A recursive approach to local mesh refinement in two
  and three dimensions}, J. Comput. Appl. Math., 55 (1994), pp.~275--288.

\bibitem{LIM72}
{\sc J.~Lions and F.~Magenes}, {\em Nonhomogeneous boundary value problems and
  applications}, Springer--Verlag, Berlin, 1972.

\bibitem{MAU95}
{\sc J.~Maubach}, {\em Local bisection refinement for $n$--simplicial grids
  generated by reflection}, SIAM J. Sci. Stat. Comp., 16 (1995), pp.~210--227.

\bibitem{MCL00}
{\sc W.~McLean}, {\em Strongly Elliptic Systems and Boundary Integral
  Equations}, Cambridge University Press, Cambridge, UK, 2000.

\bibitem{MIT89}
{\sc W.~Mitchell}, {\em A comparison of adaptive refinement techniques for
  elliptic problems}, ACM Trans. Mathematical Software, 15 (1989),
  pp.~326--347.

\bibitem{MIT92}
\leavevmode\vrule height 2pt depth -1.6pt width 23pt, {\em Optimal multilevel
  iterative methods for adaptive grids}, SIAM J. Sci. Stat. Comput, 13 (1992),
  pp.~146--167.

\bibitem{MON03}
{\sc P.~Monk}, {\em Finite Element Methods for {M}axwell's Equations},
  Clarendon Press, Oxford, UK, 2003.

\bibitem{NED80}
{\sc J.~N{\'e}d{\'e}lec}, {\em Mixed finite elements in {$\mathbb{R}^3$}},
  Numer. Math., 35 (1980), pp.~315--341.

\bibitem{OSW90}
{\sc P.~Oswald}, {\em On function spaces related to the finite element
  approximation theory}, Z. Anal. Anwendungen, 9 (1990), pp.~43--64.

\bibitem{OSW92}
\leavevmode\vrule height 2pt depth -1.6pt width 23pt, {\em On discrete norm
  estimates related to multilevel preconditioners in the finite element
  method}, in Constructive Theory of Functions, Proc. Int. Conf. Varna 1991,
  K.~Ivanov, P.~Petrushev, and B.~Sendov, eds., Bulg. Acad. Sci., 1992,
  pp.~203--214.

\bibitem{OSW94}
\leavevmode\vrule height 2pt depth -1.6pt width 23pt, {\em Multilevel finite
  element approximation}, Teubner Skripten zur Numerik, B.G. Teubner,
  Stuttgart, 1994.

\bibitem{PAZ02}
{\sc J.~Pasciak and J.~Zhao}, {\em Overlapping {S}chwarz methods in {H(curl)}
  on polyhedral domains}, J. Numer. Math., 10 (2002), pp.~221--234.

\bibitem{ALBT}
{\sc A.~Schmidt and K.~Siebert}, {\em {ALBERTA} -- {A}n adaptive hierarchical
  finite element toolbox}.
\newblock Website.
\newblock ALBERTA is available online from http://www.alberta-fem.de.

\bibitem{ALBTA}
{\sc A.~Schmidt and K.~Siebert}, {\em {D}esign of {A}daptive {F}inite {E}lement
  {S}oftware: {T}he {F}inite {E}lement {T}oolbox {ALBERTA}}, Lecture Notes in
  Computational Science and Engineering, Springer, Heidelberg, 2005.

\bibitem{SHO01}
{\sc J.~Sch\"oberl}, {\em Commuting quasi-interpolation operators for mixed
  finite elements}, Preprint ISC-01-10-MATH, Texas A\&M University, College
  Station, TX, 2001.

\bibitem{SHO05p}
\leavevmode\vrule height 2pt depth -1.6pt width 23pt, {\em A multilevel
  decomposition result in {$H(curl)$}}, in Proceedings of the 8th European
  Multigrid Conference 2005, Scheveningen, P.~H. P.~Wesseling, C.W.~Oosterlee,
  ed., 2006.

\bibitem{SHO05}
\leavevmode\vrule height 2pt depth -1.6pt width 23pt, {\em A posteriori error
  estimates for {M}axwell equations}, Math. Comp., 77 (2008), pp.~633--649.

\bibitem{SCZ90}
{\sc L.~R. Scott and Z.~Zhang}, {\em Finite element interpolation of nonsmooth
  functions satisfying boundary conditions}, Math. Comp., 54 (1990),
  pp.~483--493.

\bibitem{SHW06}
{\sc O.~Sterz, A.~Hauser, and G.~Wittum}, {\em Adaptive local multigrid methods
  for solving time-harmonic eddy current problems}, IEEE Trans. Magnetics, 42
  (2006), pp.~309--318.

\bibitem{STE06}
{\sc R.~Stevenson}, {\em Optimality of a standard adaptive finite element
  method}, Foundations of Computational Mathematics, 7 (2007), pp.~245--269.

\bibitem{STE08}
\leavevmode\vrule height 2pt depth -1.6pt width 23pt, {\em The completion of
  locally refined simplicial partitions created by bisection}, Math. Comp., 77
  (2008), pp.~227--241.

\bibitem{TRA97}
{\sc C.~Traxler}, {\em An algorithm for adaptive mesh refinement in $n$
  dimensions}, Computing, 59 (1997), pp.~115--137.

\bibitem{WEB05}
{\sc B.~Weiss and O.~Biro}, {\em Multigrid for time-harmonic 3-d eddy-current
  analysis with edge elements}, IEEE Trans. Magnetics, 41 (2005),
  pp.~1712--1715.

\bibitem{WIT57}
{\sc H.~Whitney}, {\em Geometric Integration Theory}, Princeton University
  Press, Princeton, 1957.

\bibitem{WUC05}
{\sc H.-J. Wu and Z.-M. Chen}, {\em Uniform convergence of multigrid
  {$V$}-cycle on adaptively refined finite element meshes for second order
  elliptic problems}, Science in China: Series A Mathematics, 49 (2006),
  p.~1C28.

\bibitem{JXU92}
{\sc J.~Xu}, {\em Iterative methods by space decomposition and subspace
  correction}, SIAM Review, 34 (1992), pp.~581--613.

\bibitem{JXU97}
\leavevmode\vrule height 2pt depth -1.6pt width 23pt, {\em An introduction to
  multilevel methods}, in Wavelets, Multilevel Methods and Elliptic {PDE}s,
  M.~Ainsworth, K.~Levesley, M.~Marletta, and W.~Light, eds., Numerical
  Mathematics and Scientific Computation, Clarendon Press, Oxford, 1997,
  pp.~213--301.

\bibitem{XUZ07}
{\sc J.~Xu and Y.-R. Zhu}, {\em Uniformly convergent multigrid methods for
  elliptic problems with strongly discontinuous coefficients}, Math. Models
  Methods Appl. Sci., 18 (2008), pp.~77--105.

\bibitem{XUZ00}
{\sc J.~Xu and L.~Zikatanov}, {\em The method of alternating projections and
  the method of subspace corrections in {H}ilbert space}, J. Am. Math. Soc., 15
  (2002), pp.~573--597.

\bibitem{YSE86}
{\sc H.~Yserentant}, {\em On the multi--level splitting of finite element
  spaces}, Numer. Math., 58 (1986), pp.~379--412.

\bibitem{YSE93}
\leavevmode\vrule height 2pt depth -1.6pt width 23pt, {\em Old and new
  convergence proofs for multigrid methods}, Acta Numerica,  (1993),
  pp.~285--326.

\bibitem{ZHA92}
{\sc X.~Zhang}, {\em Multilevel {S}chwarz methods}, Numer. Math., 63 (1992),
  pp.~521--539.

\end{thebibliography}



\end{document}